\documentclass[final]{siamltex}

\usepackage{xcolor}
\usepackage{amsmath}
\usepackage{amssymb}
\usepackage{mathrsfs}
\usepackage{graphicx}
\usepackage{tikz}
\usepackage{bm}
\usepackage{ucs}
\usepackage[utf8x]{inputenc}
\usepackage{wrapfig}
\usepackage{color}
\usepackage{float}
\usepackage{showlabels}
\usepackage{fullpage}
\usepackage{epstopdf}
\usepackage[breaklinks,colorlinks=true,linkcolor=blue,citecolor=red,
  backref=page]{hyperref}

\newtheorem{remark}{Remark}
\newtheorem{assum}{Assumption}

\newcommand{\R}{\mathbb{R}}
\newcommand{\C}{\mathbb{C}}

\newcommand{\vecB}{{\boldsymbol{B}}}

\newcommand{\vecI}{{\boldsymbol{I}}}
\newcommand{\vecE}{{\boldsymbol{E}}}
\newcommand{\vecf}{{\boldsymbol{f}}}
\newcommand{\vecg}{{\boldsymbol{g}}}
\newcommand{\vech}{{\boldsymbol{h}}}

\newcommand{\vecH}{{\boldsymbol{H}}}
\newcommand{\vecu}{{\boldsymbol{u}}}
\newcommand{\vecw}{  {\boldsymbol{w}} }
\newcommand{\vecv}{  {\boldsymbol{v}}  }

\newcommand{\vecp}{  {\boldsymbol{p}}  }
\newcommand{\vecq}{  {\boldsymbol{q}}  }
\newcommand{\vecG}{{\boldsymbol{G}}}
\newcommand{\vecK}{{\boldsymbol{K}}}
\newcommand{\vecL}{{\boldsymbol{L}}}

\newcommand{\vecN}{{\boldsymbol{N}}}

\newcommand{\vecU}{{\boldsymbol{U}}}

\newcommand{\vecS}{{\boldsymbol{S}}}

\newcommand{\bz}{{\boldsymbol{z}}}
\newcommand{\by}{{\boldsymbol{y}}}
\newcommand{\bx}{{\boldsymbol{x}}}
\newcommand{\be}{{\boldsymbol{e}}}

\newcommand{\mdiv}{\mbox{div\,}}
\newcommand{\mcurl}{\mbox{curl\,}}

\begin{document}

\title{The Interior Inverse Electromagnetic Scattering for an Inhomogeneous Cavity}
\author{Fang Zeng\footnotemark[1]~and~Shixu Meng\footnotemark[2]}
\renewcommand{\thefootnote}{\fnsymbol{footnote}}
  \footnotetext[1]{College of Mathematics and Statistics, Chongqing University, Chongqing, 401331, China.  {\tt fzeng@cqu.edu.cn}}
\footnotetext[2]{Institute of Applied Mathematics, Academy of Mathematics and Systems Science, Chinese Academy of Sciences, Beijing, 100190, China.  {\tt shixumeng@amss.ac.cn}}
\maketitle

\begin{abstract}
In this paper we consider the inverse electromagnetic scattering for a cavity surrounded by an inhomogeneous medium in three dimensions. The measurements are scattered wave fields measured on some surface inside the cavity, where such scattered wave fields are due to sources emitted on the same surface. We first prove that the measurements uniquely determine the shape of the cavity, where we make use of a boundary value problem called the exterior transmission problem. We then complete the inverse scattering problem by designing the linear sampling method to reconstruct the cavity. Numerical examples are further provided to illustrate the viability of our algorithm.
\end{abstract}

\section{Introduction}  \label{Introduction}
Inverse scattering  plays an important role in non-destructive testing, medical imaging, geophysical exploration and numerous problems associated with target identification. { Qualitative methods, such as linear sampling method and factorization method, have played an important role in inverse scattering. For a survey of these methods we refer to \cite{CaCo, CK,kirsch2008factorization}.} There have been recent interests in the inverse scattering problem for cavities using measurements inside \cite{Ik1,
liu2019regularized,L,QCa,QCo2,qu2019shape,sun2016reciprocity}. Such inverse scattering problems have potential applications in  monitoring the structural integrity of the fusion reactor. The measurements can also be used to design eigenvalue problems to detect anomalies in the exterior of a cavity \cite{cogar2018using}. To the author's knowledge, \cite{ZCS} is the first result on the inverse scattering for cavities using electromagnetic waves. In particular, they considered a perfectly conducting cavity and derived the linear sampling method to reconstruct the cavity. In the case that the cavity is penetrable, the inverse scattering problem is more challenging. The acoustic case was first considered in \cite{CaCoMe,MHC}. The aim of this paper is to extend the results obtained in \cite{CaCoMe} to the electromagnetic case. In particular, we consider the inverse electromagnetic scattering for a cavity surrounded by an inhomogeneous medium in three dimensions. The measurements are scattered electric fields measured on some surface inside the cavity, where such scattered electric fields are due to electric dipoles emitted on the same surface. The corresponding inverse problem is to determine the support of the cavity. { We remark that the cavity problems are "transposition" of the typical scattering problems \cite{haddar2002linear,CaCo1,Ch}.}

The first goal of this paper is to prove the uniqueness of the inverse scattering problem. To this end,
we follow the approach made in \cite{CaCo1, CaCoMe, Hah,Ch}. Note that this approach allows us to consider a general inhomogeneous medium since we avoid constructing the Green's function for the background media. To analyze this method, we propose and make use of a boundary value problem, the so-called {\it exterior transmission problem} for Maxwell's equations. The exterior transmission problem for the acoustic wave equation was proposed in \cite{CaCoMe}.
The second goal is to design a robust imaging algorithm. Here we are particularly interested in the so-called qualitative methods. For a more detailed introduction to qualitative methods we refer to the monographs \cite{CaCo, CK,kirsch2008factorization}. As a first step towards the qualitative methods for our inverse interior electromagnetic scattering problem, we investigate the linear sampling method. It is our further interest to investigate the factorization method, the generalized linear sampling method, and other related robust imaging methods.

To begin with, we introduce the inverse scattering problem in a heuristic setting. We consider the time harmonic electromagnetic scattering in a penetrable cavity $D$ surrounded by inhomogeneous medium in $\R^3$. To be more precise, the medium in $D$ is homogeneous with permittivity $\epsilon_0 \vecI$ and permeability $\mu_0\vecI$. The inhomogeneous medium in the exterior $\mathbb{R}^3 \backslash \overline{D}$ is characterized by matrix-valued permittivity $\boldsymbol{\epsilon}(\bx)$ and permeability $\boldsymbol{\mu}(\bx)$. Furthermore, we assume that $\boldsymbol{\epsilon}(\bx) \equiv \epsilon_0 \vecI$ and $\boldsymbol{\mu}(\bx) \equiv \mu_0\vecI$ outside a large enough ball $B_R$, i.e., the compact support of $\boldsymbol{\epsilon}(\bx) - \epsilon_0 \vecI$ and $\boldsymbol{\mu}(\bx) - \mu_0\vecI$ are contained in $B_R \backslash \overline{D}$. For convenience, we denote the inhomogeneous medium by $D_1 \backslash \overline{D}$. See Fig.~\ref{cavity} for an example of the geometry. {Throughout the paper, we make the following assumptions on $D$, $D_1$, $A$, and $N$. Assume that $D$ and $D_1$ are open, connected regions with $C^2$-smooth boundaries $\partial D$ and $\partial D_1$. Assume that the matrix real-valued functions $A$ and $N$ are symmetric and have $C^1(\overline{D}_1 \backslash D)$ entries. We remark that these assumptions are needed because strong regularities are needed in the technical proof of the uniqueness result (for instance, such regularity assumptions were assumed in \cite{CaCo1} for the uniqueness result in the typical exterior case).}

        \begin{figure}[hb!]
        \centering
\includegraphics[width=0.6422\linewidth]{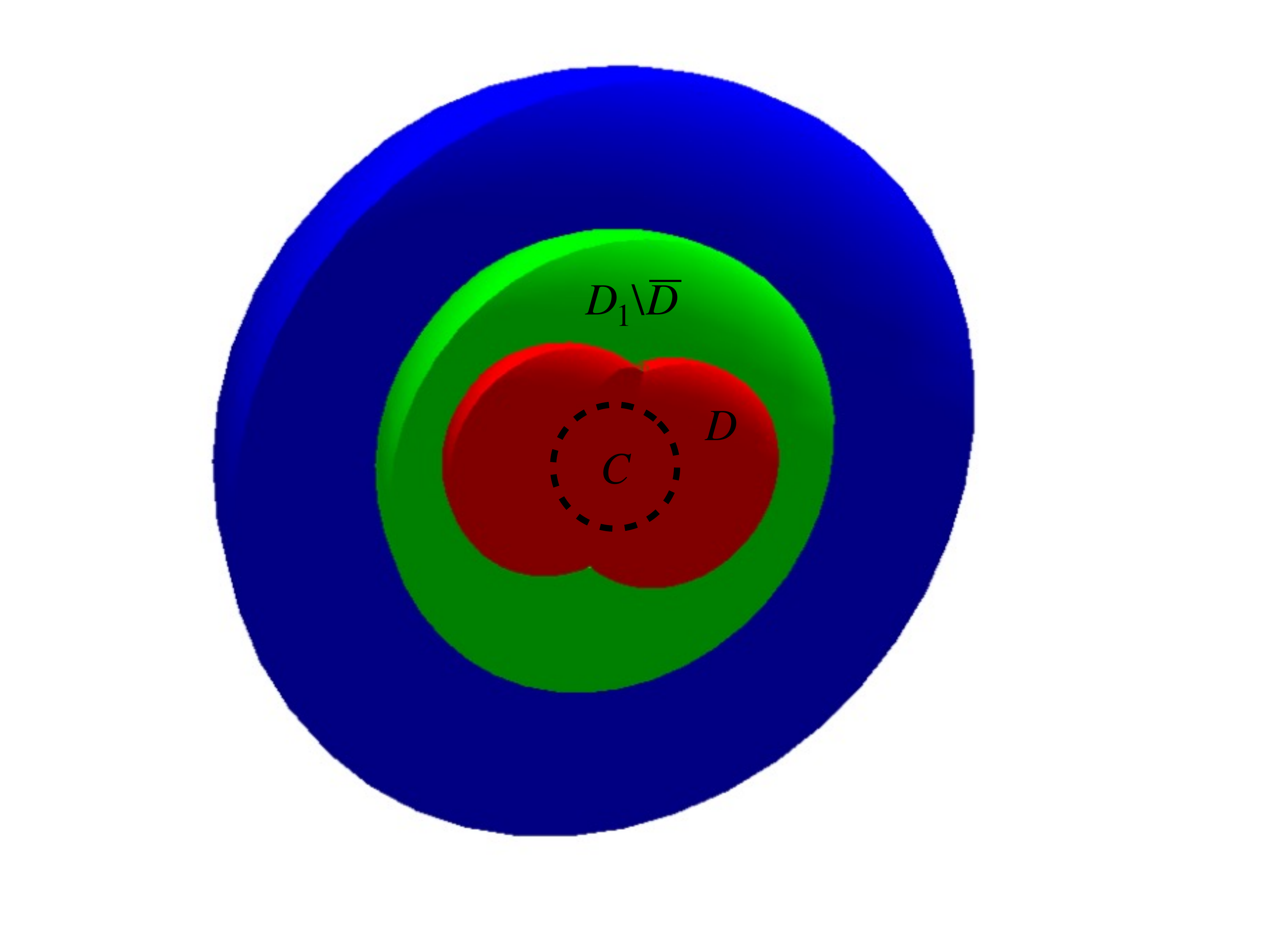}
     \caption{
     \linespread{1}
     An example of the geometry. $D$: cavity. $D_1 \backslash \overline{D}$: inhomogeneous medium. $C$: measurement ball.
     } \label{cavity}
    \end{figure}

For a given fixed frequency $\omega$, let $k:=\sqrt{\epsilon_0\mu_0}~\omega$ be the wave number. Assume the electric incident field is given by an electric dipole at $\by \in D$ with polarization $\vecp \in \R^3$. More specifically, the incident field is represented by
\begin{eqnarray} \label{Intro_Ei}
\vecE^i(\cdot, \by, \vecp):=\vecG (\cdot, \by, \vecp) = \frac{i}{k} \mcurl \mcurl \Phi(\cdot,\by)\vecp,
\end{eqnarray}  
where $\vecG(\cdot, \by, \vecp)$ is the Green's tensor of the free space and $\Phi(\bx,\by)$ is the fundamental solution of the Helmholtz equation in $\mathbb{R}^3$ given by
$$
\Phi(\bx,\by)= \frac{e^{ik|\bx-\by|}}{4 \pi |\bx-\by|}.
$$

Given the incident field $\vecE^i(\cdot, \by, \vecp)$, the corresponding direct scattering problem is finding scattered electric field $\vecE^s(\cdot, \by, \vecp)$ and total electric field $\vecE(\cdot, \by, \vecp)$ such that
\begin{eqnarray}
\mcurl^2 \vecE^s - k^2  \vecE^s = 0 \quad &\mbox{in}& \quad D, \label{TotalEqn1} \\
\mcurl (A \mcurl  \vecE) - k^2 N \vecE = 0 \quad &\mbox{in}& \quad \R^3 \backslash \overline{D}, \label{TotalEqn2} \\
\nu \times \vecE - \nu \times \vecE^s = \nu \times \vecE^i \quad &\mbox{on}& \quad \partial D, \label{TotalEqn3} \\
\nu \times  A\mcurl \vecE - \nu \times \mcurl \vecE^s = \nu \times \mcurl\vecE^i \quad &\mbox{on}& \quad \partial D, \label{TotalEqn4} \\
\lim\limits_{|\bx| \to \infty}\left(\mcurl \vecE \times \bx-ik|\bx|\vecE\right)=0. \label{TotalEqn5}
\end{eqnarray}
where $\nu$ denotes the unit outward normal to $\partial D$. Moreover, we denote by $A(\bx) := \mu_0 \boldsymbol{\mu}(\bx)^{-1}$, and $N(x):= \boldsymbol{\epsilon}(\bx)/\epsilon_0$. Outside the inhomogeneity $D_1 \backslash \overline{D}$, $A(\bx)$ and $N(\bx)$ are identity matrices by our assumption on $\boldsymbol{\epsilon}(\bx), \boldsymbol{\mu}(\bx)$. The last equation \eqref{TotalEqn5} is the Silver-M{\"u}ller radiation condition. For a detailed modeling of the propagation of electromagnetic waves, we refer to the monograph \cite{KH}.

Using the fact that $\vecE(\cdot, \by, \vecp)=\vecE^s(\cdot, \by, \vecp)+\vecE^i(\cdot, \by, \vecp)$. We can represent the forward scattering problem \eqref{TotalEqn1}--\eqref{TotalEqn5} in terms of the scattered electric field $\vecE^s$ as
\begin{eqnarray}
\mcurl( A\mcurl \vecE^s) - k^2 N \vecE^s = \mcurl \left( (\boldsymbol{I}-A)\mcurl\vecf \right)  - k^2\left(\boldsymbol{I}- N\right) \vecf  \quad &\mbox{in}& \quad \R^3, \label{Scattered} \\
\lim\limits_{|\bx| \to \infty}\left(\mcurl \vecE^s\times \bx-ik|\bx|\vecE^s\right)=0.\label{SilverMuller}
\end{eqnarray}
where $\vecf= \vecE^i$. Note that the problem  \eqref{Scattered}--\eqref{SilverMuller} can be studied for general $\vecf$.

{Let $\Sigma := \partial C$, where $C$ is a ball inside $D$}. Our measurements are the scattered  electric fields $\vecE^s(\bx, \by, \vecp)$ for all $\bx \in  \Sigma$, $\by \in  \Sigma$ and $\vecp \in \R^3$. The {\bf inverse problem} we considered in this paper is:

\medskip
{\it
Given $\vecE^s(\bx,\by,\vecp)$ for all $\bx \in  \Sigma$, $\by \in  \Sigma$ and $\vecp \in \R^3$,  can we uniquely determine  the boundary $\partial D$? How can we reconstruct $\partial D$?
}

\vspace{1\baselineskip}

The paper is further organized as follows. In Section \ref{Preliminary} we present several preliminary results. In Section \ref{MaxwellETE}, we study the so-called {\it exterior transmission problem} for Maxwell's equations. We show that the exterior transmission problem is well-posed when the wave number $k$ is purely imaginary with sufficiently large amplitude. This allows us, in particular, to prove that the measurements uniquely determine the cavity in Section \ref{Uniqueness}. In section \ref{Uniqueness}, we show the uniqueness based on the exterior transmission problem. This allow the inhomogeneous background medium is general since we avoid constructing Green's function in the proof. In Section \ref{LSM}, we introduce the linear sampling method and design a reconstruction algorithm to reconstruct the cavity. In Section \ref{Numerics}, we provide numerical examples to illustrate the viability of the linear sampling method algorithm. Finally, we conclude with a summary in Section \ref{Discussion}.
\section{Preliminary results} \label{Preliminary}
We first introduce the following functional spaces. For any bounded connected domain $\Omega$ {(where $\Omega$ satisfies the assumptions on $D$, i.e., $\Omega$ has $C^2$-smooth boundaries)}, let $\nu$ denote its outward normal, we set $\vecL^2(\Omega):=L^2(\Omega)^3$,
$\vecH^m(\Omega):=H^m(\Omega)^3$ and
define
\begin{eqnarray*}
\vecH(\mcurl, \Omega)&:=&\left \{  \vecu \in \vecL^2(\Omega):\mcurl \vecu\in \vecL^2(\Omega) \right\}, \\
\vecH(\mdiv, \Omega)&:=&\left \{  \vecu \in \vecL^2(\Omega):\mdiv \vecu\in L^2(\Omega) \right\}, \\
\vecL_t^2(\partial \Omega) &:=& \left \{  \vecu \in L^2(\partial \Omega)^3:\nu \cdot \vecu=0 \right\}, \\
\vecH^{-\frac{1}{2}}(\mdiv,\partial  \Omega)&:=&\left \{  \vecu \in \vecH^{-1/2}(\partial \Omega): \nu\cdot \vecu =0 ,~\mbox{div}_{\partial \Omega} \vecu \in H^{-\frac{1}{2}}(\partial \Omega) \right\}, \\
\vecH^{-\frac{1}{2}}(\mcurl, \partial \Omega)&:=&\left \{  \vecu \in \vecH^{-1/2}(\partial \Omega): \nu\cdot \vecu =0 ,~\mbox{curl}_{\partial \Omega} \vecu \in H^{-\frac{1}{2}}(\partial \Omega) \right\}.
\end{eqnarray*}
We define $\vecH_{loc}(\mcurl, \R^3 \backslash \overline{\Omega})$ as the usual local space. For any $\vecu \in \vecH(\mcurl, \Omega)$, the tangential trace $\nu \times \vecu \in \vecH^{-\frac{1}{2}}(\mdiv, \partial \Omega)$ and $(\nu \times \vecu) \times \nu \in \vecH^{-\frac{1}{2}}(\mcurl, \partial \Omega)$  are well defined (c.f. \cite{KH,buffa2002traces}).
\subsection{Well-posedness of the scattering problem}

The forward scattering problem \eqref{Scattered}--\eqref{SilverMuller} can be studied variationally as in \cite{KM}. We can directly obtain the following lemma.
\begin{lemma} \label{ForwardMaxwellFredholm}
Assume there exists $\gamma_{1,2} >0$ such that
\begin{eqnarray*}
(A \xi,\xi) \ge \gamma_1 |\xi|^2, \quad \forall \xi \in \C^3\quad \mbox{a.e. in} \quad \R^3, \\
(N \xi,\xi) \ge \gamma_2 |\xi|^2, \quad \forall \xi \in \C^3\quad \mbox{a.e. in} \quad \R^3.
\end{eqnarray*}
Then the forward problem  (\ref{Scattered})-(\ref{SilverMuller}) satisfies the Fredholm alternative.
\end{lemma}

From Lemma \ref{ForwardMaxwellFredholm}, the existence of a unique solution  to (\ref{Scattered})-(\ref{SilverMuller}) follows directly from   following lemma.
\begin{lemma} \label{ForwardMaxwellUnique}
There exists at most one solution to (\ref{Scattered})-(\ref{SilverMuller}).
\end{lemma}

The proof of the lemma is classical and for completeness we postpone the proof in the Appendix Section \ref{Appendix}. Combine Lemma \ref{ForwardMaxwellFredholm} and Lemma \ref{ForwardMaxwellUnique} we can have
\begin{lemma} \label{ForwardMaxwell}
Assume the assumptions in Lemma \ref{ForwardMaxwellFredholm} holds.
Then there exists a unique solution $\vecE^s \in \vecH_{loc}(\mcurl, \R^3)$ to (\ref{Scattered})-(\ref{SilverMuller}) depending continuously on $\vecf \in \vecH(\mcurl, D_1 \backslash \overline{D})$ such that
$$
\|\vecE^s\|_{\vecH(\mcurl,B_R)} \le c \|\vecf\|_{\vecH(\mcurl, D_1 \backslash \overline{D})},
$$
where $B_R$ is a sufficiently large ball in $\R^3$ that contains $D_1$, {and $c>0$ is a constant that depends on $D_1$, $D$, $B_R$, $A$, and $N$.}
\end{lemma}

\subsection{Maxwell eigenvalue and regularity properties}
In the exterior inverse scattering problem using far field measurements, one can conclude that two scattered fields coincide in the exterior of an inhomogeneous medium if their far field patterns coincide. As an analog in the interior inverse scattering problem using internal measurements, one expects to conclude that two scattered fields coincide in the homogeneous cavity if their internal measurements coincide on {$\partial C$, where $\partial C$ (introduced in Section \ref{Introduction}) is the ball surface where the measurements are on}. However, to achieve this, one need to assume that $k^2$ is not a Maxwell eigenvalue for $C$. More precisely, we define $k^2$ is a {\bf Maxwell eigenvalue} for $C$ if there exists a nontrivial solution $\vecu \in \vecH(\mcurl, C)$ such that
\begin{eqnarray*}
\mcurl^2 \vecu- k^2 \vecu = 0 \quad &\mbox{in}& \quad C, \\
\nu \times \vecu = 0 \quad &\mbox{on}& \quad \partial C.
\end{eqnarray*}
We now state the following lemma.
\begin{lemma} [Theorem 4.32 \cite{KH}]
Assume that $k^2$ is not a Maxwell eigenvalue for $C$. Then there exists a unique solution $\vecu \in \vecH(\mcurl, C)$ depending continuously on $\vecg \in \vecL^2(C)$ such that
\begin{eqnarray*}
\mcurl^2 \vecu- k^2 \vecu = \vecg \quad &\mbox{in}& \quad C, \\
\nu \times \vecu = 0 \quad &\mbox{on}& \quad \partial C.
\end{eqnarray*}
\end{lemma}

In the {remainder} of the paper, we assume that $k^2$ is not a Maxwell eigenvalue for $C$ (note that this is not a restriction since we can choose the measurement ball $C$ such that this assumption holds).

For our uniqueness result in Section \ref{Uniqueness}, we need the following regularity property which allows us to obtain the compact embedding for Maxwell's equations.

\begin{lemma} \label{MaxwellTransmissionRegularity}
Let $\vecE^s$ be the unique solution to the scattering problem  (\ref{Scattered})-(\ref{SilverMuller}) where $\vecf=\vecE^i(\cdot, \by, \vecp)$ defined via \eqref{Intro_Ei}. Let $B_R$ be a sufficiently large ball in $\R^3$ that contains $D_1$. Then $\vecE^s \in \vecH^1(D)$, $\vecE^s \in \vecH^1(D_1 \backslash \overline{D})$, $\vecE^s \in \vecH^1(B_R \backslash \overline{D}_1)$ and
\begin{eqnarray*}
\|\vecE^s\|_{\vecH^1(\Omega)} \le C \left(  \|\vecE^s\|_{\vecH(\mcurl, B_R)} +\|\vecE^i\|_{\vecH^1(B_R \backslash \overline{D})} \right),
\end{eqnarray*}
where $C$ is a constant and $\Omega$ can be $D$, $D_1 \backslash \overline{D}$ and $B_R \backslash \overline{D}_1$.
\end{lemma}

For sake of completeness we include the proof in the Appendix Section \ref{Appendix}.

\begin{remark}\label{MaxwellTransmissionRegularityRemark}
If in addition $\mcurl \vecE^i(\cdot, \by, \vecp) \in \vecH_{loc}^1(\R^3 \backslash \overline{D})$, Lemma \ref{MaxwellTransmissionRegularity} implies that $A\mcurl \vecE^s \in \vecH^1(\Omega)$ where $\Omega$ can be $D$, $D_1 \backslash \overline{D}$ and $B_R \backslash \overline{D}_1$ (see for instance, Theorem 2.5, Theorem 2.6, and Remark 2.7 in \cite{CaCo1}). This can be proved exactly in the same way by introducing $\vecH = A\mcurl \vecE$ and by writing the electromagnetic scattering problem in terms of $\vecH$.
\end{remark}

For later analyses, we need the following regularity result. It is a direct consequence of $\mdiv \vecu=0$ and interior elliptic regularity properties (c.f. \cite{GiT}).

\begin{lemma} \label{MaxwellInteriorRegularity}
Assume $\vecu \in \vecH(\mcurl,\Omega)$ satisfies
$$
\mcurl^2 \vecu - k^2 \vecu =0 \quad \mbox{in} \quad \Omega.
$$
Then
$\vecu \in \vecH^s(\Omega')$ for any $s>0$ where $\Omega' \,{\subset\subset}\, \Omega$, i.e. $\vecu$ is analytic in $\Omega'$.
\end{lemma}
\section{The exterior transmission problem} \label{MaxwellETE}
To show that the cavity $D$ is uniquely determined by the measurements $\vecE^s(\bx,\by,\vecp)$ for all $\bx \in \Sigma$, $\by \in \Sigma$ and $\vecp \in \R^3$, we follow the approach made in \cite{CaCo1, CaCoMe, Hah,Ch}. In this approach, one needs to consider the following boundary value problem which is refereed as the {\it exterior transmission problem}: find nontrivial $\widetilde{\vecw} \in \vecH_{loc}(\mcurl, \R^3 \backslash \overline{D} )$ and $\widetilde{\vecv} \in \vecH_{loc}(\mcurl, \R^3 \backslash \overline{D} )$ that satisfies the following
\begin{eqnarray}
\mcurl(A\mcurl \widetilde{\vecw}) - k^2 N \widetilde{\vecw} = \widetilde{\vecf}_1 \quad &\mbox{in}& \quad \R^3 \backslash \overline{D}, \label{ETEMaxwell1} \\
\mcurl^2 \widetilde{\vecv} - k^2  \widetilde{\vecv} = \widetilde{\vecf}_2 \quad &\mbox{in}& \quad \R^3 \backslash \overline{D}, \label{ETEMaxwell2} \\
\nu \times \widetilde{\vecw} - \nu \times \widetilde{\vecv} = \widetilde{\vecg} \quad &\mbox{on}& \quad \partial D, \label{ETEMaxwell3} \\
\nu \times A \mcurl \widetilde{\vecw}  - \nu \times \mcurl \widetilde{\vecv} = \widetilde{\vech} \quad &\mbox{on}& \quad \partial D, \label{ETEMaxwell4}\\
\lim\limits_{|\bx|\to \infty}\left(\mcurl \widetilde{\vecu} \times \bx-ik|\bx|\widetilde{\vecu} \right)=0,&& \label{ETEMaxwell5} \\
\lim\limits_{|\bx|\to \infty}\left(\mcurl \widetilde{\vecv} \times \bx-ik|\bx|\widetilde{\vecv}\right)=0,&& \label{ETEMaxwell6}
\end{eqnarray}
where $\widetilde{\vecf}_{1,2} \in \vecL_{loc}^2(\R^3 \backslash \overline{D} )$ have compact support in $B_R \backslash \overline{D}$,  $\widetilde{\vecg} \in \vecH^{-\frac{1}{2}}(\mdiv,\partial D)$ and $\widetilde{\vech} \in \vecH^{-\frac{1}{2}}(\mdiv,\partial D)$.
\begin{definition}
Values of $k \in \R$ such that the homogeneous problem (i.e. $\widetilde{\vecf}_{1,2}=0$, $\widetilde{\vecg}=0$ and $\widetilde{\vech}=0$) of (\ref{ETEMaxwell1})-(\ref{ETEMaxwell6}) has a nontrivial solution are called {\it exterior transmission eigenvalues}.
\end{definition}

In the following, we will show that the exterior transmission problem  (\ref{ETEMaxwell1})-(\ref{ETEMaxwell6}) is well-posed when $k = i \kappa$ where $\kappa \in \R$ is sufficiently large. We formulate the exterior transmission problem variationally and employ the T-coercity approach in \cite{Ch}. We also mention related work \cite{li2019maxwell} on exterior transmission eigenvalues for Maxwell equations.

\subsection{Variational formulation}
We first formulate the exterior transmission  problem  (\ref{ETEMaxwell1})-(\ref{ETEMaxwell6}) in a bounded domain. {Recall that $B_R$ is a sufficiently large ball that contains the support of $D_1 \backslash \overline{D}$, let $S_R=\partial B_R$.} Let us first recall the following results. 
\begin{enumerate}
\item
Let $G_{k}$ be the exterior Calderon operator, an isomorphism between $\vecH^{-\frac{1}{2}}(\mdiv, S_R)$ and $\vecH^{-\frac{1}{2}}(\mdiv, S_R)$, that maps a tangential vector field $\hat{\bx} \times \vecE^s$ on $S_R$ to $\hat{\bx} \times \mcurl \vecE^s$ on $S_R$ where $\vecE^s$ satisfies
\begin{eqnarray*}
\mcurl^2 \vecE^s -k^2  \vecE^s = 0 \quad &\mbox{in}& \quad \R^3 \backslash \overline{B_R}, \\
\lim\limits_{|\bx|\to \infty}\left(\mcurl \vecE^s \times \bx-ik|\bx|\vecE^s\right)=0,
\end{eqnarray*}
here $\hat{\bx}=\bx/|\bx|$. The exterior Calderon operator allows {us to} reformulate the unbounded domain problem on a bounded domain (c.f. \cite{KM}).
\item For any $\widetilde{\vecg} \in \vecH^{-\frac{1}{2}}(\mdiv, \partial D)$, there exists a lifting function $\vecv_l\in \vecH_{loc}(\mcurl, \R^3 \backslash \overline{D} )$ such that
\begin{eqnarray*}
\mcurl^2 \vecv_l -k^2  \vecv_l = 0 \quad &\mbox{in}& \quad \R^3 \backslash \overline{D}, \\
\nu \times \vecv_l  = \widetilde{\vecg} \quad &\mbox{on}& \quad \partial D, \\
\lim\limits_{|\bx|\to \infty}\left(\mcurl \vecv_l \times \bx-ik|\bx|\vecv_l\right)=0.
\end{eqnarray*}
Moreover we have
\begin{eqnarray} \label{MaxwellLifting}
\|\vecv_l\|_{\vecH(\mcurl,B_R \backslash \overline{D})} \le M  \|\widetilde{\vecg}\|_{ H^{-\frac{1}{2}}(\mdiv, \partial D)},
\end{eqnarray}
where $M$ is a constant.
\end{enumerate}
From the two results above, we can formulate the exterior transmission problem  (\ref{ETEMaxwell1})-(\ref{ETEMaxwell6}) in a bounded domain. To be more precise, let $\vecw:= \widetilde{\vecw}$, $\vecv := \widetilde{\vecv} + \vecv_\ell$ and ${\vecf}_{1,2}:=\widetilde{\vecf}_{1,2}$ in $B_R \backslash \overline{D}$, and $\vech:= \widetilde{\vech} - \nu \times \mcurl \vecv_\ell$ on $\partial D$, then we have that $\vecw, \vecv \in \vecH(\mcurl, B_R \backslash \overline{D} )$ satisfy 
\begin{eqnarray}
\mcurl(A\mcurl \vecw) - k^2 N \vecw = \vecf_1 \quad &\mbox{in}& \quad B_R \backslash \overline{D}, \label{WellposedETEMaxwell1} \\
\mcurl^2 \vecv - k^2  \vecv = \vecf_2 \quad &\mbox{in}& \quad B_R \backslash \overline{D}, \label{WellposedETEMaxwell2} \\
\nu \times \vecw - \nu \times \vecv = 0 \quad &\mbox{on}& \quad \partial D, \label{WellposedETEMaxwell3} \\
\nu \times A \mcurl \vecw  - \nu \times \mcurl\vecv = \vech \quad &\mbox{on}& \quad \partial D, \label{WellposedETEMaxwell4}\\
\hat{\bx} \times  \mcurl \vecw   = G_{k} (\hat{\bx} \times \vecw) \quad &\mbox{on}& \quad S_R, \label{WellposedETEMaxwell5} \\
\hat{\bx} \times  \mcurl \vecv   = G_{k} (\hat{\bx} \times \vecv) \quad &\mbox{on}& \quad S_R. \label{WellposedETEMaxwell6}
\end{eqnarray}
Now we introduce the Hilbert space $\vecU$
$$
\vecU := \{ (\vecw,\vecv): \vecw \in \vecH(\mcurl, B_R \backslash \overline{D}),\,\,  \vecv \in \vecH(\mcurl, B_R \backslash \overline{D}),\,\, \nu \times \vecw=\nu\times\vecv \,\, \mbox{on} \,\, \partial D\}.
$$
Let  $\left(\cdot,\cdot\right)$ be the usual $\vecL^2(B_R \backslash \overline{D})$ inner product. For a generic bounded domain $\Omega$, let $\left<\cdot,\cdot\right>_{\partial \Omega}$ be the duality pairing between  $\vecH^{-\frac{1}{2}}(\mdiv, \partial \Omega)$ and $\vecH^{-\frac{1}{2}}(\mcurl, \partial \Omega)$. And denote the tangential component of $\vecu$ by
$$
\vecu_T:=(\nu \times \vecu) \times \nu \quad \mbox{on} \quad \partial \Omega.
$$
Now we can derive a variational formulation of (\ref{WellposedETEMaxwell1})-(\ref{WellposedETEMaxwell6}).
\begin{lemma}
The exterior transmission problem (\ref{WellposedETEMaxwell1})-(\ref{WellposedETEMaxwell6}) is equivalent to the following variational problem
\begin{eqnarray} \label{MaxwellETEVariational}
a_k\left( \left(\vecw,\vecv\right), \left(\vecw',\vecv'\right) \right) = F(\vecw',\vecv'), \quad \forall \quad (w',v') \in \vecU
\end{eqnarray}
where
\begin{eqnarray*}
\hspace{-2.422cm}a_k \left( (\vecw,\vecv), (\vecw',\vecv') \right)&:=& \left(A\mcurl \vecw, \mcurl \vecw' \right) -\left(\mcurl \vecv, \mcurl \vecv' \right) - k^2 \left(N \vecw,  \vecw' \right) +k^2  \left(\vecv,  \vecv' \right) \\
&+& \left<G_{k} (\hat{x} \times \vecw),\vecw'_T \right>_{S_R} -\left<G_{k} (\hat{x} \times \vecv),\vecv'_T\right>_{S_R},
\end{eqnarray*}
and
\begin{eqnarray*}
F (\vecw',\vecv') :=  \left(\vecf_1,  \vecw' \right) -  \left(\vecf_2,  \vecv' \right) + \left<\vech,\vecv'_T\right>_{\partial D}.
\end{eqnarray*}

\end{lemma}
\begin{proof}
On one hand, assume that $(\vecw,\vecv)$ is a solution to the boundary value problem  (\ref{WellposedETEMaxwell1})-(\ref{WellposedETEMaxwell6}). Let $(\vecw',\vecv') \in \vecU$ be a test function. Multiply $\vecw'$ to equation (\ref{WellposedETEMaxwell1}) and integrate by parts ({first for smooth fields and then followed by a density argument}) to get
$$
\left( A\mcurl \vecw, { \mcurl}\vecw'\right) -k^2 \left(N \vecw,  \vecw' \right) + \left<G_{k} (\hat{x} \times \vecw),\vecw'_T\right>_{S_R} - \left<\nu \times A\mcurl \vecw,\vecw'_T\right>_{\partial D} = \left(\vecf_1,\vecw'\right).
$$
Multiply $\vecv'$ to equation (\ref{WellposedETEMaxwell2}) and integrate by parts to get
$$
\left( \mcurl \vecv, { \mcurl}\vecv'\right) -k^2 \left( \vecv,  \vecv' \right) + \left<G_{k} (\hat{x} \times \vecv),\vecv'_T\right>_{S_R} - \left<\nu \times \mcurl \vecv,\vecv'_T\right>_{\partial D} = \left(\vecf_2,\vecv'\right).
$$
Take the difference of the above two equations and note that $ \nu \times \vecw'=\nu\times\vecv'$ on $\partial D$, we can derive
\begin{eqnarray*}
&& \left(A\mcurl \vecw, \mcurl \vecw' \right) -\left(\mcurl \vecv, \mcurl \vecv' \right) - k^2 \left(N \vecw,  \vecw' \right) +k^2  \left(\vecv,  \vecv' \right) \\
&+& \left<G_{k} (\hat{x} \times \vecw),\vecw_T'\right>_{S_R} -\left<G_{k} (\hat{x} \times \vecv),\vecv'_T\right>_{S_R} \\
&=&  \left(\vecf_1,  \vecw' \right) -  \left(\vecf_2,  \vecv' \right) + \left<\vech,\vecv'_T\right>_{\partial D}.
\end{eqnarray*}
This yields the variational formulation (\ref{MaxwellETEVariational}).

On the other hand, if $(\vecw,\vecv)$ satisfies the variational formulation (\ref{MaxwellETEVariational}). Let $\vecv'=0$, then
$$
\left( A\mcurl \vecw, \mcurl \vecw'\right) -k^2 \left(N \vecw,  \vecw' \right) + \left<G_{k} (\hat{x} \times \vecw),\vecw'_T\right>_{S_R} - \left<\nu \times A \mcurl \vecw,\vecw'_T\right>_{\partial D} = \left(\vecf_1,\vecw'\right).
$$
Let $\vecw' \in \big(C_0^\infty(B_R \backslash \overline{D})\big)^3$, we have that $\vecw$ satisfies equation  (\ref{WellposedETEMaxwell1}) in the distributional sense. Then by duality pairing between  $\vecH^{-\frac{1}{2}}(\mdiv, S_R)$ and $\vecH^{-\frac{1}{2}}(\mcurl, S_R)$, we can derive equation  (\ref{WellposedETEMaxwell5}). In the exact same way, we can have $\vecv$ satisfies equations  (\ref{WellposedETEMaxwell2}) and (\ref{WellposedETEMaxwell6}). Finally we can have again from duality pairing between  $\vecH^{-\frac{1}{2}}(\mdiv, \partial D)$ and $\vecH^{-\frac{1}{2}}(\mcurl, \partial D)$ that equation (\ref{WellposedETEMaxwell4}) holds. This proves the lemma. 
\end{proof}
\subsection{T-coercivity approach}
Now we introduce the T-coercivity approach. Note that by means of the Riesz representation theorem we can define the operator ${\mathcal A}_k: \vecU \to \vecU$  by
$$
\left({\mathcal A}_k(\vecw,\vecv),(\vecw',\vecv')\right)_\vecU= a_k((\vecw,\vecv),(\vecw',\vecv')) \qquad \mbox{ for all } ((\vecw,\vecv),(\vecw',\vecv')) \in \vecU\times \vecU.
$$
The idea behind the T-coercivity method is to consider an equivalent formulation of (\ref{MaxwellETEVariational}) where $a_k$ is replaced by $a_k^{{\mathcal T}}$ defined by
\begin{equation}\label{forme T coercive}
a_k^{\mathcal T}((\vecw,\vecv),(\vecw',\vecv')):=a_k((\vecw,\vecv),{\mathcal T}(\vecw',\vecv')),\quad\forall ((\vecw,\vecv),(\vecw',\vecv'))\in \vecU \times \vecU
\end{equation}
with ${\mathcal T}$ being an \textit{ad hoc} isomorphism of $\vecU$. To show that ${\mathcal A}_k$ is an isomorphism on ${\vecU}$, it is sufficient to show that $a_k^{\mathcal T}$ is coercive.

Let $\Omega$ be a neighborhood (sufficiently close to $\partial D$) of $\partial D$. Setting  ${\mathcal N}:= (B_R \backslash \overline{D})\cap \Omega$, we  denote by
\begin{equation}\label{MaxwellAssumpETEAN}
\begin{array}{lll}
A_{*}:=\underset{x\in{\mathcal N}}{\mbox{ inf }}\underset{|\xi|=1}{\mbox{ inf }} \overline{\xi}\cdot A(x)\xi>0,  & A^*:=\underset{x\in {\mathcal N}}{\mbox{ sup }}\underset{|\xi|=1}{\mbox{ sup }}\overline{\xi}\cdot A(x)\xi <\infty,\ \\
N_{*}:=\underset{x\in{\mathcal N}}{\mbox{ inf }}\underset{|\xi|=1}{\mbox{ inf }} \overline{\xi}\cdot N(x)\xi>0,  & N^*:=\underset{x\in {\mathcal N}}{\mbox{ sup }}\underset{|\xi|=1}{\mbox{ sup }}\overline{\xi}\cdot N(x)\xi <\infty,\ \\
\end{array}
\end{equation}
for all ${\xi}\in\mathbb{C}^3$.
Then we can prove the following result.
\begin{lemma}\label{MaxwellETEInvertible}
Assume that either $A^*<1$ and ${N}^*<1$ or $A_*>1$ and ${N}_*>1$.  Then there exists $\kappa>0$ such that ${\mathcal A}_{i\kappa}$ is invertible.
\end{lemma}
\begin{proof}
We first consider the case when  $A^*<1$ and $N^*<1$. Let $\chi \in C^{\infty}(\overline{B_R})$ be a cut off function equal to 1 in a neighborhood of $\partial D$  such that $|\chi|\le 1$, and $\chi=0$ in $\{B_R\backslash \overline{D}\} \backslash \overline{\mathcal N}$.
Let ${\mathcal T}(w,v)=(w-2\chi v,-v)$. We then have that
\begin{eqnarray}
a_{i\kappa}^{\mathcal T}((\vecw,\vecv),(\vecw,\vecv))
&=& (A \mcurl \vecw , \mcurl \vecw) + (\mcurl \vecv ,\mcurl \vecv) - 2(A \mcurl \vecw , \mcurl (\chi \vecv)) \nonumber \\
&+& {\kappa}^2(   (N \vecw,\vecw) +(\vecv,\vecv) - 2(N\vecw,\chi \vecv) ) +\left<G_{i{\kappa }} (\nu \times \vecw),\vecw_T\right>_{S_R} \nonumber \\
&+& \left<G_{i\kappa} (\nu \times \vecv),\vecv_T\right>_{S_R}-2\left<G_{i\kappa} (\nu \times \vecw),\chi \vecv_T\right>_{S_R}. \label{MaxwellETEInvertible_1}
\end{eqnarray}
By Young's inequality we have that
\begin{eqnarray}
2|(A\mcurl \vecw, \mcurl (\chi \vecv))|
&\le&  2 |(A\mcurl \vecw, \chi \mcurl \vecv)|+ 2 |(A\mcurl \vecw, \nabla \chi \times \vecv)| \nonumber \\
&\le& \alpha(A\mcurl \vecw, \mcurl \vecw)_{\mathcal N}+\alpha^{-1}(A\mcurl \vecv, \mcurl \vecv)_{\mathcal N} \nonumber \\
&+& \beta (A\mcurl \vecw,\mcurl \vecw)_{\mathcal N}+\beta^{-1}(A\nabla \chi \times \vecv, \nabla \chi \times \vecv)_{\mathcal N}, \qquad \label{MaxwellETEInvertible_2}
\end{eqnarray}
where $(\cdot,\cdot)_{\mathcal N}$ denotes the $\vecL^2$ inner product in $\mathcal N$ and $\alpha>0$, $\beta>0$ are constants to be chosen later. Similarly we have that
\begin{equation}
2|(N \vecw,\chi \vecv)|
\le 2|(N \vecw,\vecv)_{\mathcal N}| \le \eta(N \vecw,\vecw)_{\mathcal N}+\eta^{-1}(N\vecv,\vecv)_{\mathcal N}
 \label{MaxwellETEInvertible_3}
\end{equation}
for some constant $\eta>0$ to be chosen later.

Since $k=i \kappa$, we have that $\vecw$ and $\vecv$ decay exponentially when  $|\bx| \to \infty$. This yields
\begin{equation}
\left<G_{i{\kappa }} (\nu \times \vecw),\vecw_T\right>_{S_R}=\int\limits_{{\mathbb R}^3\setminus \overline{B_R}}\left(|\mcurl \vecw|^2+\kappa^2|\vecw|^2\right)\,d\bx \ge 0, \label{MaxwellETEInvertible_4}
\end{equation}
and similarly  $\left<G_{i{\kappa }} (\nu \times \vecv),\vecv_T\right>_{S_R} \ge 0$. Note that $\chi$ is compactly supported near $\partial D$, then
\begin{equation}
\left< G_{i\kappa} (\nu \times \vecw),\chi \vecv_T\right>_{S_R}=0. \label{MaxwellETEInvertible_5}
\end{equation}
From \eqref{MaxwellETEInvertible_1} -- \eqref{MaxwellETEInvertible_5}, we can obtain that
\begin{eqnarray*}
\hspace{-2.522cm}|a_{i\kappa}^{\mathcal T}((\vecw,\vecv),(\vecw,\vecv))|
&\ge&  (A \mcurl \vecw , \mcurl \vecw)_{\{B_R\backslash \overline{D}\} \backslash \overline{\mathcal N}}+ (\mcurl \vecv ,\mcurl \vecv)_{\{B_R\backslash \overline{D}\} \backslash \overline{\mathcal N}}\\
&+& \kappa^2\left(  (N \vecw,\vecw)_{\{B_R\backslash \overline{D}\} \backslash \overline{\mathcal N}}+(\vecv,\vecv)_{\{B_R\backslash \overline{D}\} \backslash \overline{\mathcal N}}\right) \\
&+& (1-\alpha-\beta)(A \mcurl \vecw , \mcurl \vecw)_{\mathcal N}+((I-\alpha^{-1}A)\mcurl \vecv ,\mcurl \vecv)_{\mathcal N}\\
&+& \kappa^2(1-\eta)(N\vecw,\vecw)_{\mathcal N}+(\kappa^2(1-\eta^{-1}N)-\sup |\nabla \chi|^2A^* \beta^{-1})\vecv,\vecv)_{\mathcal N}.
\end{eqnarray*}
Taking $\alpha,\beta,\eta, \kappa$ such that $A^*<\alpha<1$, $N^*<\eta<1$, and $\beta < 1-\alpha$, then $a_{i\kappa}^{\mathcal T}$ is coercive for sufficiently large $\kappa$.

The case when   $A^*>1$ and $N^*>1$ can be proven the same way using ${\mathcal T}(\vecw,\vecv)=(\vecw,-\vecv+2\chi \vecw)$. This proves the lemma. 
\end{proof}

Assume that $A$ and $N$ satisfy the assumptions in Lemma \ref{MaxwellETEInvertible}.  Let $k=i\kappa$ where $\kappa \in \R$ is sufficiently large. Then from Lemma \ref{MaxwellETEInvertible}, there is a unique solution $(\vecw,\vecv)$ to \eqref{WellposedETEMaxwell1} -- \eqref{WellposedETEMaxwell6}, and
\begin{eqnarray*}
&&||\vecw||_{\vecH(\mcurl,B_R \backslash \overline{D})}+||\vecv||_{\vecH(\mcurl, B_R \backslash \overline{D})} \\
&\le& C\left( ||\vecg||_{ \vecH^{-\frac{1}{2}}(\mdiv, \partial D)}+||\vech||_{ \vecH^{-\frac{1}{2}}(\mdiv, \partial D)} + ||\vecf_1||_{\vecL^2(B_R \backslash \overline{D})} +||\vecf_2||_{\vecL^2(B_R \backslash \overline{D})} \right).
\end{eqnarray*}
Let $\widetilde{\vecw}$ and $\widetilde{\vecv}$ be the solution to the exterior transmission problem (\ref{ETEMaxwell1})-(\ref{ETEMaxwell6}). Recall that $\vecw= \widetilde{\vecw}$, $\vecv= \widetilde{\vecv} + \vecv_\ell$, ${\vecf}_{1,2}=\widetilde{\vecf}_{1,2}$, $\vech= \widetilde{\vech} - \nu \times \mcurl \vecv_\ell$, and the lifting function $\vecv_\ell$ satisfies equation (\ref{MaxwellLifting}), then we can immediately obtain the following theorem.
\begin{theorem}\label{MaxwellETEFredholm}
Assume that $A$ and $N$ satisfy the assumptions in Lemma \ref{MaxwellETEInvertible}.  Let $k=i\kappa$ where $\kappa \in \R$ is sufficiently large. Then the exterior transmission problem (\ref{ETEMaxwell1})-(\ref{ETEMaxwell6}) has a unique solution which depends continuously on the data $\widetilde{\vecf}_1$, $\widetilde{\vecf}_2$, $\widetilde{\vecg}$ and $\widetilde{\vech}$
\begin{eqnarray*}
&&||\widetilde{\vecw}||_{\vecH(\mcurl,B_R \backslash \overline{D})}+||\widetilde{\vecv}||_{\vecH(\mcurl, B_R \backslash \overline{D})} \\
&\le& C\left( ||\widetilde{\vecg}||_{ \vecH^{-\frac{1}{2}}(\mdiv, \partial D)}+||\widetilde{\vech}||_{ \vecH^{-\frac{1}{2}}(\mdiv, \partial D)} + ||\widetilde{\vecf}_1||_{\vecL^2(B_R \backslash \overline{D})} +||\widetilde{\vecf}_2||_{\vecL^2(B_R \backslash \overline{D})} \right),
\end{eqnarray*}
where $C>0$ is a constant independent of $\widetilde{\vecf}_1$, $\widetilde{\vecf}_2$, $\widetilde{\vecg}$ and $\widetilde{\vech}$.\end{theorem}

\begin{remark} \label{InhomogeneousCaderon}
If we replace (\ref{WellposedETEMaxwell5}) and (\ref{WellposedETEMaxwell6}) by
\begin{eqnarray}
\hat{x} \times  \mcurl \vecw   - G_{k} \hat{x} \times \vecw =\vech_1 \quad &\mbox{on}& \quad S_R,  \label{WellposedETEMaxwell5_1}\\
\hat{x} \times  \mcurl \vecv   - G_{k} \hat{x} \times \vecv =\vech_2 \quad &\mbox{on}& \quad S_R, \label{WellposedETEMaxwell6_1}
\end{eqnarray}
where $\vech_{1,2} \in \vecH^{-\frac{1}{2}}(\mdiv, S_R)$, we can see from Lemma \ref{MaxwellETEInvertible} and Theorem \ref{MaxwellETEFredholm} that the new exterior transmission problem (where we replace (\ref{WellposedETEMaxwell5}) -- (\ref{WellposedETEMaxwell6}) by (\ref{WellposedETEMaxwell5_1}) -- (\ref{WellposedETEMaxwell6_1}))   has a unique solution that depends continuously on the data ($k=i \kappa$ and $\kappa$ is sufficiently large).
\end{remark}
\section{Uniqueness of the inverse problem} \label{Uniqueness}
In order to show that boundary $\partial D$ is uniquely determined by the measurements $\vecE^s(\bx,\by,\vecp)$ for all $\bx \in \Sigma$, $\by \in \Sigma$ and $\vecp \in \R^3$. We follow the approach made in \cite{ CaCo1, CaCoMe, Hah}. 
\subsection{Reciprocity relation}
First of all, we need the following reciprocity relation to prove the uniqueness. The reciprocity relation also plays an important role in the linear sampling method discussed later.
\begin{lemma} \label{ReciprocityMaxwell}
Let $\vecE^s(\cdot,\by,\vecp)$ be the scattered electric field satisfying equations  (\ref{Scattered})-(\ref{SilverMuller}) corresponding to $\vecE^i(\cdot, \by, \vecp)$, then the following reciprocity relation holds
$$
\vecp \cdot \vecE^s(\bx_p,\bx_q,\vecq) = \vecq \cdot \vecE^s(\bx_q,\bx_p,\vecp)
$$
for $\bx_p, \bx_q \in D$ and $\vecp, \vecq \in \R^3$.
\end{lemma}
\begin{proof}
For any $\bx_p, \bx_q \in D$ and $\vecp, \vecq \in \R^3$, let  $\vecE(\cdot,\by,\vecp)=\vecE^s(\cdot,\by,\vecp)+\vecE^i(\cdot,\by,\vecp)$ be the total electric field. We can directly derive that (c.f. the proof of Theorem 2.1 in \cite{ZCS})
\begin{eqnarray*}
&&\hspace{-2.522cm}\vecp \cdot \vecE^s(\bx_p,\bx_q,\vecq) - \vecq \cdot \vecE^s(\bx_q,\bx_p,\vecp) \\
&&\hspace{-2.522cm}=\int_{\partial D} \left(\nu \times \vecE(\by,\bx_q,\vecq) \right) \cdot \mcurl \vecE(\by,\bx_p,\vecp) ds(\by) - \int_{\partial D} \left( \nu \times \vecE(\by,\bx_p,\vecp) \right) \cdot \mcurl \vecE(\by,\bx_q,\vecq) ds(\by).
\end{eqnarray*}
From the continuity of the wave fields across the boundary, we have that
\begin{eqnarray*}
\nu \times \vecE^+ = \nu \times \vecE^- \quad &\mbox{on}& \quad \partial D, \\
\nu \times A \mcurl \vecE^+  = \nu \times \mcurl\vecE^- \quad &\mbox{on}& \quad \partial D,
\end{eqnarray*}
where $\nu \times \vecE^+$ denotes the tangential trace of $\vecE|_{\R^3 \backslash \overline{D}}$ and  $\nu \times \vecE^-$ denotes the tangential trace of $\vecE|_{D}$ (similar definitions hold for $\nu \times A \mcurl \vecE^+$ and $\nu \times \mcurl \vecE^-$).  Now we can derive
\begin{eqnarray*}
&&\hspace{-2.522cm}\vecp \cdot \vecE^s(\bx_p,\bx_q,\vecq) - \vecq \cdot \vecE^s(\bx_q,\bx_p,\vecp) \\
&&\hspace{-2.522cm}=\int_{\partial D} \bigg( \nu \times \vecE^+(\by,\bx_q,\vecq) \cdot A\mcurl \vecE^+(\by,\bx_p,\vecp) - \nu \times \vecE^+(\by,\bx_p,\vecp) \cdot A\mcurl \vecE^+(\by,\bx_q,\vecq) \bigg)ds(\by).
\end{eqnarray*}
{ Note that $\vecE(\by,\bx_q,\vecq)$ and $\vecE(\by,\bx_p,\vecp)$ both satisfy equation \eqref{TotalEqn2}, therefore we have from integration by parts (first for smooth fields and then followed by a density argument) that}
\begin{eqnarray*}
&&\hspace{-2.122cm}\int_{\partial D} \bigg( \nu \times \vecE^+(\by,\bx_q,\vecq) \cdot A\mcurl \vecE^+(\by,\bx_p,\vecp) - \nu \times \vecE^+(\by,\bx_p,\vecp) \cdot A\mcurl \vecE^+(\by,\bx_q,\vecq) \bigg)ds(\by) \\
&&\hspace{-2.122cm} =\int_{\partial B_R} \bigg( \nu \times \vecE(\by,\bx_q,\vecq) \cdot \mcurl \vecE(\by,\bx_p,\vecp) - \nu \times \vecE(\by,\bx_p,\vecp) \cdot \mcurl \vecE(\by,\bx_q,\vecq) \bigg)ds(\by) \\
&&\hspace{-2.122cm} + \int_{B_R \backslash \overline{D}} \bigg( \mcurl A \mcurl \vecE(\by,\bx_p,\vecp) \vecE(\by,\bx_q,\vecq) - \mcurl A \mcurl \vecE(\by,\bx_q,\vecq) \vecE(\by,\bx_p,\vecp) \bigg)d \by \\
&&\hspace{-2.122cm} - \int_{B_R \backslash \overline{D}} \bigg( A \mcurl \vecE(\by,\bx_p,\vecp) \mcurl \vecE(\by,\bx_q,\vecq) - A \mcurl \vecE(\by,\bx_q,\vecq) \mcurl \vecE(\by,\bx_p,\vecp) \bigg)d\by.
\end{eqnarray*}
Note that $A$ and $N$ are symmetric, we have the third integral in the above equation is zero. Note that $\vecE$ satisfies equation (\ref{TotalEqn2}) and $N$ is symmetric, we have that the second integral in the above equation is zero. Note that $\vecE$ satisfies the Silver-M{\"u}ller radiation condition, we have that the first integral in the above equation is zero. Now we can conclude that
$$
\vecp \cdot \vecE^s(x_\vecp,x_\vecq,\vecq) = \vecq \cdot \vecE^s(x_\vecq,x_\vecp,\vecp).
$$
This completes the proof. 
\end{proof}
\subsection{Uniqueness}
To prove the uniqueness result, we { first} show the following lemma.
\begin{lemma}\label{Maxwellintgone1} Assume that $A$ and $N$ satisfy the assumptions in Lemma \ref{MaxwellETEInvertible}.  For $n\in{\mathbb N}$, let $\{(\vecw_n,\,\vecv_n)\}\in \vecH_{loc}(\mcurl, \R^3 \backslash \overline{D}))\times
\vecH_{loc}(\mcurl, \R^3 \backslash \overline{D}))$, be a sequence of solutions to the
exterior transmission problem (\ref{ETEMaxwell1})-(\ref{ETEMaxwell6}) with
boundary data $\vecg_n\in \vecH^{-\frac{1}{2}}(\mdiv, \partial D)$ and $\vech_n\in
\vecH^{-\frac{1}{2}}(\mdiv, \partial D)$. If the sequences $\{\vecg_n\}$ and $\{\vech_n\}$ both
converge in $\vecH^{-\frac{1}{2}}(\mdiv, \partial D)$ , and the sequences $\{(\vecw_n|_\Omega,\mcurl \vecw_n|_\Omega)\}$ and $\{(\vecv_n|_\Omega,\mcurl \vecv_n|_\Omega)\}$ are
bounded in $\vecH^1(\Omega) \times \vecH^1(\Omega)$ for $\Omega=B_R\backslash\overline{D_1}$ and $\Omega=D_1 \backslash\overline{D}$, then there exists a convergent subsequence $\{(\vecw_{n_k}, \vecv_{n_k})\}_{k=1}^\infty$ in $\vecH(\mcurl, B_R\backslash\overline{D}) \times \vecH(\mcurl, B_R\backslash\overline{D})$.
\end{lemma}
\begin{proof}
Let $\left\{(\vecw_n,\vecv_n)\right\}$ be the sequence stated in the lemma. Since the sequences $\{\vecw_n|_\Omega\}$ and $\{\vecv_n|_\Omega\}$ are
bounded in $\vecH^1(\Omega)$ for $\Omega=B_R\backslash\overline{D_1}$ and $\Omega=D_1 \backslash\overline{D}$, then from the compact embedding of
$H^1(\Omega)$ into $L^2(\Omega)$, we can select  convergent subsequences
$\{\vecw_{n_k}\}$ and $\{\vecv_{n_k}\}$ in both $\vecL^2(B_R\backslash\overline{D_1})$ and $\vecL^2(D_1 \backslash\overline{D})$.

Note that $\mdiv(\nu \times \vecw_{n_k}) = -\nu \cdot \mcurl \vecw_{n_k}$, and $\{(\vecw_{n_k}|_\Omega,\mcurl \vecw_{n_k}|_\Omega)\}$ is a bounded sequence in $\vecH^1(\Omega) \times \vecH^1(B_R\backslash\overline{D_1})$, then from the compact embedding of
$H^{\frac{1}{2}}(S_R)$ into $H^{-\frac{1}{2}}(S_R)$, we can select a subsequence, still denote as $\{\vecw_{n_k}\}$ such that $\hat{\bx} \times \vecw_{n_k}|_{S_R}$ converges in $\vecH^{-\frac{1}{2}}(\mdiv, S_R)$. Note that $G_k-G_{i\kappa}$ are bounded from $\vecH^{-\frac{1}{2}}(\mdiv, S_R)$ to $\vecH^{-\frac{1}{2}}(\mdiv, S_R)$, then $\left( G_{k}-G_{i\kappa} \right) (\hat{\bx} \times \vecw_{n_k})$ converges in $\vecH^{-\frac{1}{2}}(\mdiv, S_R)$. Similarly we can have a subsequence (still denoted by) $\left( G_{k}-G_{i\kappa} \right) (\hat{\bx} \times \vecv_{n_k})$ converges in $\vecH^{-\frac{1}{2}}(\mdiv, S_R)$.

Since $(\vecw_{n_k},\,\vecv_{n_k})$ is solution to the
exterior transmission problem (\ref{ETEMaxwell1})-(\ref{ETEMaxwell6}), then $\{\vecw_{n_k}\}$ and
$\{\vecv_{n_k}\}$ satisfy
\begin{eqnarray*}
\mcurl(A\mcurl \vecw_{n_k}) + \kappa^2 N \vecw_{n_k} = (\kappa^2 +k^2) N \vecw_{n_k} \quad &\mbox{in}& \quad B_R \backslash \overline{D}, \\
\mcurl^2 \vecv_{n_k} + \kappa^2  \vecv_{n_k} = (\kappa^2 +k^2)  \vecv_{n_k} \quad &\mbox{in}& \quad B_R \backslash \overline{D},  \\
\nu \times \vecw_{n_k} - \nu \times \vecv_{n_k} = \vecg_{n_k} \quad &\mbox{on}& \quad \partial D, \\
\nu \times A \mcurl \vecw_{n_k}  - \nu \times \mcurl\vecv_{n_k} = \vech_{n_k} \quad &\mbox{on}& \quad \partial D, \\
\hat{\bx} \times  \mcurl \vecw_{n_k}   - G_{i\kappa} (\hat{\bx} \times \vecw_{n_k}) = \left( G_{k}-G_{i\kappa} \right) (\hat{\bx} \times \vecw_{n_k}) \quad &\mbox{on}& \quad S_R, \\
\hat{\bx} \times  \mcurl \vecv_{n_k}   - G_{i\kappa} (\hat{\bx} \times \vecv_{n_k}) =\left( G_{k}-G_{i\kappa} \right) (\hat{\bx} \times \vecw_{n_k}) \quad &\mbox{on}& \quad S_R,
\end{eqnarray*}
where $\kappa>0$ is chosen as in Lemma \ref{MaxwellETEInvertible}. Each term on the right hand sides converges, then from Remark \ref{InhomogeneousCaderon}, we have $\{(\vecw_{n_k}, \vecv_{n_k})\}$ converges in $\vecH(\mcurl, B_R\backslash\overline{D}) \times \vecH(\mcurl, B_R\backslash\overline{D})$. 
\end{proof}

Now we are ready to prove the following uniqueness result.
\begin{theorem} \label{MaxwellUniqueness}
Assume that $A$ and $N$ satisfy the assumptions in Lemma \ref{MaxwellETEInvertible} and $k$ is not a Maxwell eigenvalue for $C$. Then the cavity is uniquely determined by the measurements $\vecE^s(\bx,\by,\vecp)$ for all $\bx \in \Sigma$, $\by \in \Sigma$ and $\vecp \in \R^3$.
\end{theorem}

        \begin{figure}[hb!]
        \centering
\includegraphics[width=0.522\linewidth]{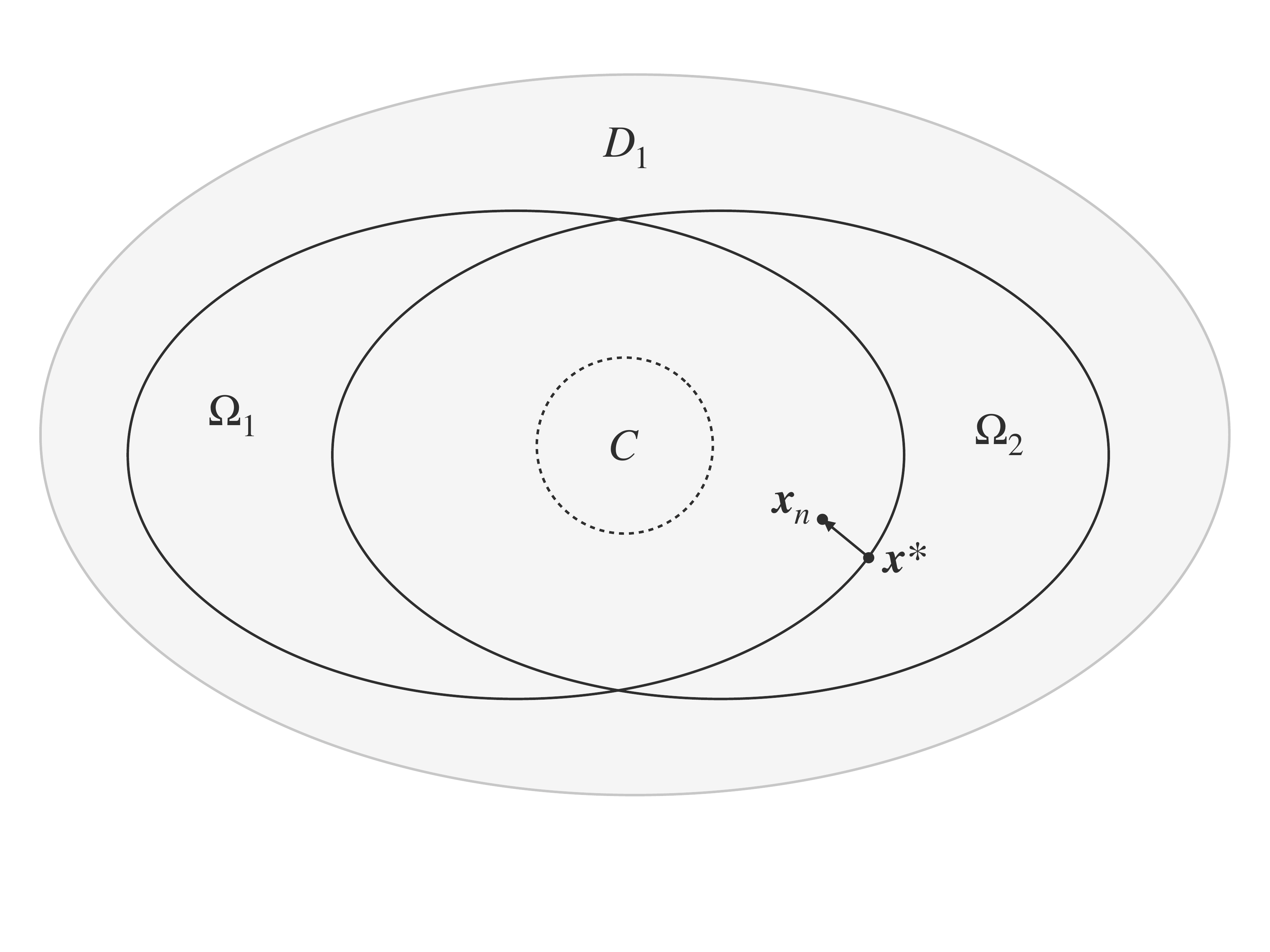}
 \vspace{-1\baselineskip}
     \caption{
     \linespread{1}
     An example of the geometry for uniqueness.
     } \label{TwoCavities}
    \end{figure}
\begin{proof}
To begin with,  let us assume that $D_1 \backslash \overline{\Omega}_1$ and $D_1 \backslash \overline{\Omega}_2$ are two inhomogeneities with material properties characterized by $A_j$ and $N_j$, respectively. $\vecE_j^s (j=1,2)$ are the corresponding scattered electric field. Moreover, we always have that the measurement ball $C$ belongs to  $\Omega:=\Omega_1 \cap \Omega_2$.
See Fig. \ref{TwoCavities} for a configuration of the problem. Now if $\vecE_1^s(\bx,\by,\vecp)=\vecE_2^s(\bx,\by,\vecp)$ for all $\bx, \by \in \Sigma:= \partial C$ and $\vecp \in \R^3$, we are going to show that $\Omega_1 = \Omega_2$.

We prove the theorem by a contradiction argument. Assume on the contrary that $\Omega_1 \not = \Omega_2$. 

\textit{Step 1:} We first show that $\vecE_1^s(\bx,\by,\vecp)=\vecE_2^s(\bx,\by,\vecp)$  for all $\bx, \by \in \Omega$ and $\vecp \in \R^3$. Note that for a fixed $\by \in \Sigma$,
\begin{eqnarray} 
\mcurl^2 \vecE_j^s(\cdot,\by,\vecp)- k^2 \vecE_j^s(\cdot,\by,\vecp) = 0 &\quad &\mbox{in} \quad C, \label{MaxwellUniqueness eqn1}\\
\nu \times  \vecE_1^s(\cdot,\by,\vecp) =  \nu \times \vecE_2^s(\cdot,\by,\vecp) &\quad &\mbox{on} \quad \Sigma. \label{MaxwellUniqueness eqn2}
\end{eqnarray}
Since $k$ is not a Maxwell eigenvalue for $C$, we have that
\begin{equation} \label{MaxwellUniqueness eqn3}
\vecE_1^s(\cdot,\by,\vecp)=\vecE_2^s(\cdot,\by,\vecp) \quad \mbox{in} \quad C. 
\end{equation}
From the interior regularity result in Lemma \ref{MaxwellInteriorRegularity} we have that
\begin{equation} \label{MaxwellUniqueness eqn4}
\vecE_1^s(\cdot,\by,\vecp)=\vecE_2^s(\cdot,\by,\vecp) \quad \mbox{in} \quad \Omega.
\end{equation}
From the reciprocity relation in Lemma \ref{ReciprocityMaxwell}, we have that for any $\by \in \Omega$ and $\bx\in \Sigma$
$$
\vecq \cdot \vecE_1^s(\bx,\by,\vecp) = \vecp \cdot \vecE_1^s(\by,\bx,\vecq) =  \vecp \cdot \vecE_2^s(\by,\bx,\vecq) =\vecq \cdot \vecE_2^s(\bx,\by,\vecp).
$$
This is valid for any $\vecq \in \R^3$. Therefore we have that
$$
\vecE_1^s(\bx,\by,\vecp)=\vecE_2^s(\bx,\by,\vecp) \quad \mbox{for} \quad \bx \in \Sigma \quad \mbox{and} \quad \by \in \Omega.
$$
Repeating the argument above {(i.e. from equations \eqref{MaxwellUniqueness eqn1}--\eqref{MaxwellUniqueness eqn2} where $\by \in \Omega$, we can derive that equations \eqref{MaxwellUniqueness eqn3}--\eqref{MaxwellUniqueness eqn4} hold for all $\by \in \Omega$)}, we can obtain that $\vecE_1^s(\bx,\by,\vecp)=\vecE_2^s(\bx,\by,\vecp)$  for all $\bx, \by \in \Omega$ and $\vecp \in \R^3$.

\textit{Step 2:} Now without loss of generality we can pick $\bx^* \in \partial \Omega_1$ such that $\bx^* \in \Omega_2$. Let $\bx_n:= \bx^*-\frac{\epsilon}{n}\nu \in \Omega$, where $\epsilon$ is small enough and $\nu$ is the unit outward normal to $\Omega_1$. It is seen that $\bx_n \to \bx^*$ as $n \to \infty$. Let $\vecE_{j,n}^s(\cdot,\bx_n,\vecp)$ be the scattered electric fields corresponding to the inhomogeneity $D_1 \backslash \overline{\Omega_j}$ and electric dipole $\vecE^i(\cdot, \bx_n, \vecp)$.  Define $\vecv_{n}$ and $\vecu_{j,n}$ by
$$
\vecv_{n}:= \frac{\vecG(\cdot,\bx_n,\vecp)}{\|\vecG(\cdot,\bx_n,\vecp)\|_{\vecH^1(\mcurl, D_1 \backslash \overline{\Omega}_1)}}, \quad \vecu_{j,n}:= \frac{\vecE_j^s(\cdot,\bx_n,\vecp)}{\|\vecG(\cdot,\bx_n,\vecp)\|_{\vecH^1(\mcurl,D_1 \backslash \overline{\Omega}_1)}},
$$
where $\|\vecG(\cdot,\bx_n,\vecp)\|^2_{\vecH^1(\mcurl, D_1 \backslash \overline{\Omega}_1)}= \|\vecG(\cdot,\bx_n,\vecp)\|^2_{\vecH^1( D_1 \backslash \overline{\Omega}_1)} + \|\mcurl \vecG(\cdot,\bx_n,\vecp)\|^2_{\vecH^1( D_1 \backslash \overline{\Omega}_1)}$.
Since $\bx^*$ is an interior point in $\Omega_2$, we have that
\begin{equation} \label{uniq_vto0}
\|\vecv_{n}\|_{\vecH^2(B_R \backslash \overline{\Omega}_2)} \to 0 ~~\mbox{as}~~ n \to \infty.
\end{equation}
Note that $\vecu_{2,n}$ is the unique solution to the forward problem (\ref{Scattered})-(\ref{SilverMuller}) depending continuously on data $\vecv_n$, then from Lemma \ref{ForwardMaxwell}, Lemma \ref{MaxwellTransmissionRegularity} and Remark \ref{MaxwellTransmissionRegularityRemark}
\begin{equation} \label{uniq_u2to0}
\|\vecu_{2,n}\|_{\vecH^1(\Omega)} \to 0 \quad \mbox{and} \quad \|\mcurl \vecu_{2,n}\|_{\vecH^1(\Omega)} \to 0 ~~\mbox{as}~~ n \to \infty.
\end{equation}
We have shown above that $\vecE_1^s(\bx,\by,\vecp)=\vecE_2^s(\bx,\by,\vecp)$  for all $\bx \in \Omega$, $\by \in \Omega$ and $\vecp \in \R^3$, then
\begin{equation} \label{uniq_u1to0}
\|\vecu_{1,n}\|_{\vecH^1(\Omega)} \to 0 \quad \mbox{and} \quad \|\mcurl \vecu_{1,n}\|_{\vecH^1(\Omega)} \to 0 ~~\mbox{as}~~ n \to \infty.
\end{equation}

Now, with the help of a cutoff function $\chi\in C_0^{\infty}(\R^3)$ with small enough compact support near $\bx^*$ and note that $
\mcurl (\chi \vecu_{1,n}  ) = \chi \mcurl \vecu_{1,n} + \nabla \chi \times \vecu_{1,n}
$, we can have from equation \eqref{uniq_u1to0} that
\begin{equation}\label{Maxwellgone1}
\|\nu \times \chi \vecu_{1,n}\|_{\vecH^{-\frac{1}{2}}(\mdiv,\partial \Omega_1)}\to 0  \quad \mbox{as} \quad n\to \infty.
\end{equation}
Similarly we can have from equation \eqref{uniq_u1to0} that
\begin{equation}\label{Maxwellgone2}
\|\nu \times \chi \mcurl \vecu_{1,n}\|_{\vecH^{-\frac{1}{2}}(\mdiv,\partial \Omega_1)} \to 0 \quad \mbox{as} \quad n\to \infty.
\end{equation}

\textit{Step 3:}
We next note that in the exterior of a small enough neighborhood $\Omega_{\epsilon_1}$ of $\bx^*$, $\vecv_n$ and $\mcurl \vecv_n$ are uniformly bounded in $\vecH^2(B_R \backslash \overline{\Omega}_{\epsilon_1})$. Then for a small enough neighborhood $\Omega_{\epsilon_2}$ of $\bx^*$, we have from Lemma \ref{MaxwellTransmissionRegularity} and Remark \ref{MaxwellTransmissionRegularityRemark} that $\vecu_{1,n} \in \vecH_{loc}^1(\Omega_1 \backslash \overline{\Omega}_{\epsilon_2})$, $A\mcurl \vecu_{1,n} \in \vecH_{loc}^1(\Omega_1 \backslash \overline{\Omega}_{\epsilon_2})$ and $\mcurl\big(A\mcurl \vecu_{1,n}\big) \in \vecH_{loc}^1(\Omega_1 \backslash \overline{\Omega}_{\epsilon_2})$. Since $A$ has $C^1(\overline{D}_1 \backslash D)$ entries, then $\mcurl \vecu_{1,n} \in \vecH_{loc}^1(\Omega_1 \backslash \overline{\Omega}_{\epsilon_2})$ and $\mcurl^2 \vecu_{1,n} \in \vecH_{loc}^1(\Omega_1 \backslash \overline{\Omega}_{\epsilon_2})$. Recall that $\chi$ has compact support near $\bx^*$, we have $(1-\chi)\vecu_{1,n} \in \vecH^1(\Omega_1)$ and $(1-\chi)\mcurl  \vecu_{1,n}  \in \vecH^1(\Omega_1)$.

Since
$
\mcurl \left((1-\chi)\vecu_{1,n} \right) = (1-\chi) \mcurl \vecu_{1,n} + \nabla (1-\chi) \times \vecu_{1,n} \in \vecH^1(\Omega_1)
$, then from the compact embedding of $H^1(\Omega_1)$ to $L^2(\Omega_1)$, there exists a $\vecH(\mcurl, \Omega_1)$ convergent subsequence $\{(1-\chi)\vecu_{1,n_k}\}$, still denoted (abusively) by $\{(1-\chi)\vecu_{1,n}\}$ . Similarly there exists a $\vecH(\mcurl, \Omega_1)$ convergent subsequence $\{(1-\chi)\mcurl \vecu_{1,n_k}\}$, still denoted (abusively) by $\{(1-\chi)\mcurl \vecu_{1,n}\}$. Therefore we immediately have that
\begin{equation}\label{Maxwellgone11}
\|\nu \times (1-\chi) \vecu_{1,n}\|_{\vecH^{-\frac{1}{2}}(\mdiv,\partial \Omega_1)}   \quad \mbox{converges as} \quad n\to \infty,
\end{equation}
\begin{equation}\label{Maxwellgone22}
\|\nu \times (1-\chi) \mcurl  \vecu_{1,n}\|_{\vecH^{-\frac{1}{2}}(\mdiv,\partial \Omega_1)} \quad \mbox{converges as} \quad n\to \infty.
\end{equation}

Finally from equations (\ref{Maxwellgone1})-(\ref{Maxwellgone22}) we have that
\begin{eqnarray}
\|\nu \times  \vecu_{1,n}\|_{\vecH^{-\frac{1}{2}}(\mdiv,\partial \Omega_1)}   \quad &\mbox{converges as}& \quad n\to \infty, \label{uniq_step3_1}\\
\|\nu \times \mcurl  \vecu_{1,n}\|_{\vecH^{-\frac{1}{2}}(\mdiv,\partial \Omega_1)} \quad &\mbox{converges as}& \quad n\to \infty. \label{uniq_step3_2}
\end{eqnarray}

\textit{Step 4:} It is first seen that the total wave field and incident wave field pair $(\vecu_{1,n}+\vecv_n, \vecv_n)$ satisfies equations  (\ref{ETEMaxwell1})-(\ref{ETEMaxwell6}) with $D=\Omega_1$, $\vecf_1=0$, $\vecf_2=0$, $\vecg = \nu \times \vecu_{1,n}$, and $\vech=\nu \times \mcurl \vecu_{1,n}$. Note that $\|\vecv_n\|_{\vecH^1(\mcurl, D_1 \backslash \overline{\Omega}_1)}=1$ and $\bx_n \in \Omega_1$, then $\vecv_n$ and $\mcurl \vecv_n$ are uniformly bounded in $\vecH^1(B_R \backslash \overline{\Omega_1})$. Moreover $\vecu_{1,n}$ is the unique solution to the forward problem (\ref{Scattered})-(\ref{SilverMuller}) depending continuously on data $\vecv_n$, then from Lemma \ref{MaxwellTransmissionRegularity} and Remark \ref{MaxwellTransmissionRegularityRemark}, $\vecu_{1,n}$ and $\mcurl \vecu_{1,n}$ are uniformly bounded in $\vecH^1(B_R \backslash \overline{D}_1)$ and $\vecH^1(D_1 \backslash \overline{\Omega}_1)$.

Together with \eqref{uniq_step3_1} -- \eqref{uniq_step3_2}, we have that all the assumptions in Lemma \ref{Maxwellintgone1} are satisfied, hence there exists a subsequence, still denoted (abusively) by $\vecv_n$, that converges in $\vecH(\mcurl, B_R \backslash \overline{\Omega}_1)$. Assume $\vecv_n \to \vecv^*$. Then
$$
\mcurl^2 \vecv^* -k^2 \vecv^*=0 \quad \mbox{in} \quad B_R \backslash \overline{\Omega}_1.
$$
Moreover note that $\vecv_n \to 0$ in any open set away from $\bx^*$, then we have $\vecv^*=0$ in some open set in $B_R \backslash \overline{\Omega}_1$. By interior regularity we have that $\vecv^*=0$ in $B_R \backslash \overline{\Omega_1}$. This contradicts $\vecv_n \to \vecv^*$ since $\|\vecv_n\|_{\vecH^1(\mcurl, D_1 \backslash \overline{\Omega}_1)}=1$. This completes the proof. 
\end{proof}
\section{The linear sampling method} \label{LSM}
The second goal of this paper is to design a robust imaging algorithm. Here we are particularly interested in the so-called qualitative methods. For a more detailed introduction to qualitative methods we refer to the book \cite{CaCo, CK}. As a first step towards qualitative methods for our inverse interior electromagnetic scattering problems for a penetrable cavity, we investigate the linear sampling method. It is our further interest to consider { the factorization method and related imaging methods}.

In this section we develop the linear sampling method. Given the near field data $\vecE^s(\bx,\by,\vecp)$ for all $\bx, \by \in \Sigma$ and $\vecp \in \R^3$, we can define $\mathcal{N}: \vecL_t^2(\Sigma) \to \vecL_t^2(\Sigma)$ by
\begin{equation} \label{near field operator}
(\mathcal{N} \vecg)(\bx) := \int_\Sigma \nu(\bx) \times \vecE^s(\bx,\by,\vecg(\by)) ds(\by).
\end{equation}

We also define the electric single layer potential $\vecS : \vecH^{-\frac{1}{2}}(\mdiv,\Sigma) \to \vecH_{loc}^1(\mcurl, \mathbb{R}^3 \backslash \Sigma)$ by
$$
(\vecS \vecg)(\bx):=\int_{\Sigma} \vecG (\bx, \by) \vecg(\by) ds(\by).
$$
Denoted by $\vecH$ the function space of
$$
\vecH := \mbox{closure}\bigg\{ \vecu \in \vecH^1(\mcurl, D_1 \backslash \overline{D}): \nonumber\\
\vecu = ( \vecS \vecg )|_{D_1 \backslash \overline{D}}\, , \vecg \in \vecH^{-\frac{1}{2}}(\mdiv,\partial D) \bigg\}, \label{lsm function space H}
$$
and then define the solution operator $\vecK: \vecH \to \vecH_{loc}(\mcurl, \R^3)$ by
$$
\vecK \vecf:= \vecE^s
$$
with $\vecE^s$ and $\vecf$ (here $\vecf$ is a general function in $\vecH^1(\mcurl, D_1 \backslash \overline{D})$) satisfying equations (\ref{Scattered})-(\ref{SilverMuller}). A superposition argument yields
$$
\mathcal{\vecN} \vecg = \big(\nu \times \vecK \vecS \vecg \big)|_\Sigma.
$$
It can also be written as $\mathcal{\vecN} = \vecB \vecS$ by introducing the operator $\vecB: \vecH \to \vecL_t^2(\Sigma)$ defined by $\vecB \vecf := (\nu \times \vecK \vecf) |_{\Sigma}$.

In the remaining of the paper, we make the following assumptions.
\begin{assum} \label{ETEMaxwellDiscreteness}
Assume that $k$ is not an exterior transmission eigenvalue, and the corresponding exterior transmission problem (\ref{ETEMaxwell1})-(\ref{ETEMaxwell6}) has a unique solution depending continuously on the data.
\end{assum}

Similar to the interior transmission problem \cite{CaCo,CaHa2}, we expect that the exterior transmission eigenvalues form a discrete set and there exists infinitely many such eigenvalues. We expect that the exterior transmission eigenvalues provide information about the surrounding medium and such eigenvalues can be used in non-destructive material testing.  It is our future interest to investigate the exterior transmission problem. 
\begin{assum} \label{MaxwellEigenAssum}
Assume that $k$ is not a Maxwell eigenvalue for $C$.
\end{assum}

\subsection{A denseness result}
To develop the linear sampling method, we need the following denseness result.
\begin{lemma} \label{MaxwellDenseness}
Assume that $\vecu \in \vecH_{loc}(\mcurl, \R^3 \backslash \overline{D})$ satisfies
\begin{eqnarray*}
\mcurl^2 \vecu- k^2 \vecu = 0 &\quad &\mbox{in} \quad \R^3 \backslash \overline{D}, \\
\lim\limits_{|\bx|\to \infty}\left(\mcurl \vecu\times \bx-ik|\bx|\vecu\right)=0.&&
\end{eqnarray*}
Then the restrictions of $\vecu$ on $D_1 \backslash \overline{D}$ belongs to $\vecH$, i.e. $\vecu|_{D_1 \backslash \overline{D}} \in \vecH$.
\end{lemma}
\begin{proof}
For the given $\vecu \in \vecH_{loc}(\mcurl, \R^3 \backslash \overline{D})$, we have that there exists $\vecp_n \in \vecH^{-\frac{1}{2}}(\mdiv, \Sigma)$ such that (c.f. Theorem 3.4 in \cite{ZCS})
\begin{equation} \label{proof MaxwellDenseness 1}
\|\nu \times \vecu_n\|_{\vecH^{-\frac{1}{2}}(\mdiv,\partial D)} \to \|\nu \times \vecu\|_{\vecH^{-\frac{1}{2}}(\mdiv,\partial D)},
\end{equation}
where $\vecu_n \in \vecH$ is defined by
$$
\vecu_n = \int_{\Sigma} \vecG (\cdot, \by) \vecp_n(\by) ds(\by).
$$
Note that $\by \in C \subset D$, then we have that $\vecu_n$ satisfies the same equation as $\vecu$, i.e.
\begin{eqnarray*}
\mcurl^2 \vecu_n- k^2 \vecu_n = 0 &\quad &\mbox{in} \quad \R^3 \backslash \overline{D}, \\
\lim\limits_{|\bx|\to \infty}\left(\mcurl \vecu_n\times \bx-ik|\bx|\vecu_n\right)=0.&&
\end{eqnarray*}
Equation \eqref{proof MaxwellDenseness 1} states that the tangential trace of $\vecu_n$ converges to the tangential trace of $\vecu$ on $\partial D$, therefore the well-posedness of the Maxwell equation directly yields
$$
\|\vecu_n\|_{\vecH(\mcurl, D_1 \backslash \overline{D})} \to \|\vecu\|_{\vecH(\mcurl, D_1 \backslash \overline{D})}.
$$
This completes the proof. 
\end{proof}

\subsection{Characterization of the cavity by $\vecB$}
Now we give a characterization of the boundary $\partial D$.
\begin{theorem} \label{RangeBMaxwell}
Assume that Assumption \ref{ETEMaxwellDiscreteness} and \ref{MaxwellEigenAssum} hold. Then for any $\bz \in \R^3 \backslash \overline{C}$ and polarization $\vech \in \R^3$,
$$
\bz \in \R^3 \backslash \overline{D} \quad \mbox{iff} \quad \big(\nu \times  \vecG (\cdot,\bz) \vech\big)|_\Sigma \in \mbox{Range}(\vecB).
$$
\end{theorem}
\begin{proof}
We first show that if $\bz \in D\backslash \overline{C}$, then $(\nu \times  \vecG (\cdot,\bz) \vech)|_\Sigma \not\in \mbox{Range}(\vecB)$. Assume on the contrary that there exists $\vecu \in \vecH$ such that
$$
\vecB \vecu = \big(\nu \times  \vecG (\cdot,\bz) \vech\big)|_\Sigma.
$$
Let $\vecw = \vecK \vecu$, then
$$
\mcurl^2 \vecw - k^2 \vecw = 0 \quad \mbox{in} \quad C
$$
and
$\nu \times \vecw = \vecB \vecu= \big(\nu \times  \vecG (\cdot,\bz) \vech\big)|_\Sigma$ on  $\Sigma$.
Note that $\bz \not\in \overline{C}$, then
\begin{eqnarray*}
\mcurl^2 \big( \vecG (\cdot,\bz) \vech \big) - k^2 \vecG (\cdot,\bz) \vech = 0 &\quad &\mbox{in} \quad C.
\end{eqnarray*}
Now $\vecw$ and $\vecG (\cdot,\bz) \vech$ satisfy the same equation with the same boundary data, then the fact that $k$ is not a Maxwell eigenvalue for $C$ yields
$$
\vecw = \vecG (\cdot,\bz) \vech \quad \mbox{in} \quad C.
$$
From the interior regularity in Lemma \ref{MaxwellInteriorRegularity} we have that
$$
\vecw = \vecG (\cdot,\bz) \vech \quad \mbox{in} \quad D \backslash \{\bz\}.
$$
This contradicts  the fact that $\vecw$ is analytic and bounded in the interior of $D$.

We now show that if $\bz \in \R^3 \backslash \overline{D}$, then $\big(\nu \times  \vecG (\cdot,\bz) \vech\big)|_\Sigma \in \mbox{Range}(\vecB)$. From the well-posedness of the exterior transmission problem in Assumption \ref{ETEMaxwellDiscreteness}, we have that there exists a unique solution $(\vecw, \vecv)$ such that equations (\ref{ETEMaxwell1})-(\ref{ETEMaxwell6}) hold for $\vecf_{1,2}=0$, $\vecg = \nu \times \big( \vecG (\cdot,\bz) \vech\big)$ and $\vech = \nu \times \big(\mcurl \vecG (\cdot,\bz) \vech\big)$. Let us define
$$
\vecu:=\Big\{
\begin{array}{ccc}
 \vecw -\vecv &  \mbox{in}  & \R^3 \backslash \overline{D}  \\
 \vecG (\cdot,\bz) \vech &  \mbox{in} &  D
\end{array}.
$$
Then we have that  $\vecu \in \vecH_{loc}(\mcurl, \R^3)$ and
$$
\mcurl(A\mcurl \vecu) - k^2 N \vecu = \mcurl \left( \left(I-A\right) \mcurl \vecv\right) - k^2\left(I- N\right) \vecv  \quad \mbox{in} \quad \R^3.
$$
This yields
$$
\vecB \vecv = \nu \times \vecK \vecv = \nu \times \vecu = \nu \times  \vecG (\cdot,\bz) \vech \quad \mbox{on} \quad \Sigma,
$$
provided that $\vecv \in \vecH$. Indeed that $(\vecw, \vecv)$ satisfies equations (\ref{ETEMaxwell1})-(\ref{ETEMaxwell6}), then
\begin{eqnarray*}
\mcurl^2 \vecv- k^2 \vecv = 0 &\quad &\mbox{in} \quad \R^3 \backslash \overline{D} \\
\lim\limits_{|\bx|\to \infty}\left(\mcurl \vecv \times \bx-ik|\bx|\vecv\right)=0,&&
\end{eqnarray*}
from this and Lemma \ref{MaxwellDenseness}, we have that $\vecv \in \vecH$. Hence $\big(\nu \times  \vecG (\cdot,\bz) \vech\big)|_\Sigma \in \mbox{Range}(\vecB)$ and this completes the proof. 
\end{proof}

\subsection{Properties of $\vecB$ and $\vecS$}
From the definition of the function space $\vecH$ and the operator $\vecS$, it directly follows that $\vecS$ has dense range in $\vecH$. Now we prove the following lemma.
\begin{lemma} \label{DenseRangeBMaxwell}
Assume that $k$ is not a exterior transmission eigenvalue, then $\vecB$ is compact and has dense range in $\vecL_t^2(\Sigma)$.
\end{lemma}
\begin{proof}
Assume $\vecu = \vecK \vecv$, then $\vecB \vecv= (\nu \times \vecu)|_\Sigma$. From the definition of the solution operator $\vecK$ we have that
$$
\mcurl^2 \vecu - k^2 \vecu=0 \quad \mbox{in} \quad D.
$$
Note that $\Sigma$ lies in the interior of $D$, then the interior regularity in Lemma \ref{MaxwellInteriorRegularity}  yields that $\nu \times \vecu \in \left(H^1(\Sigma\right))^3$ and hence $\vecB$ is compact.

To show that $\vecB$ has dense range it is sufficient to show that $\mathcal{\vecN}=\vecB \vecS$ has dense range, i.e. $\mathcal{\vecN}^*$ is injective. Note that (Theorem 3.1 in \cite{ZCS})
$$
\mathcal{\vecN}^* \vecg= \overline{\int_\Sigma \nu(\cdot) \times \vecE^s(\cdot,\by,\overline{\vecg(\by)})ds(\by)} = \overline{\mathcal{\vecN} \overline{\vecg}}.
$$
It is sufficient to show that $\mathcal{\vecN}$ is injective. Now we assume $\mathcal{\vecN} \vecg=0$, then $\vecu=\vecK \vecS \vecg$ satisfies
$$
\nu \times \vecu =0 \quad \mbox{on} \quad \Sigma.
$$
Note that
$$
\mcurl^2 \vecu - k^2 \vecu=0 \quad \mbox{in} \quad C
$$
and $k^2$ is not a Maxwell eigenvalue, we have that $\vecu=0$ in $C$. Let
$$
\vecv = \int_\Sigma \vecG (\cdot,\by)\vecg(\by) ds(\by).
$$
Note that $\vecu$ and $\vecv$ satisfies equations (\ref{Scattered})-(\ref{SilverMuller}), then let $\vecw=\vecu+\vecv$ in $\R^3 \backslash \overline{D}$ we have that $(\vecw,\vecv)$ is a solution to the homogeneous exterior transmission problem (\ref{ETEMaxwell1})-(\ref{ETEMaxwell2}). However $k$ is not a transmission eigenvalue, then $\vecv=0$ in $\R^3 \backslash \overline{D}$. The interior regularity in Lemma \ref{MaxwellInteriorRegularity}  yields $\vecv=0$ in $\R^3 \backslash \overline{C}$. Then $\nu \times \vecv=0$ on $\Sigma$. From the mapping properties of the singer layer potential (c.f. Theorem 3.4 in \cite{ZCS}), we have that $\nu \times \vecv=0$ on $\Sigma$ implies $\vecg=0$ on $\Sigma$. This completes the proof. 
\end{proof}

\subsection{Main result on the linear sampling method}
Now we are ready to prove the main theorem.
\begin{theorem} \label{MaxwellLSM}
For any $\bz \in \R^3 \backslash \overline{C}$ and polarization $\vech \in \R^3$, we have that
\begin{enumerate}
\item If $\bz \in \R^3 \backslash \overline{D}$, then for any $\epsilon>0$, there exists $\vecg_{\bz}^\epsilon \in \vecL_t^2(\Sigma)$ such that
$$
\| \mathcal{\vecN}  \vecg_{\bz}^\epsilon - \nu \times \vecG (\cdot,\bz) \vech\|_{\vecL_t^2(\Sigma)} < \epsilon,
$$
and $\|\vecS   \vecg_{\bz}^\epsilon\|_{\vecH(\mcurl,D_1 \backslash \overline{D})}$ remains bounded as $\epsilon \to 0$.
\item If $\bz \in D \backslash \overline{C}$, then for any $\epsilon>0$, there exits $\vecg_{\bz}^\epsilon \in \vecL_t^2(\Sigma)$ such that
$$
\| \mathcal{\vecN}  \vecg_{\bz}^\epsilon - \nu \times \vecG (\cdot,\bz) \vech\|_{\vecL_t^2(\Sigma)} < \epsilon,
$$
and such $\vecg_{\bz}^\epsilon \in \vecL_t^2(\Sigma)$ satisfies
$$
\lim\limits_{\epsilon \to 0} \|\vecS   \vecg_{\bz}^\epsilon\|_{\vecH(\mcurl,D_1 \backslash \overline{D})} =\infty.
$$
\end{enumerate}
\end{theorem}
\begin{proof}
{\textit{(i)}}\quad Assume that $\bz \in \R^3 \backslash \overline{D}$, from Theorem \ref{RangeBMaxwell} we have that
$$
 \big(\nu \times \vecG (\cdot,\bz) \vech \big)|_\Sigma \in \mbox{Range}(\vecB).
$$
Therefore there exists $\vecu_{\bz} \in \vecH$ such that $\vecB \vecu_{\bz}= \big(\nu \times \vecG (\cdot,\bz) \vech\big)|_\Sigma$. As $\vecS$ has dense range in $\vecH$, there exists $\vecg_{\bz}^\epsilon$ such that
$$
\|\vecS   \vecg_{\bz}^\epsilon - \vecu_{\bz}\|_{\vecH(\mcurl,D_1 \backslash \overline{D})} \le \frac{1}{\|\vecB\|} \epsilon.
$$
This immediately yields that $\|\vecS   \vecg_{\bz}^\epsilon\|_{\vecH(\mcurl,D_1 \backslash \overline{D})}$ remains bounded as $\epsilon \to 0$.
Applying $\vecB$ to the above inequality yields
$$
\|\vecB\vecS   \vecg_{\bz}^\epsilon - \vecB\vecu_{\bz}\|_{\vecL_t^2(\Sigma)} < \epsilon, \quad i.e., \quad  \|\mathcal{N}   \vecg_{\bz}^\epsilon -\nu \times \vecG (\cdot,\bz) \vech\|_{\vecL_t^2(\Sigma)} < \epsilon.
$$
{\textit{(ii)}}\quad Assume $\bz \in D \backslash \overline{C}$, from Theorem \ref{RangeBMaxwell} we have that
$$
 \nu \times \vecG (\cdot,\bz) \vech \not\in \mbox{Range}(\vecB).
$$
However from Lemma \ref{DenseRangeBMaxwell} we have that $\vecB$ has dense range in $\vecL_t^2(\Sigma)$. Pick $\vecu_z^\epsilon \in \vecH$ such that
$$
\|\vecB \vecu_{\bz}^\epsilon- \nu \times \vecG (\cdot,\bz) \vech\|_{\vecL_t^2(\Sigma)} < \frac{\epsilon}{2}.
$$
As $\vecS$ has dense range in $\vecH$, there exists $\vecg_{\bz}^\epsilon$ such that
$$
\|\vecS   \vecg_{\bz}^\epsilon - \vecu_{\bz}\|_{\vecH(\mcurl,D_1 \backslash \overline{D})} \le \frac{1}{2\|\vecB\|} \epsilon.
$$
Then
\begin{eqnarray*}
&&\|\vecN   \vecg_{\bz}^\epsilon -\nu \times \vecG (\cdot,\bz) \vech\|_{\vecL_t^2(\Sigma)} \\
&\le&\|\vecB\vecS   \vecg_{\bz}^\epsilon - \vecB\vecu_{\bz}^\epsilon\|_{\vecL_t^2(\Sigma)} + \|\vecB \vecu_{\bz}^\epsilon -\nu \times \vecG (\cdot,\bz) \vech\|_{\vecL_t^2(\Sigma)} \\
&<& \epsilon.
\end{eqnarray*}

We then show that
$$
\lim\limits_{\epsilon \to 0} \|\vecS   \vecg_{\bz}^\epsilon\|_{\vecH(\mcurl,D_1 \backslash \overline{D})} =\infty.
$$
Assume on the contrary that $ \|\vecS   \vecg_{\bz}^\epsilon\|_{\vecH(\mcurl,D_1 \backslash \overline{D})}  \le c$ as $\epsilon \to 0$ for some constant $c$. Then there exists a subsequence $\vecg_{\bz}^{\epsilon_n}$ such that $\vecS   \vecg_{\bz}^{\epsilon_n}$ converges weakly to $\vecu^*$ in $\vecH$ as $n \to \infty$ (${\epsilon_n} \to 0$). From  Lemma \ref{DenseRangeBMaxwell} we have that $\vecB$ is compact, then
$$
\lim\limits_{n \to \infty} \vecB \vecS   \vecg_{\bz}^{\epsilon_n} = \vecB \vecu^*.
$$
Note that we have shown
\begin{eqnarray*}
\|\vecN   \vecg_{\bz}^\epsilon -\nu \times \vecG (\cdot,\bz) \vech\|_{\vecL_t^2(\Sigma)}< \epsilon.
\end{eqnarray*}
Then
$$
\lim\limits_{n \to \infty} \vecB \vecS   \vecg_{\bz,n}^\epsilon = \vecN \vecg_{\bz,n}^\epsilon = \nu \times \vecG (\cdot,\bz) \vech.
$$
This implies that $ \vecB \vecu^* = \big(\nu \times \vecG (\cdot,\bz) \vech \big)|_\Sigma$, which contradicts $\big(\nu \times \vecG (\cdot,\bz) \vech \big)|_\Sigma \not\in \mbox{Range}(\vecB)$. This completes the proof. 
\end{proof}

The above theorem implies that $\|\vecS   \vecg_{\bz}^\epsilon\|_{\vecH(\mcurl,D_1 \backslash \overline{D})}$ have different properties for $\bz \in \R^3 \backslash \overline{D}$ and $\bz \in D \backslash \overline{C}$. In practice we use $\frac{1}{\|\vecg_{\bz}^\epsilon\|_{\vecL^2(\Sigma)}}$ as the imaging function  to reconstruct the cavity $D$, as demonstrated in our numerical example Section \ref{Numerics}. { We remark that a complete theoretical characterization of the cavity can be expected using the factorization method \cite{kirsch2008factorization} and generalized linear sampling method \cite{audibert2014generalized,audibert2015qualitative}.}
\begin{figure}[hb!]
\includegraphics[width=0.48\linewidth]{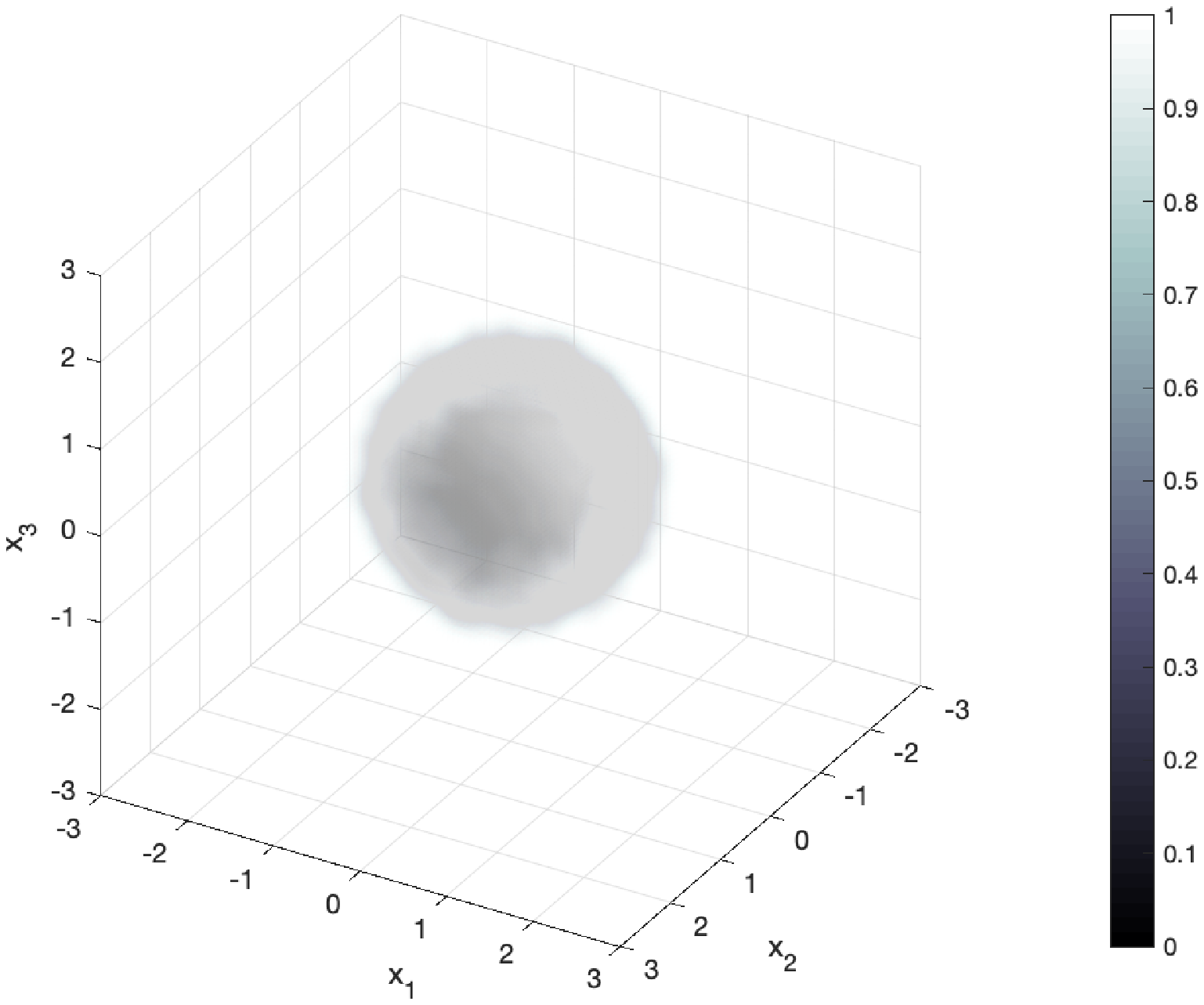}
\includegraphics[width=0.48\linewidth]{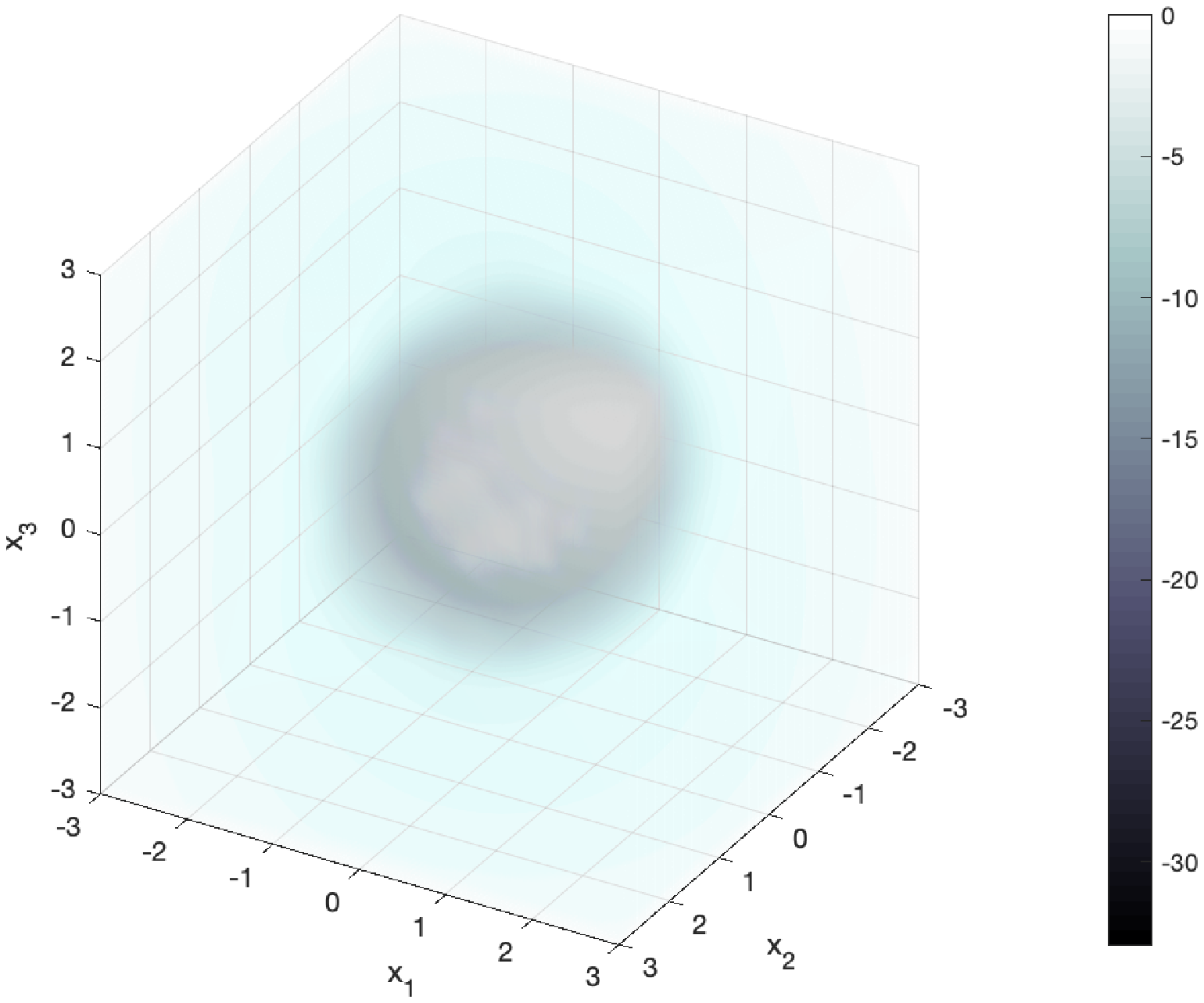}
     \caption{
     \linespread{1}
     Three dimensional image of a ball cavity. Left: exact geometry. Right: reconstructed geometry.
     } \label{ball}
    \end{figure}

    \begin{figure}[hb!]
\includegraphics[width=0.32\linewidth]{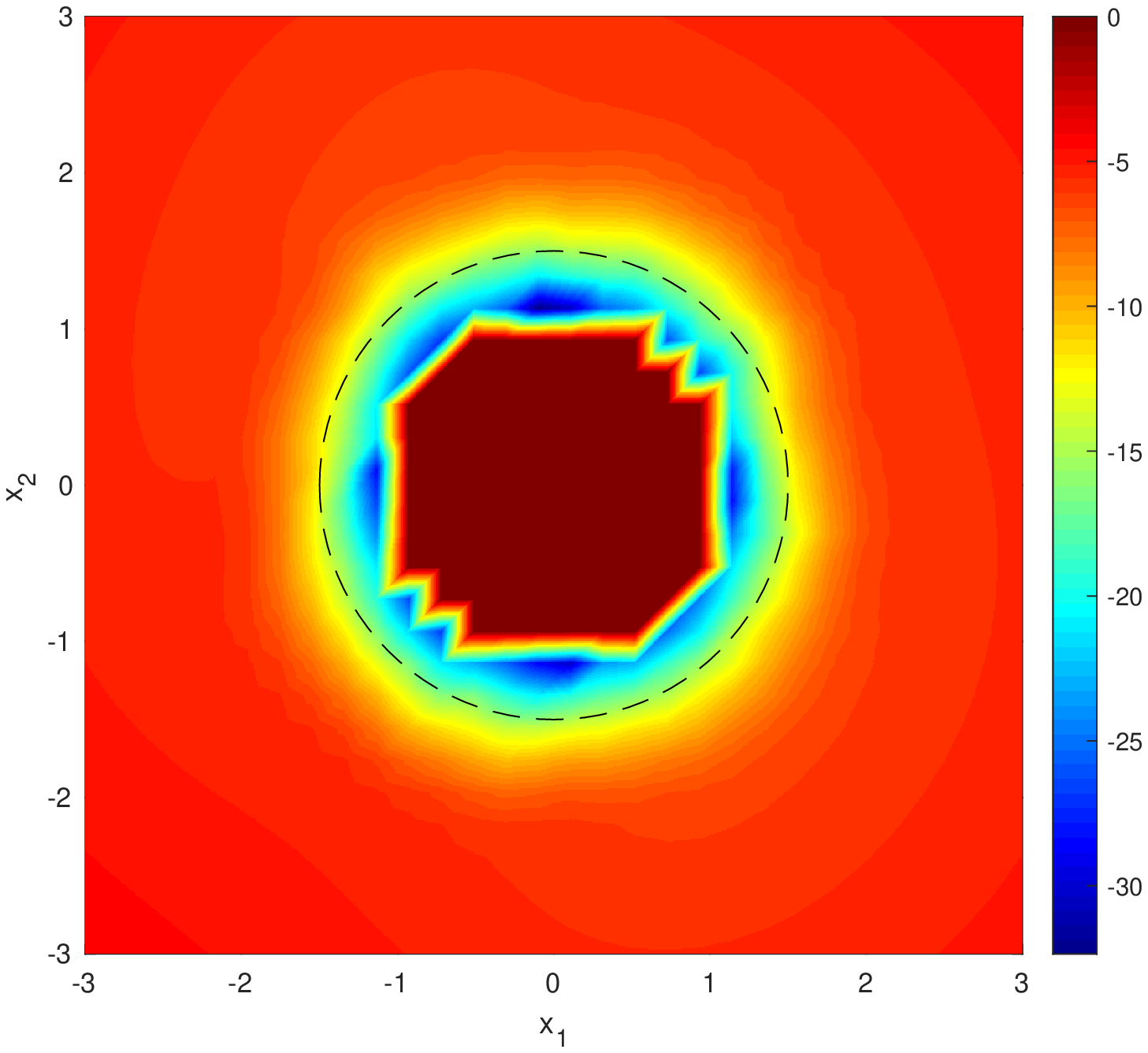}
\includegraphics[width=0.32\linewidth]{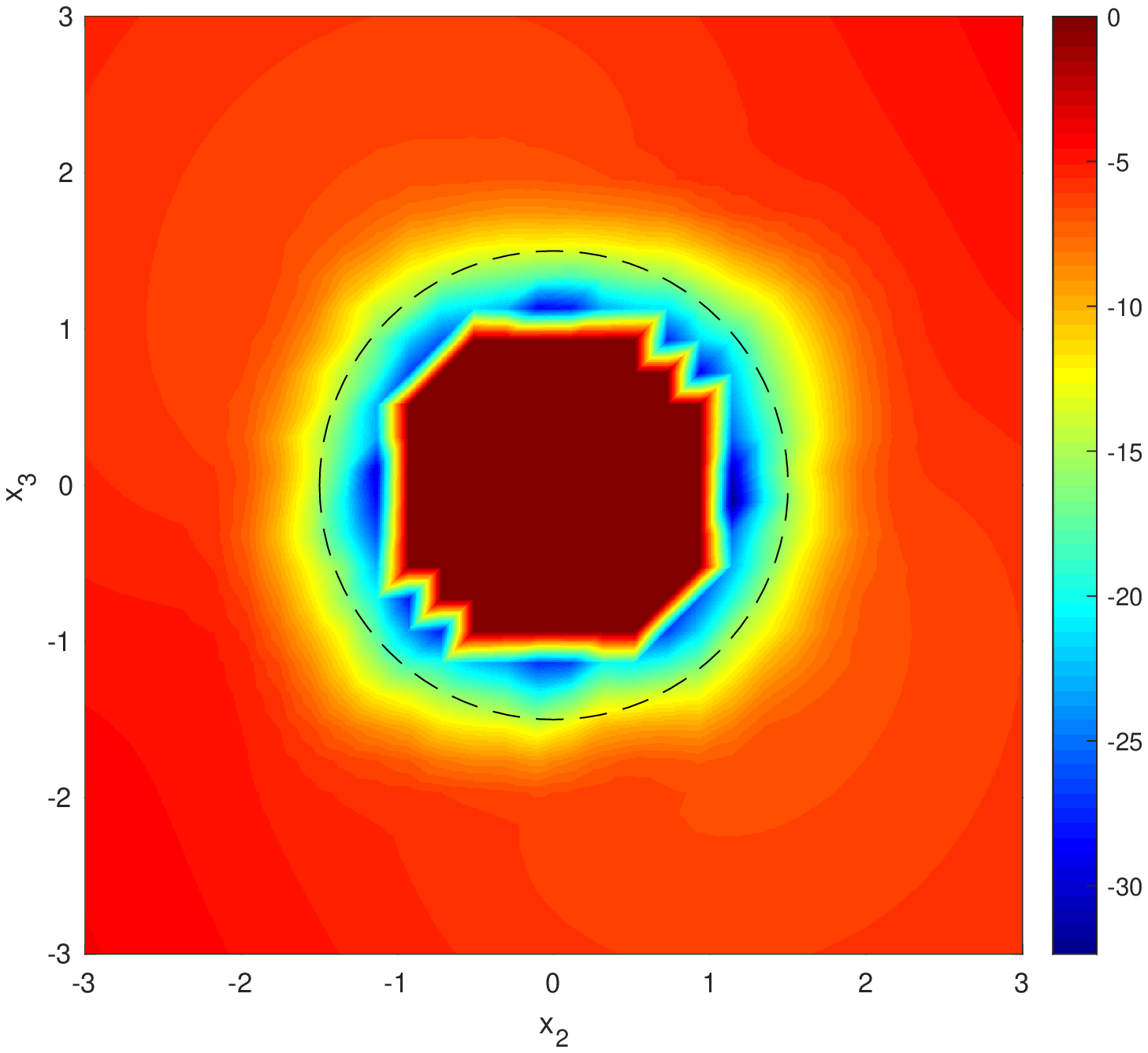}
\includegraphics[width=0.32\linewidth]{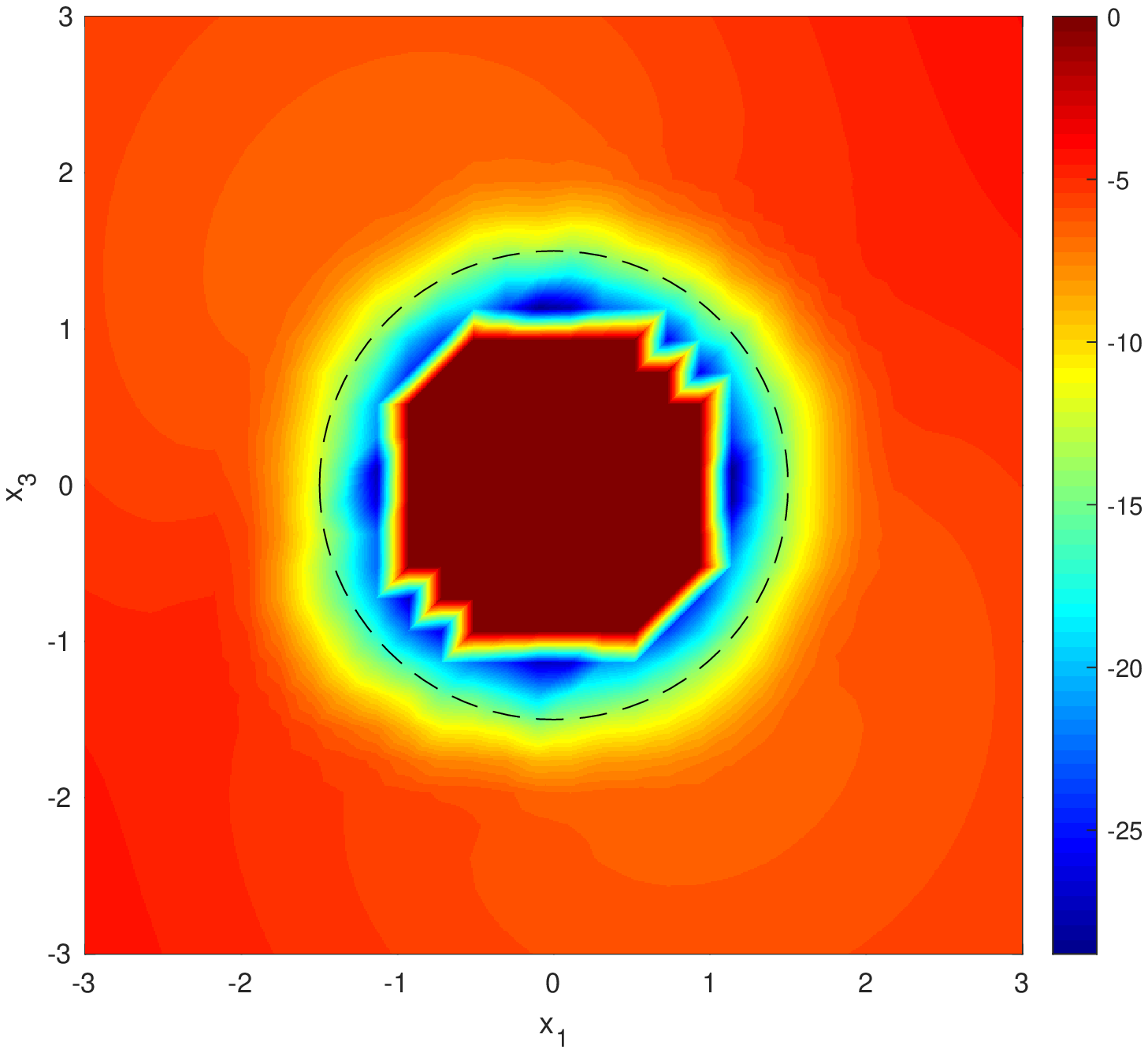}
     \caption{
     \linespread{1}
     Cross-section images of the ball cavity.
     } \label{ball_cs}
    \end{figure}
\section{Numerical Examples}\label{Numerics}
In this section, we provide several numerical examples to illustrate the viability of the linear sample method.

    \begin{figure}[ht!]
\includegraphics[width=0.48\linewidth]{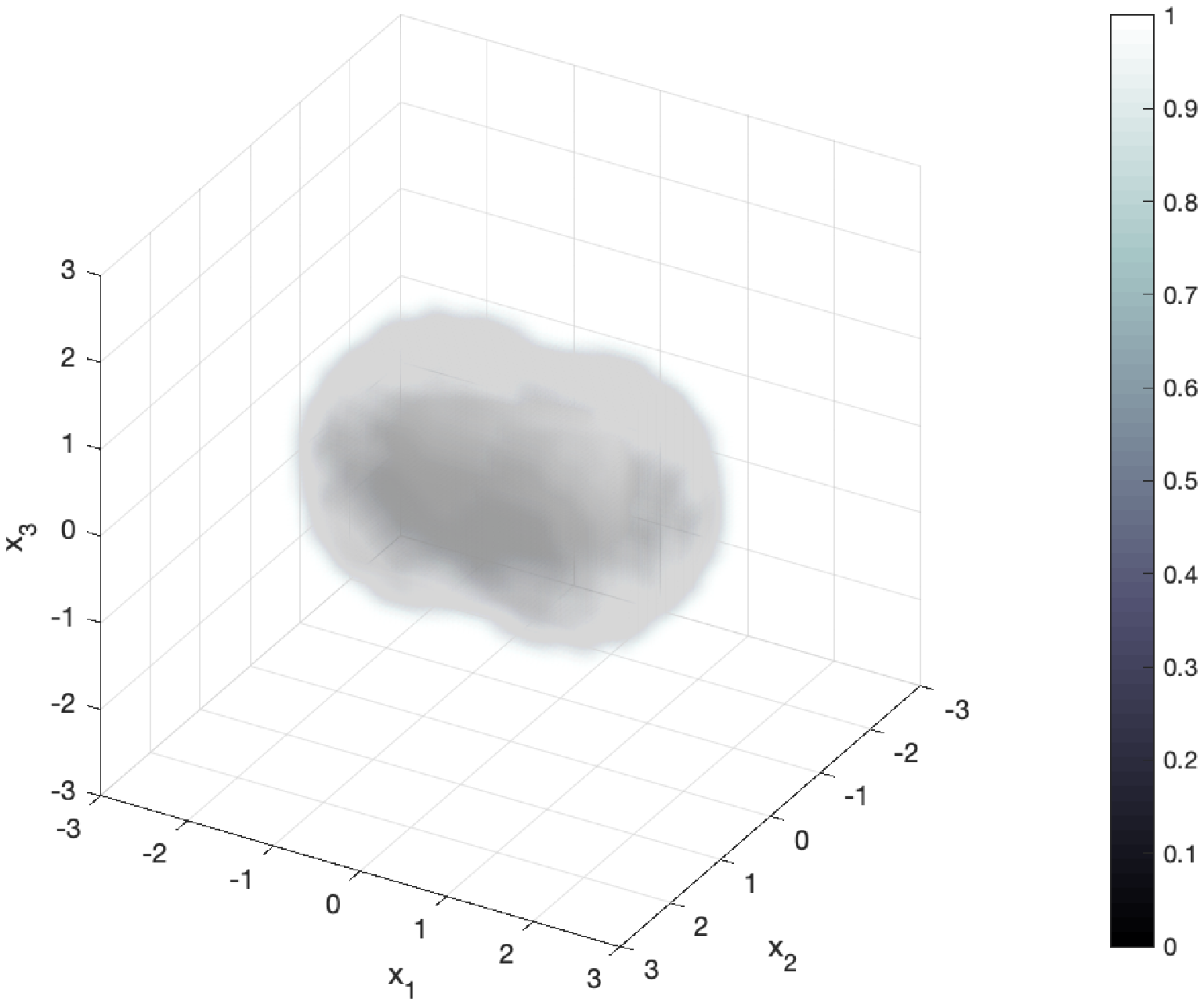}
\includegraphics[width=0.48\linewidth]{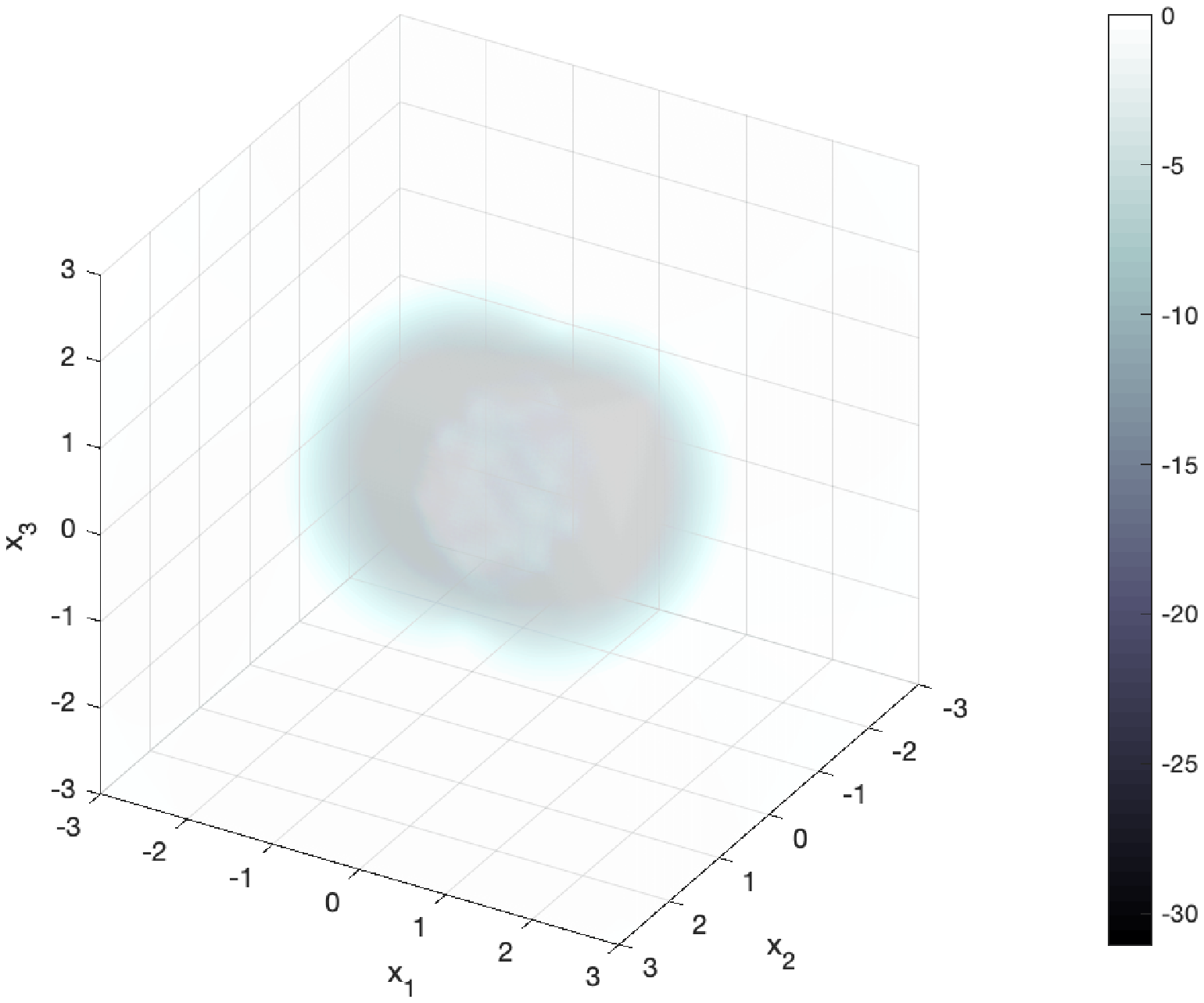}
     \caption{
     \linespread{1}
     Three dimensional image of a peanut-shape cavity. Left: exact geometry. Right: reconstructed geometry.
     } \label{peanut}
    \end{figure}
        \begin{figure}[ht!]
\includegraphics[width=0.32\linewidth]{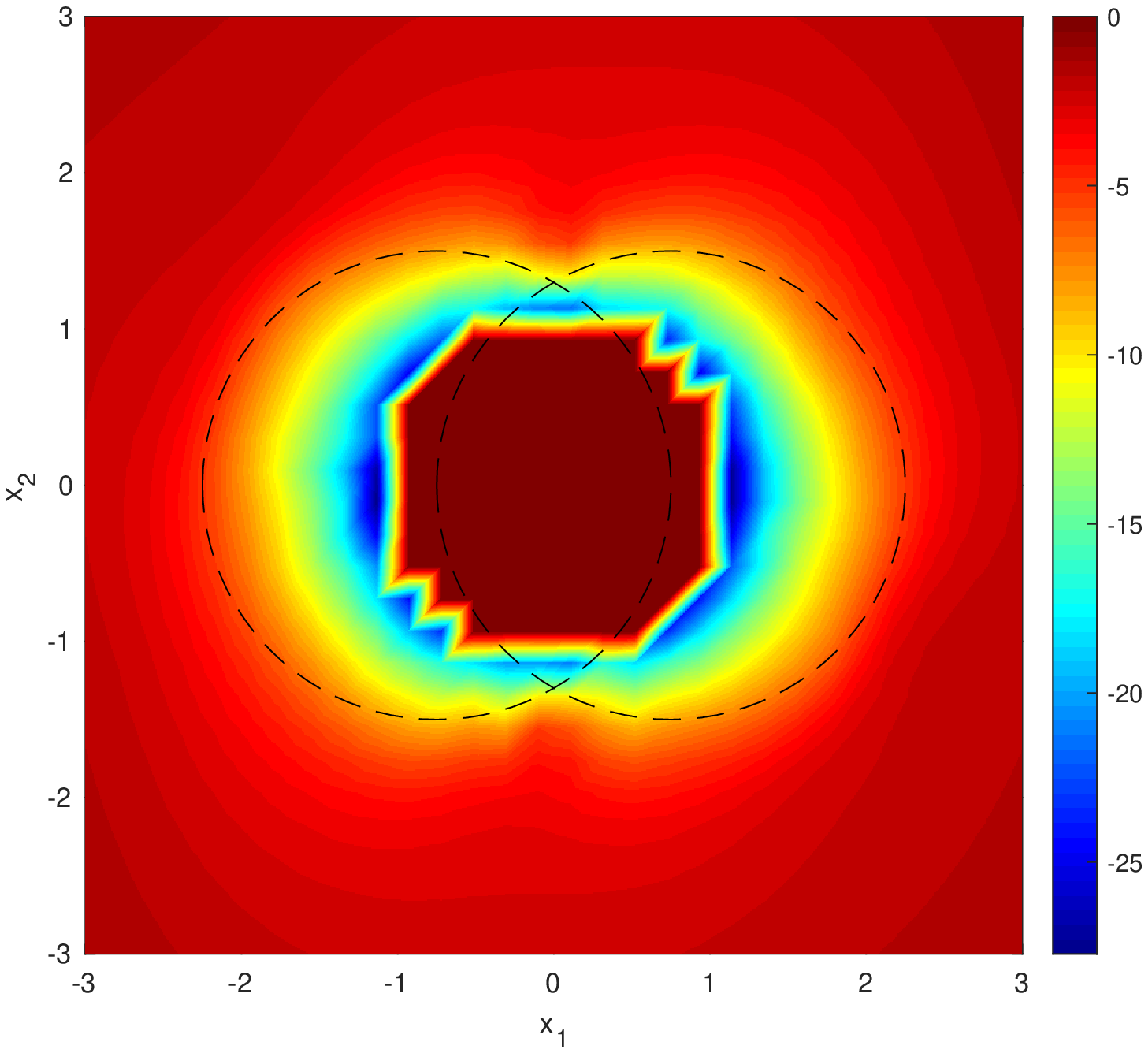}
\includegraphics[width=0.32\linewidth]{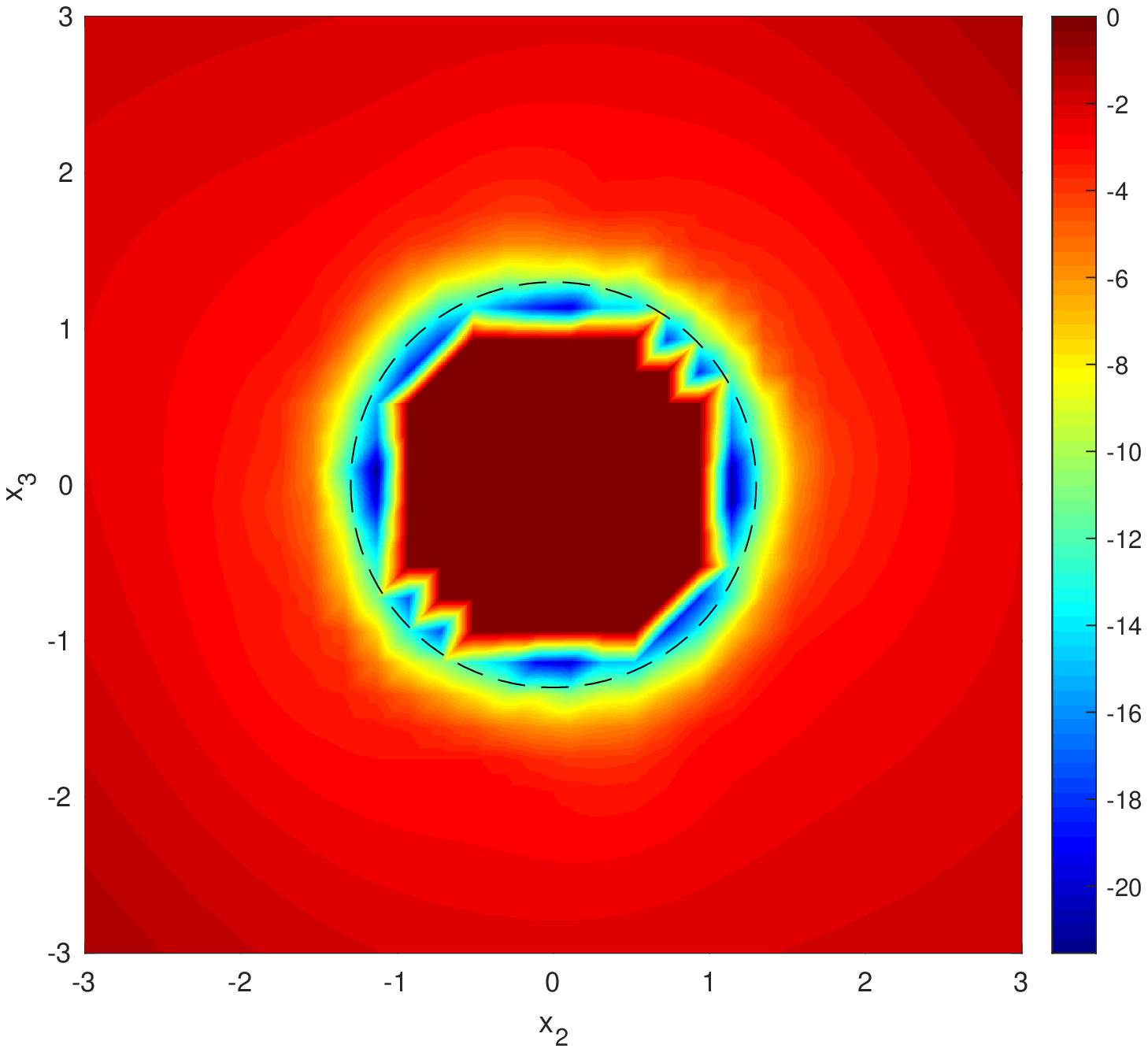}
\includegraphics[width=0.32\linewidth]{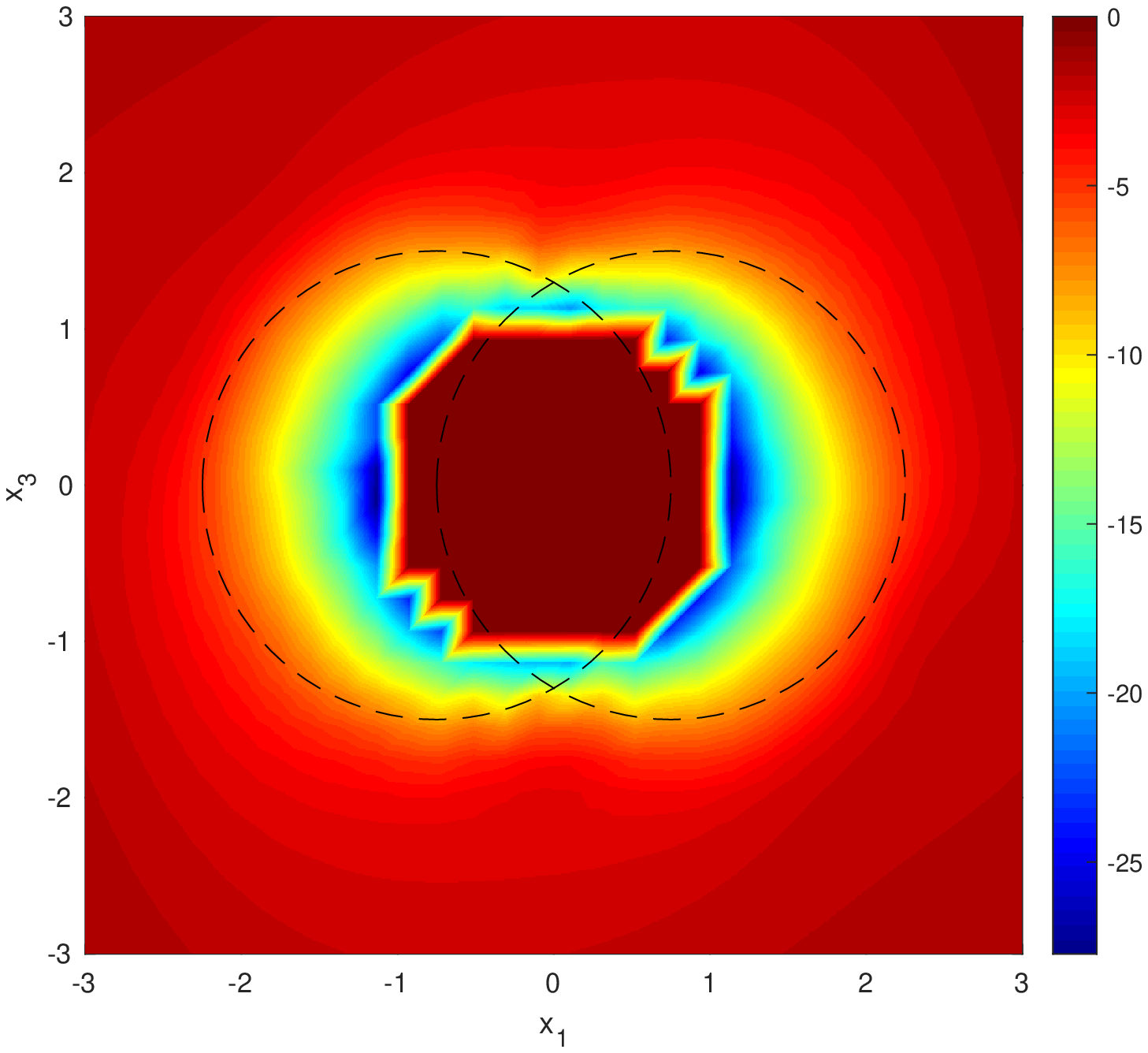}
     \caption{
     \linespread{1}
     Cross-section images of the peanut-shape cavity.
     } \label{peanut_cs}
    \end{figure}

To begin with, let $(r,\theta,\phi)$ denote the spherical coordinate such that $\bx=(x_1,x_2,x_3)$ can be represented by
\begin{eqnarray*}
\Bigg\{
\begin{array}{ccc}
x_1  & =  & r \sin(\theta) \cos(\phi)  \\
x_2  & = &  r \sin(\theta) \sin(\phi) \\
x_3  & = &  r \cos(\theta)
\end{array}, \quad \mbox{where} \quad r \ge0, ~\theta \in [0,\pi], ~\phi \in [0,2\pi].
\end{eqnarray*}
We further denote the local orthogonal unit vectors on the sphere in the directions of increasing $(r,\theta,\phi)$ as $(\hat{\bx}, \hat{\be}_1,\hat{\be}_2)$.

We chose $\Sigma$ as the unit sphere. To discretize the near field equation
\begin{equation} \label{numerics near field eqn}
\mathcal{\vecN}  \vecg_{\bz}^\epsilon = \nu \times \vecG (\cdot,\bz) \vech \mbox{ on } \Sigma,
\end{equation}
 we chose a set of spherical grid points $\{\theta_j,\phi_j\}_{j=1}^n$ on the sphere. At each grid point $\bx_j$ with spherical coordinate $(1,\theta_j,\phi_j)$, let the corresponding local orthogonal unit vector be $(\hat{\bx}(\bx_j), \hat{\be}_1(\bx_j),\hat{\be}_2(\bx_j) )$. Now equation \eqref{numerics near field eqn} can be discretized as
\begin{equation} \label{lsm_discre}
\sum_{j=1}^n w_j \hat{\bx}(\bx_i) \times \vecE^s(\bx_i,\by_j,\vecg(\by_j))  = \hat{\bx}(\bx_i) \times \vecG (\bx_i,\bz) \vech, \quad \forall i=1,\cdots,n,
\end{equation}
where $w_j$ is the quadrature weight at $\by_j$, $\bx_i$ is given by spherical coordinate $(1,\theta_i,\phi_i)$, and $\by_j$ is given by spherical coordinate $(1,\theta_j,\phi_j)$. Note that $\hat{\be}_1(\bx_j)$ and $\hat{\be}_2(\bx_j)$ are the local basis at $\bx_j$, equation \eqref{lsm_discre} is then equivalent to
$$
\sum_{j=1}^n w_j \hat{\be}_\ell(\bx_i) \cdot \vecE^s(\bx_i,\by_j,\vecg(\by_j))  = \hat{\be}_\ell(\bx_i)\cdot \vecG (\bx_i,\bz) \vech, \quad \forall i=1,\cdots,n, ~ \ell=1,2.
$$
We further apply the reciprocity relation in Lemma \ref{ReciprocityMaxwell} to obtain
\begin{equation} \label{lsm_discre1}
\sum_{j=1}^n w_j \vecg(\by_j) \cdot \vecE^s(\bx_i,\by_j,\hat{\be}_\ell(\bx_i))  = \hat{\be}_\ell(\bx_i)\cdot \vecG (\bx_i,\bz) \vech, \quad \forall i=1,\cdots,n, ~  \ell=1,2.
\end{equation}
Let $\vecg(\by_j) = g_1(\by_j) \hat{\be}_1(\by_j)+ g_2(\by_j) \hat{\be}_2(y_j)$, then  \eqref{lsm_discre1} can be written as
\begin{eqnarray} \label{lsm_discre2}
\sum_{m=1}^2\sum_{j=1}^n w_j g_m(\by_j) \hat{\be}_m(\by_j) \cdot \vecE^s(\bx_i,\by_j,\hat{\be}_\ell(\bx_i))  = \hat{\be}_\ell(\bx_i)\cdot \vecG (\bx_i,\bz) \vech,
\end{eqnarray}
for all $i=1,\cdots,n$, and $\ell=1,2$. Solve \eqref{lsm_discre2} for $(g_1(\by_j),g_2(\by_j))_{j=1}^n$, then we can approximate $\vecg$.
    \begin{figure}[hb!]
\includegraphics[width=0.48\linewidth]{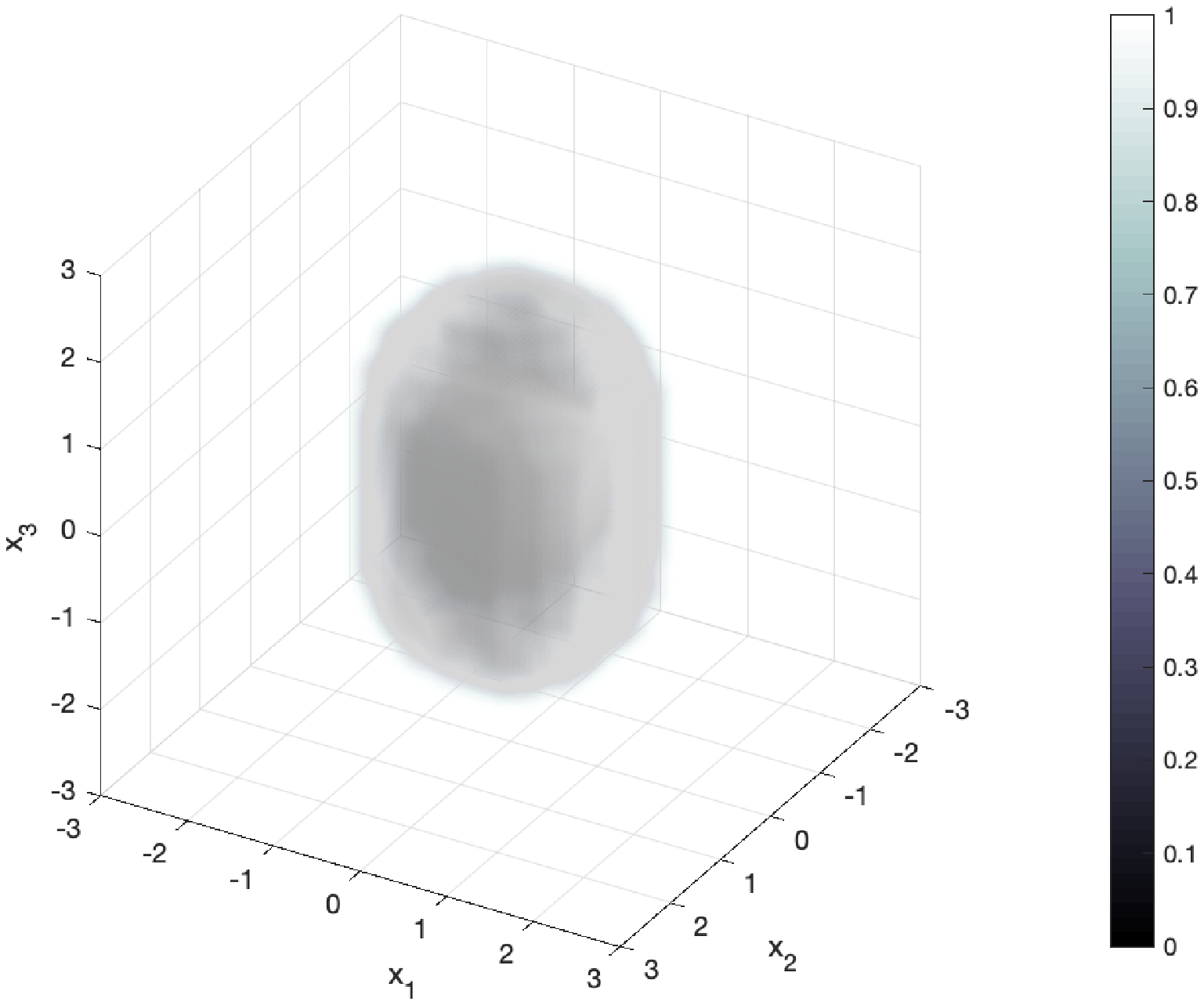}
\includegraphics[width=0.48\linewidth]{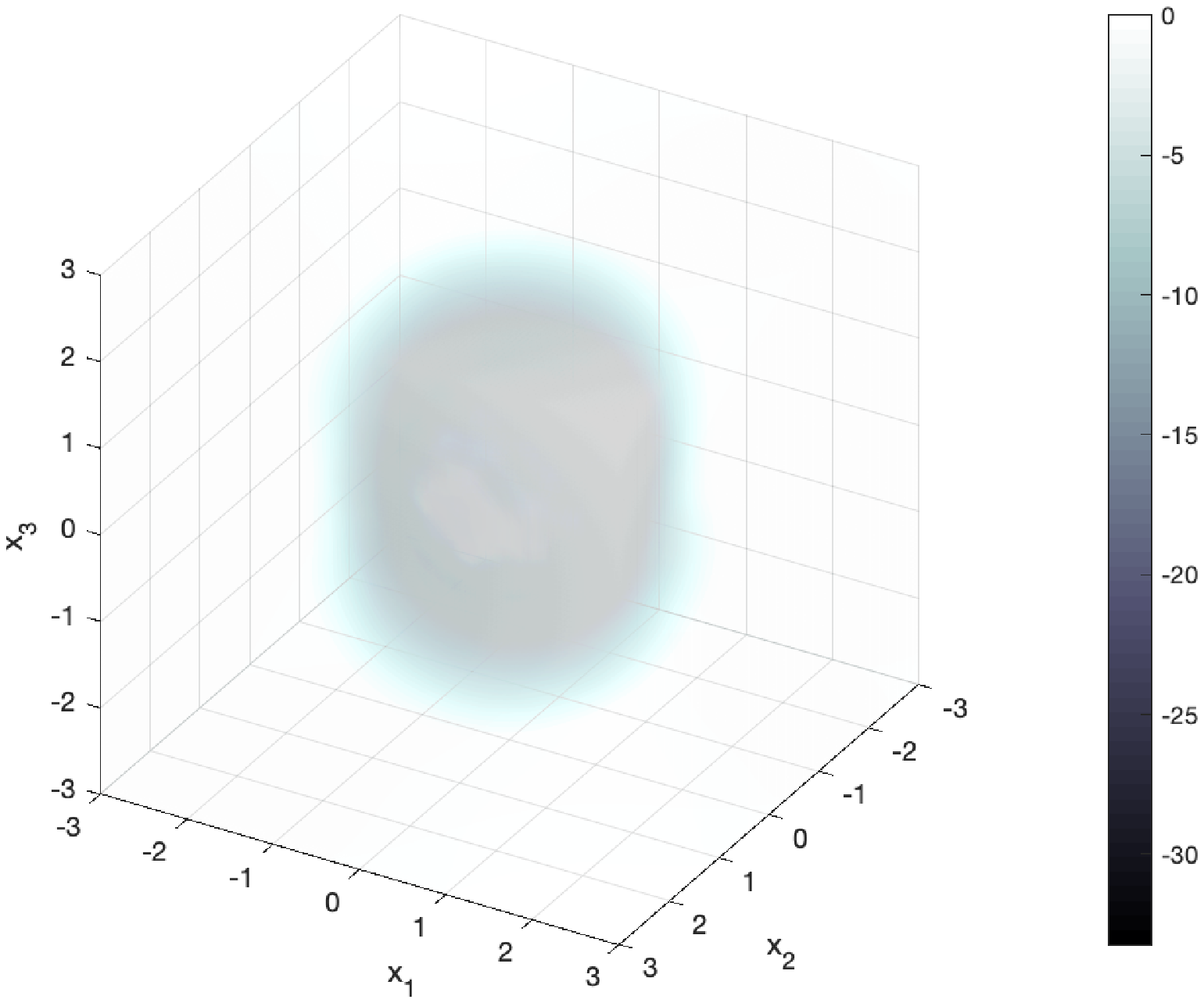}
     \caption{
     \linespread{1}
     Three dimensional image of a cylinder  cavity. Left: exact geometry. Right: reconstructed geometry.
     } \label{cyl}
    \end{figure}
        \begin{figure}[hb!]
\includegraphics[width=0.32\linewidth]{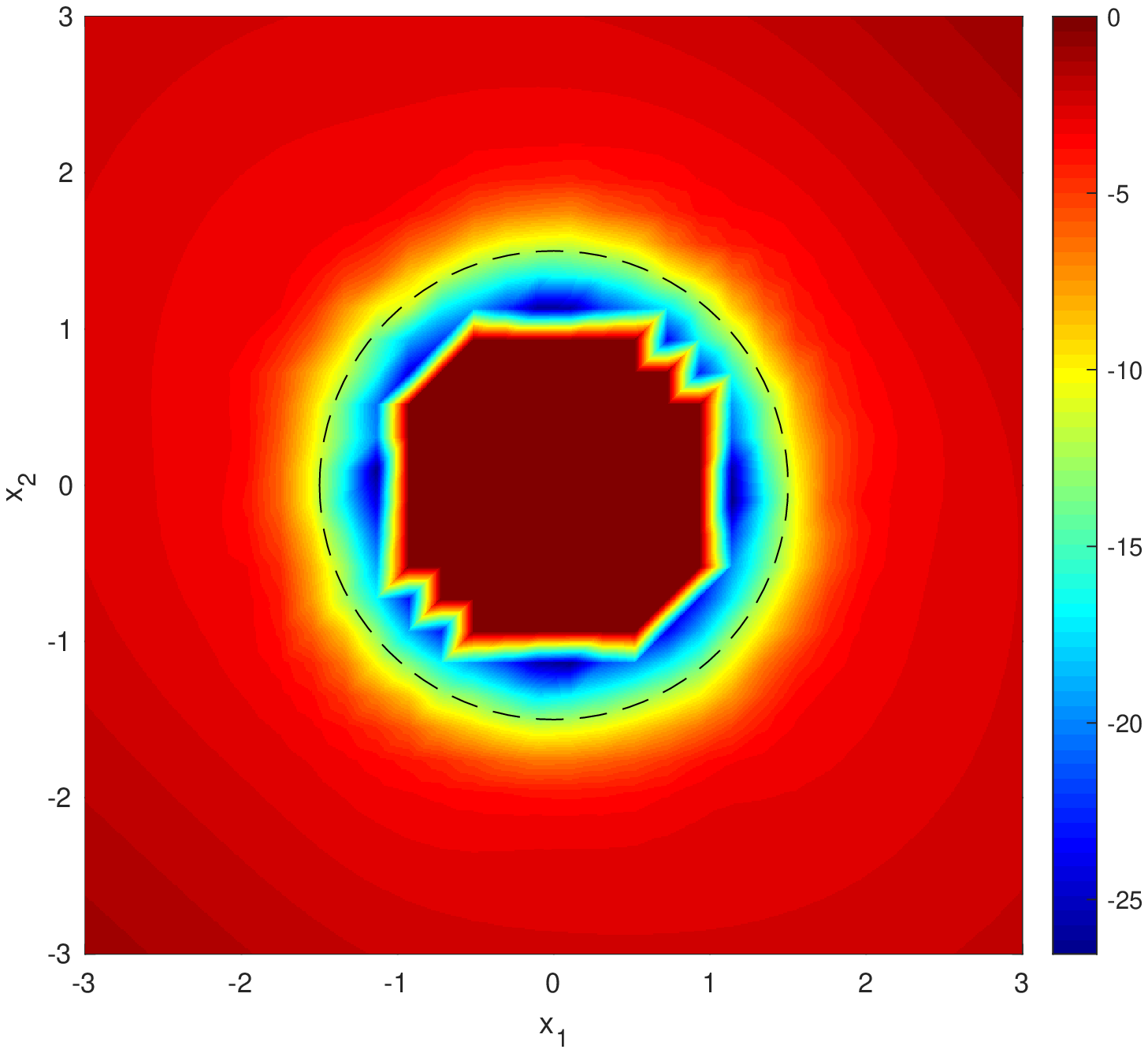}
\includegraphics[width=0.32\linewidth]{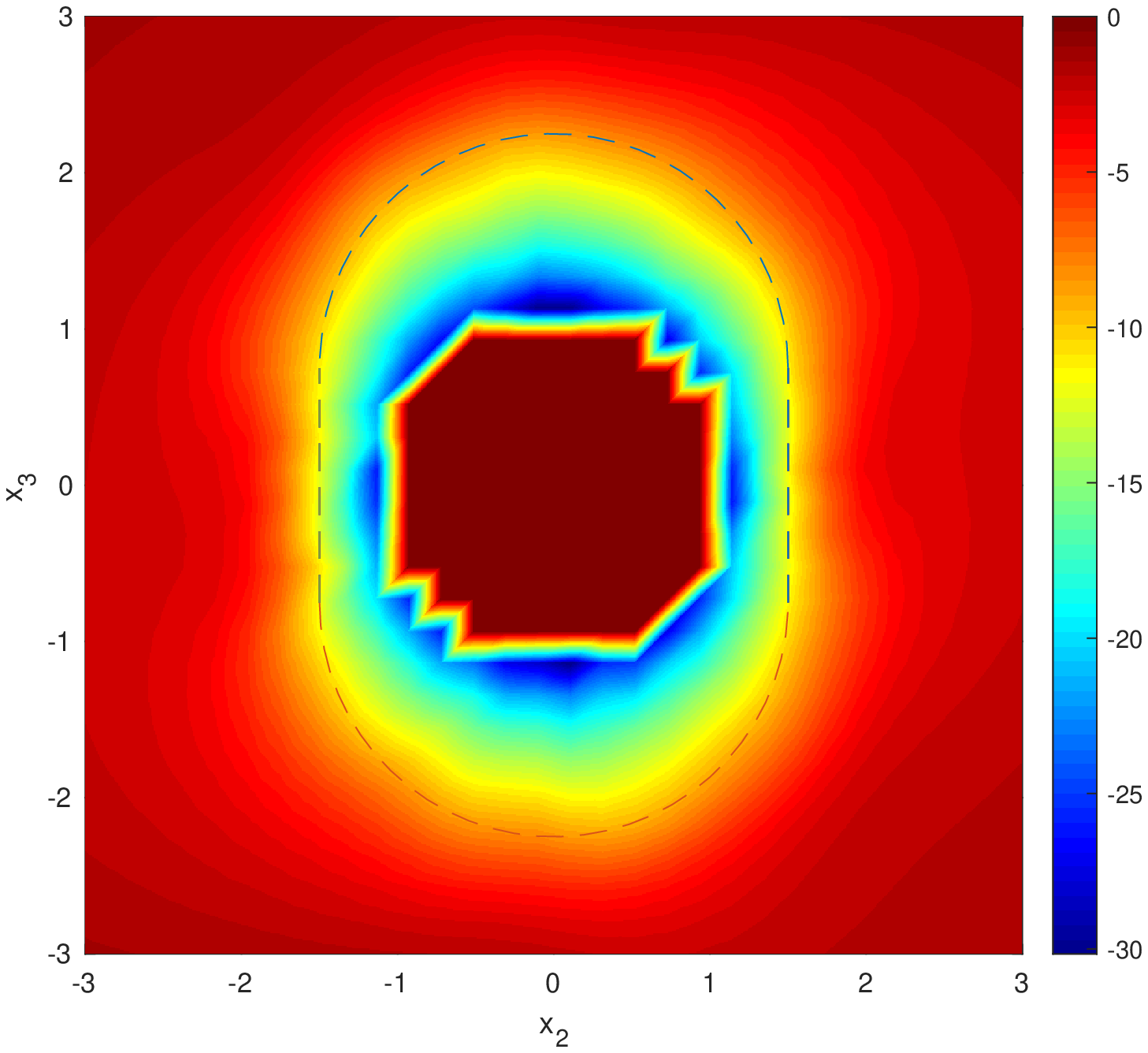}
\includegraphics[width=0.32\linewidth]{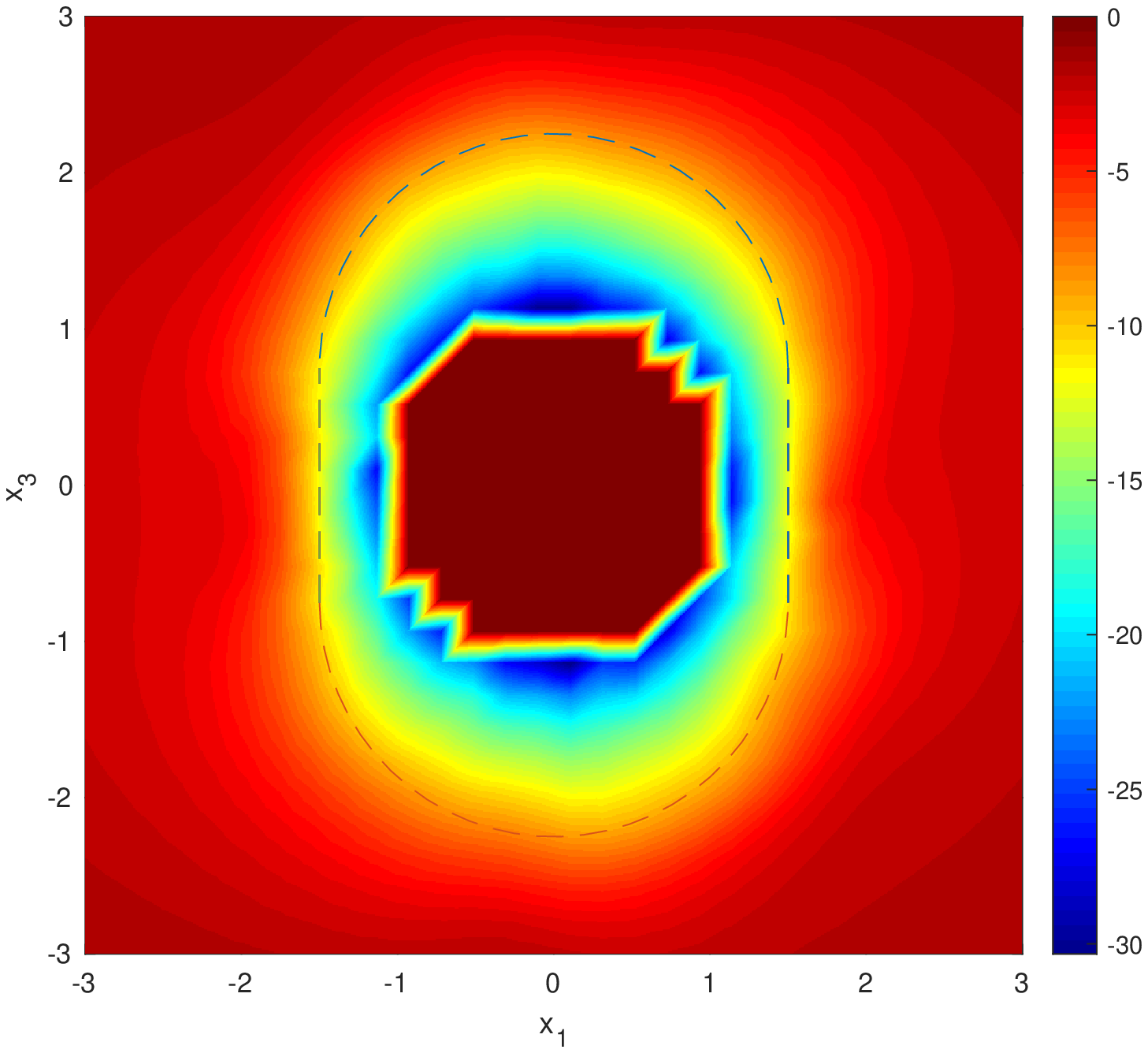}
     \caption{
     \linespread{1}
     Cross-section images of the cylinder-shape  cavity.
     } \label{cyl_cs}
    \end{figure}

In all of our numerical examples, we use the Finite Element software Netgen/NgSolve \cite{schoberl1997netgen} to generate the above synthetic data ${\{\hat{\be}_m(\by_j) \cdot \vecE^s(\bx_i,\by_j,\hat{\be}_\ell(\bx_i)): m,\ell=1,2\}_{i,j=1}^n}$. All the computations were done on a laptop with 8GB RAM. We use linear edge element to discretize the problem. The mesh size is chosen as $0.5$. We apply radial Perfectly Matched Layer (PML) to formulate the problem in a bounded domain. The material coefficients are chosen as $A=\boldsymbol{I}$, and $N=2\boldsymbol{I}$, where $\boldsymbol{I}$ is the $3\times3$ identity matrix. The wave number is chosen as $k=0.75$ in all the numerical examples, unless specified. The polarization vector is chosen as $\vech=(1/\sqrt{3},-1/\sqrt{3},1/\sqrt{3})$.

We further add $2\%$ Gaussian noise to the above synthetic data. To regularize the problem \eqref{lsm_discre2}, we apply Tikhonov regularization with Morozov principle to get the regularized solution $\vecg_{\bz}^\epsilon$, where $\epsilon$ is related to the noise level, see for instance \cite{CaCoMo}.

Our imaging function is then defined by
\begin{equation}
I(\bz): = \frac{1/\|\vecg_{\bz}^\epsilon\|}{\max_{\bz}1/\|\vecg_{\bz}^\epsilon\|}
\end{equation}
Indicated by our Theorem \ref{MaxwellLSM}, our imaging function is large for $\bz \not\in D$, while it is small for $\bz \in D \backslash \Sigma$. For best visualization, we plot the log scale of the imaging function over a mesh grid with grid distance $0.1$. Indicated by our Theorem \ref{MaxwellLSM}, we only need to sample the region outside the measurement ball $C$. Therefore in our numerical examples, we manually set the values (in log scale) of the imaging function inside the measurement ball $C$ to be $0$.

            \begin{figure}[hb!]
            \includegraphics[width=0.32\linewidth]{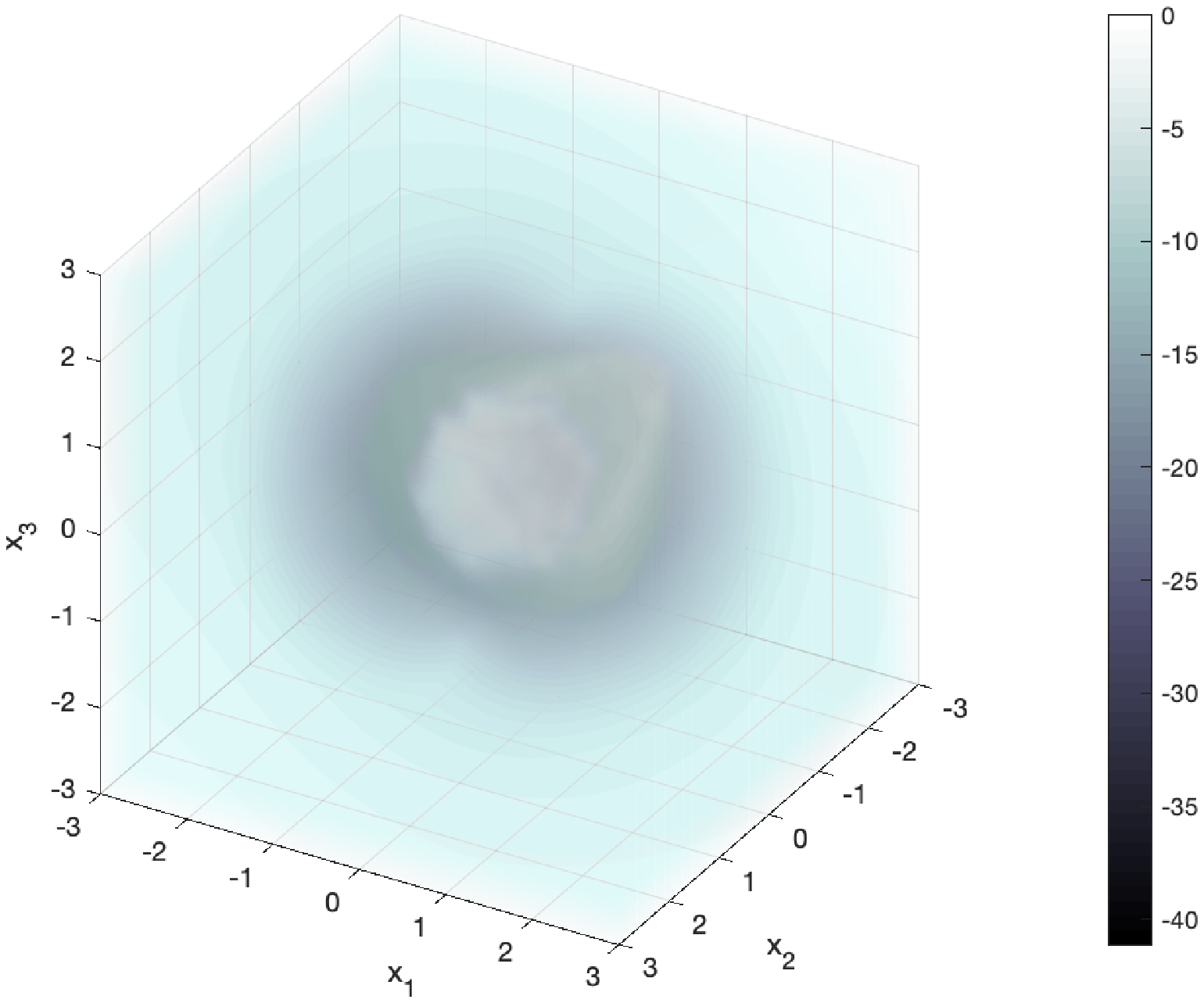}
\includegraphics[width=0.32\linewidth]{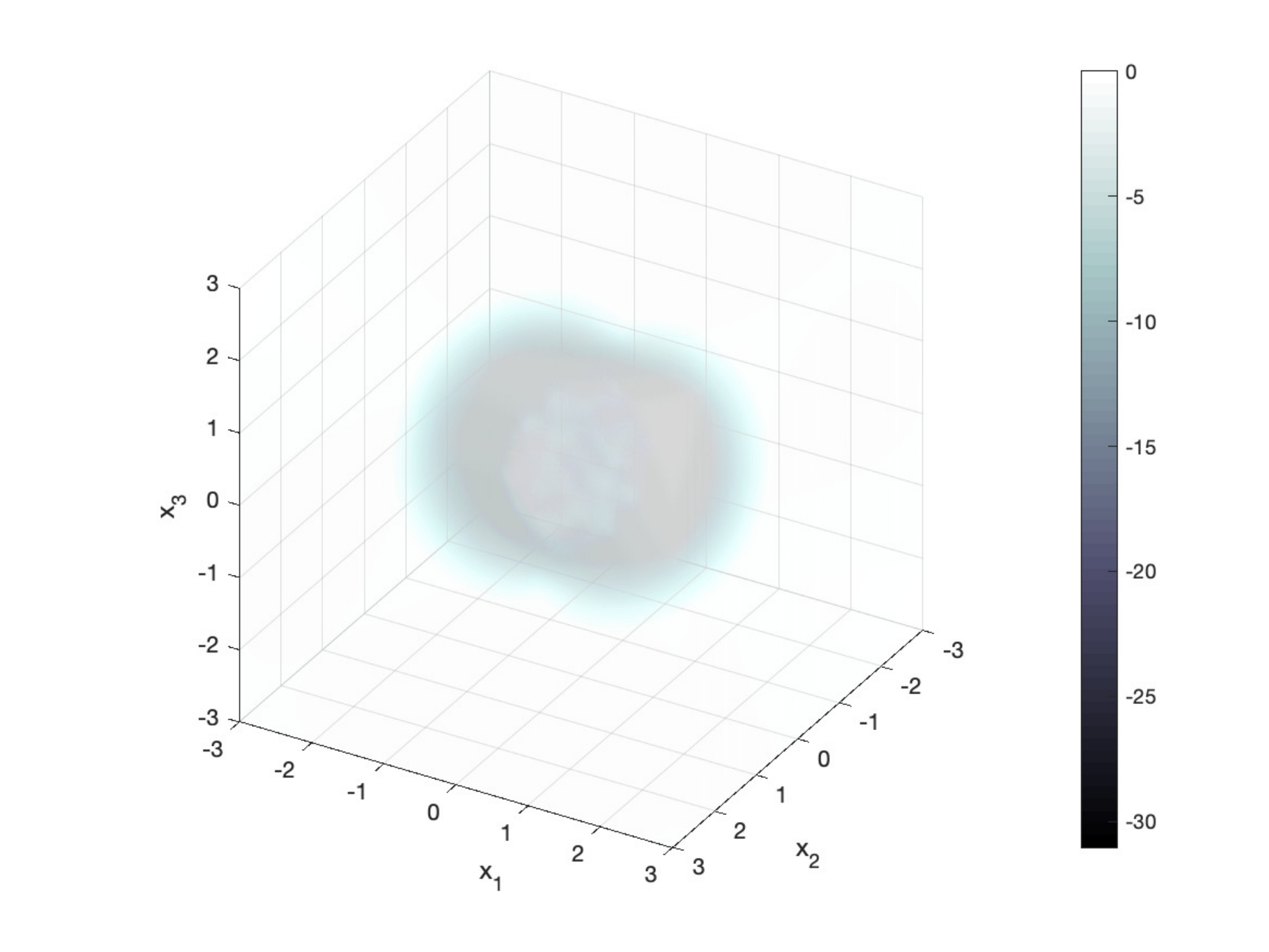}
\includegraphics[width=0.32\linewidth]{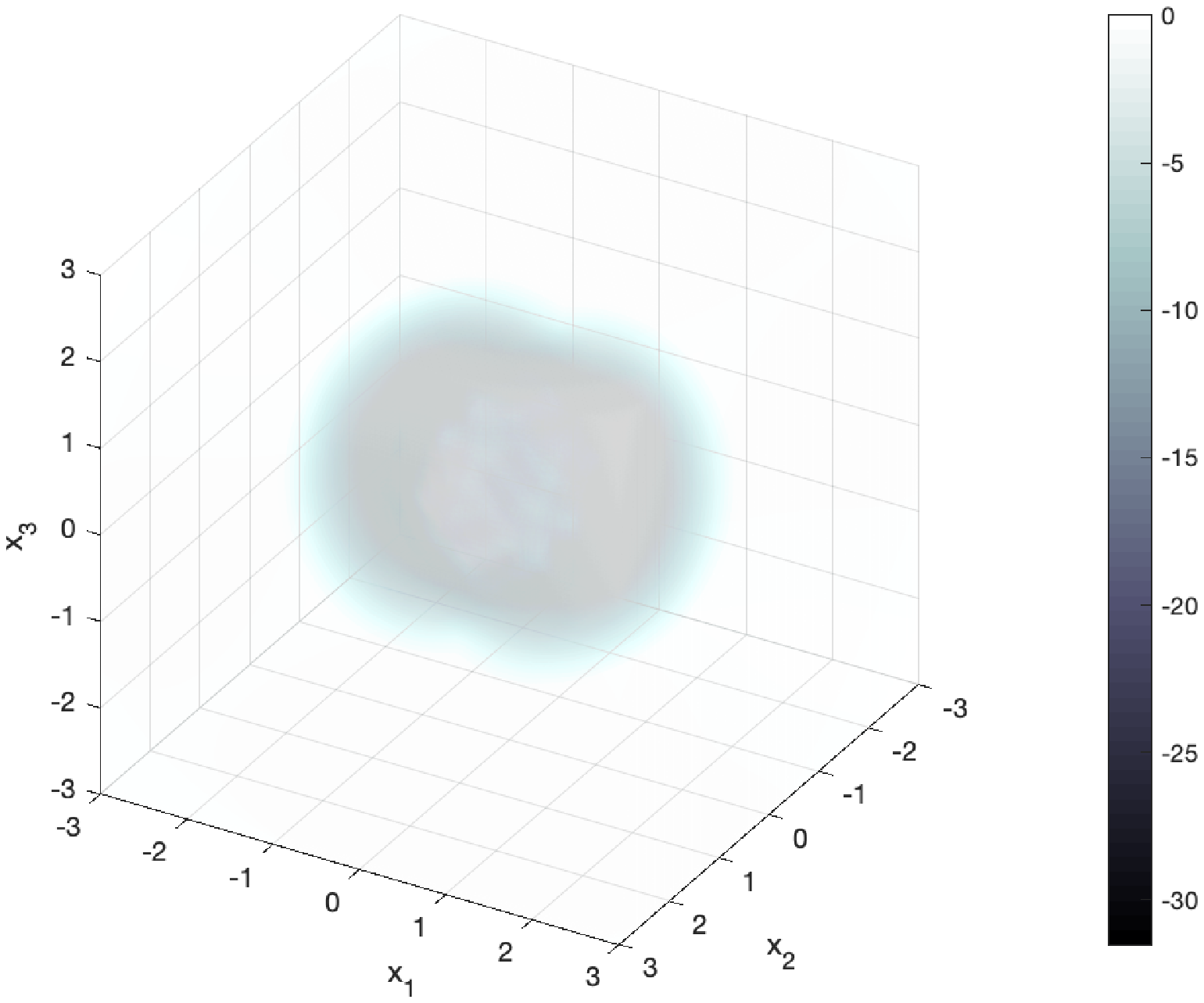}

\includegraphics[width=0.32\linewidth]{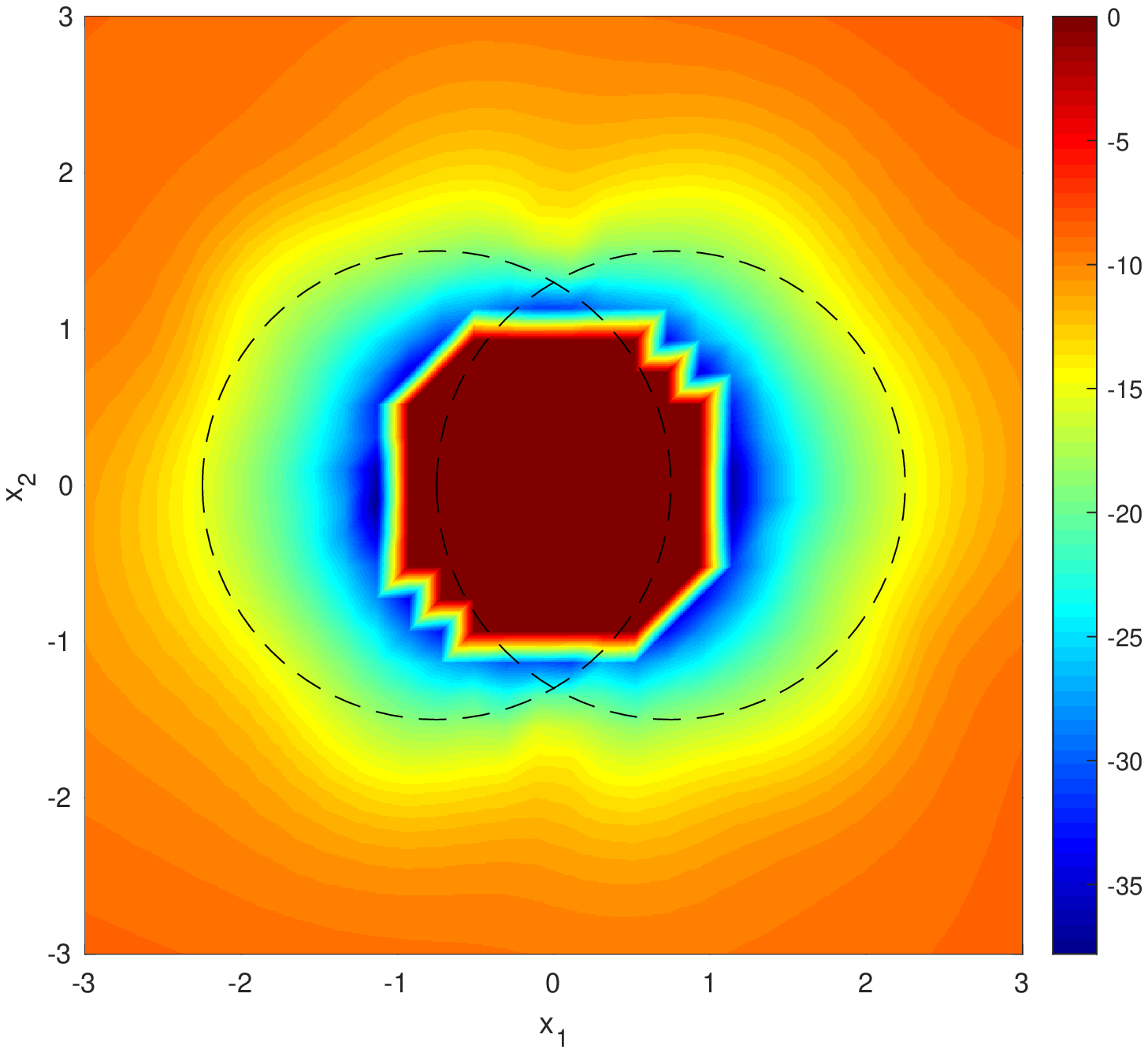}
   \includegraphics[width=0.32\linewidth]{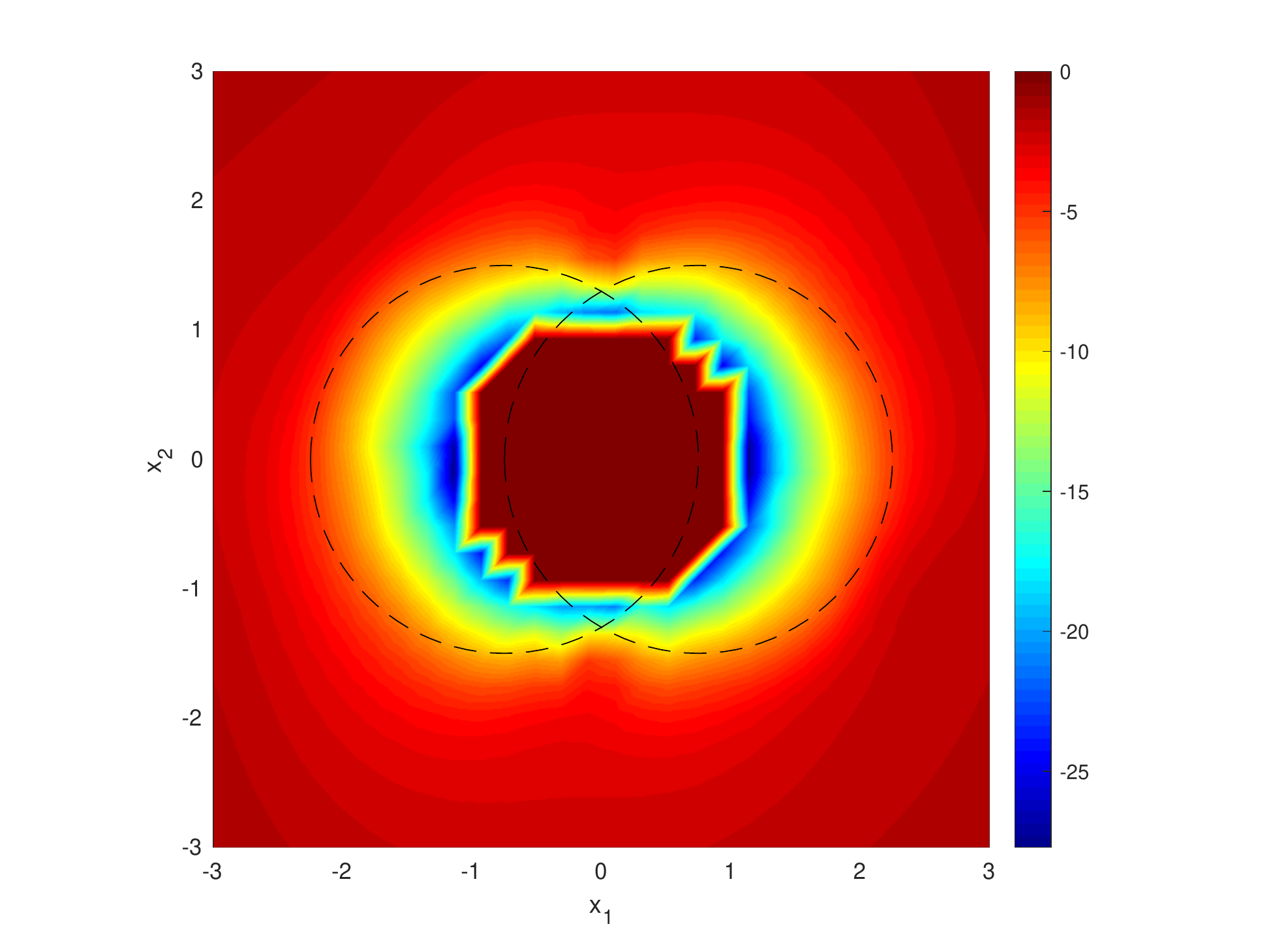}
\includegraphics[width=0.32\linewidth]{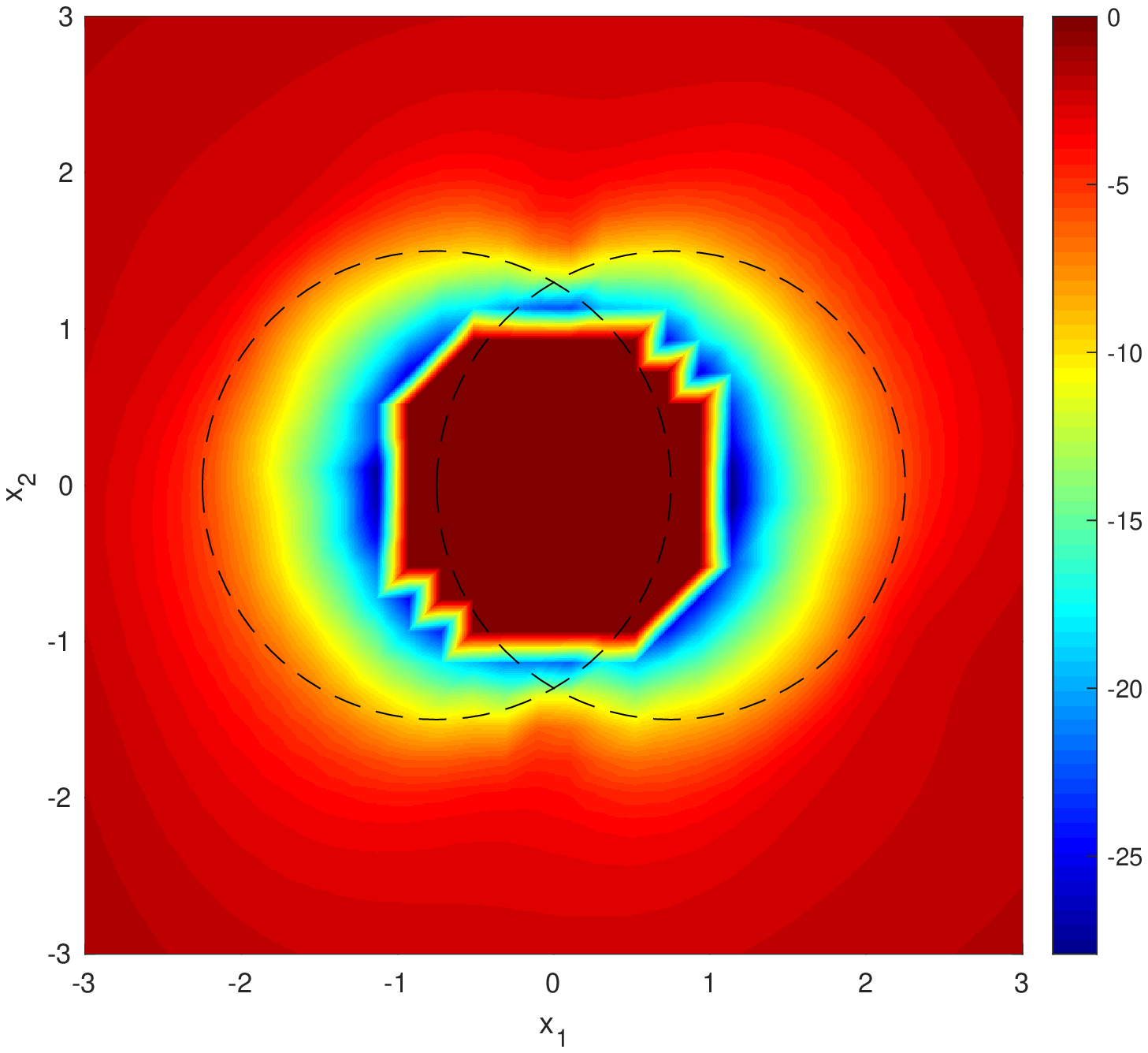}

\includegraphics[width=0.32\linewidth]{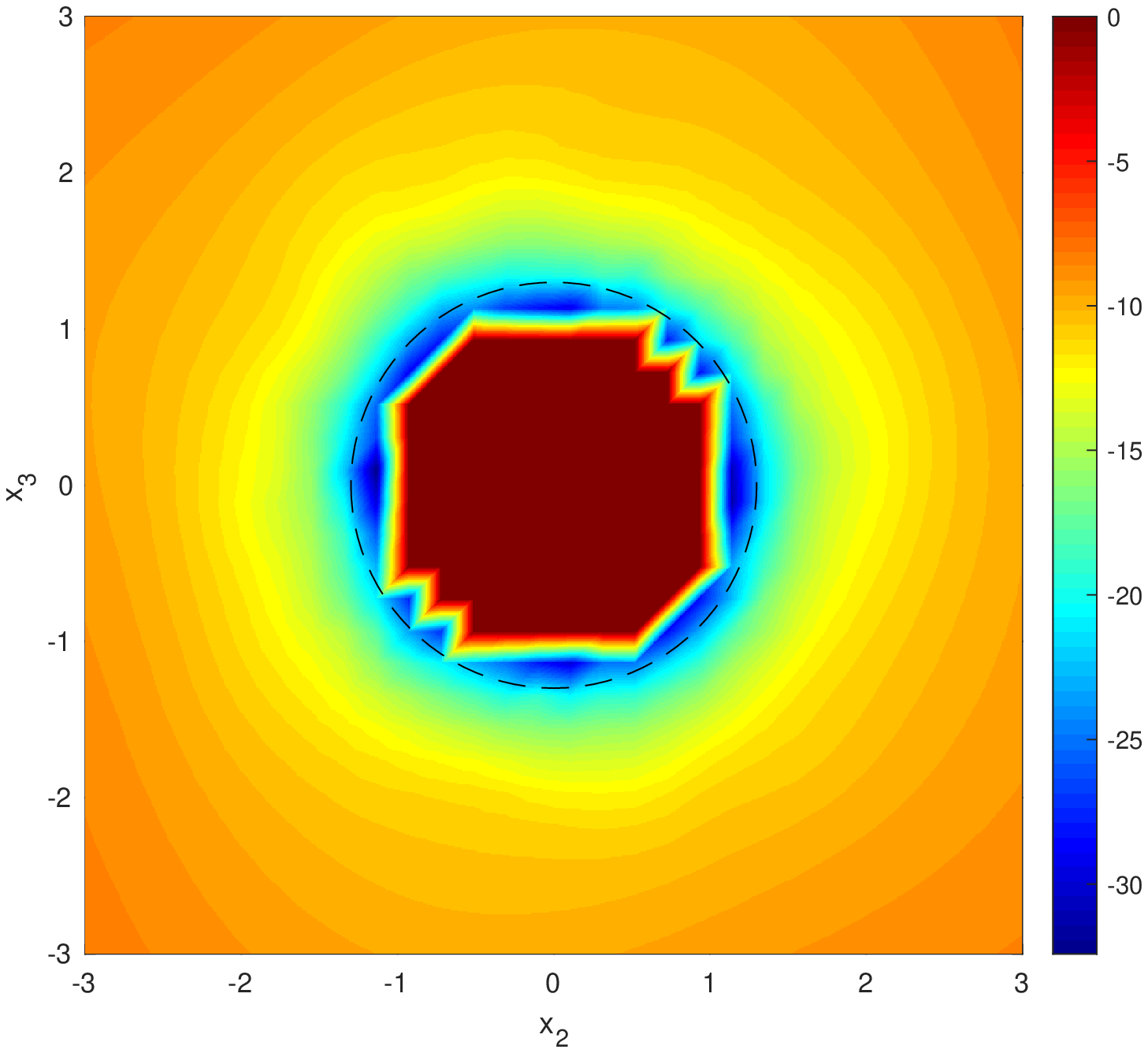}
   \includegraphics[width=0.32\linewidth]{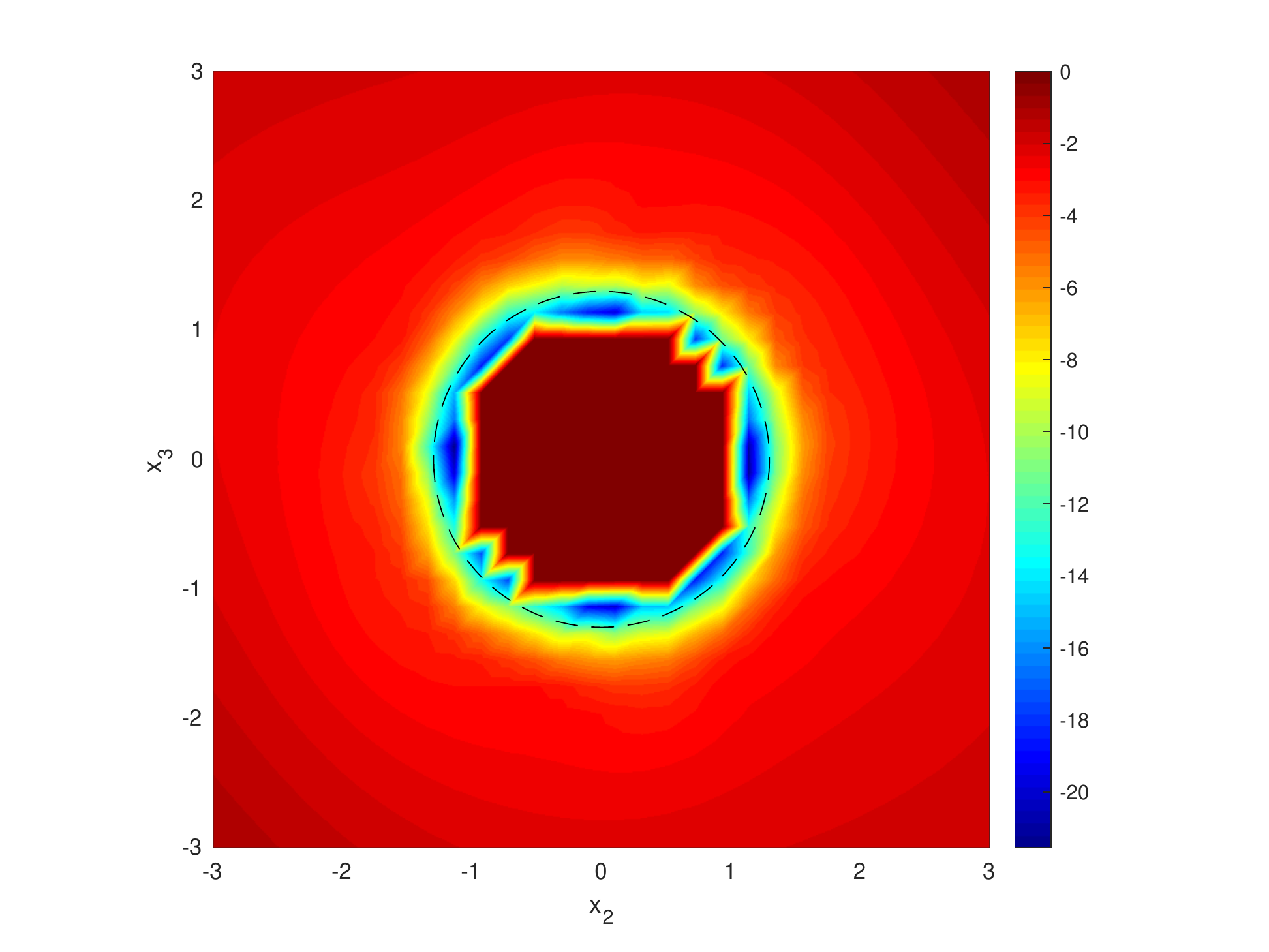}
\includegraphics[width=0.32\linewidth]{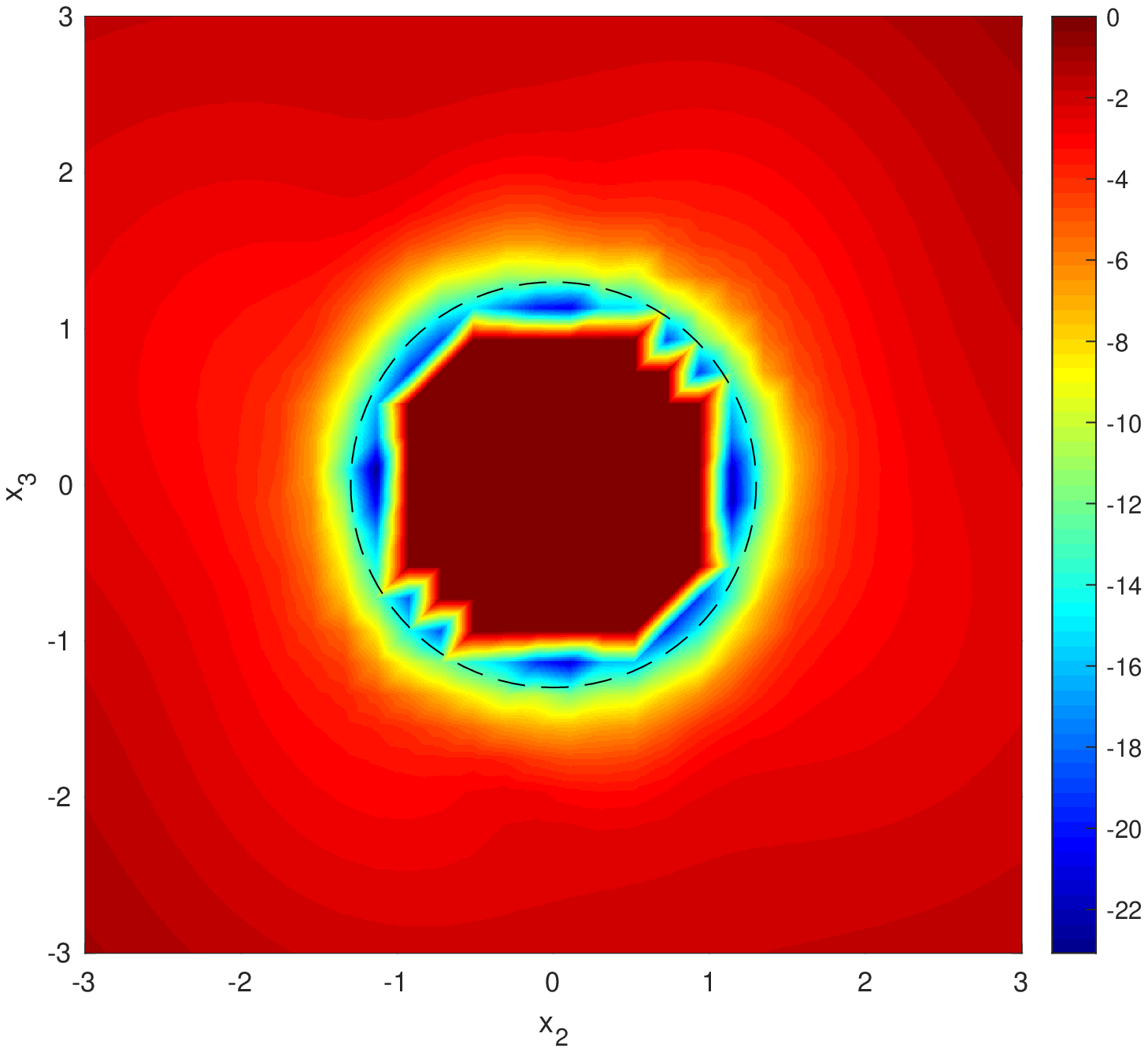}

\includegraphics[width=0.32\linewidth]{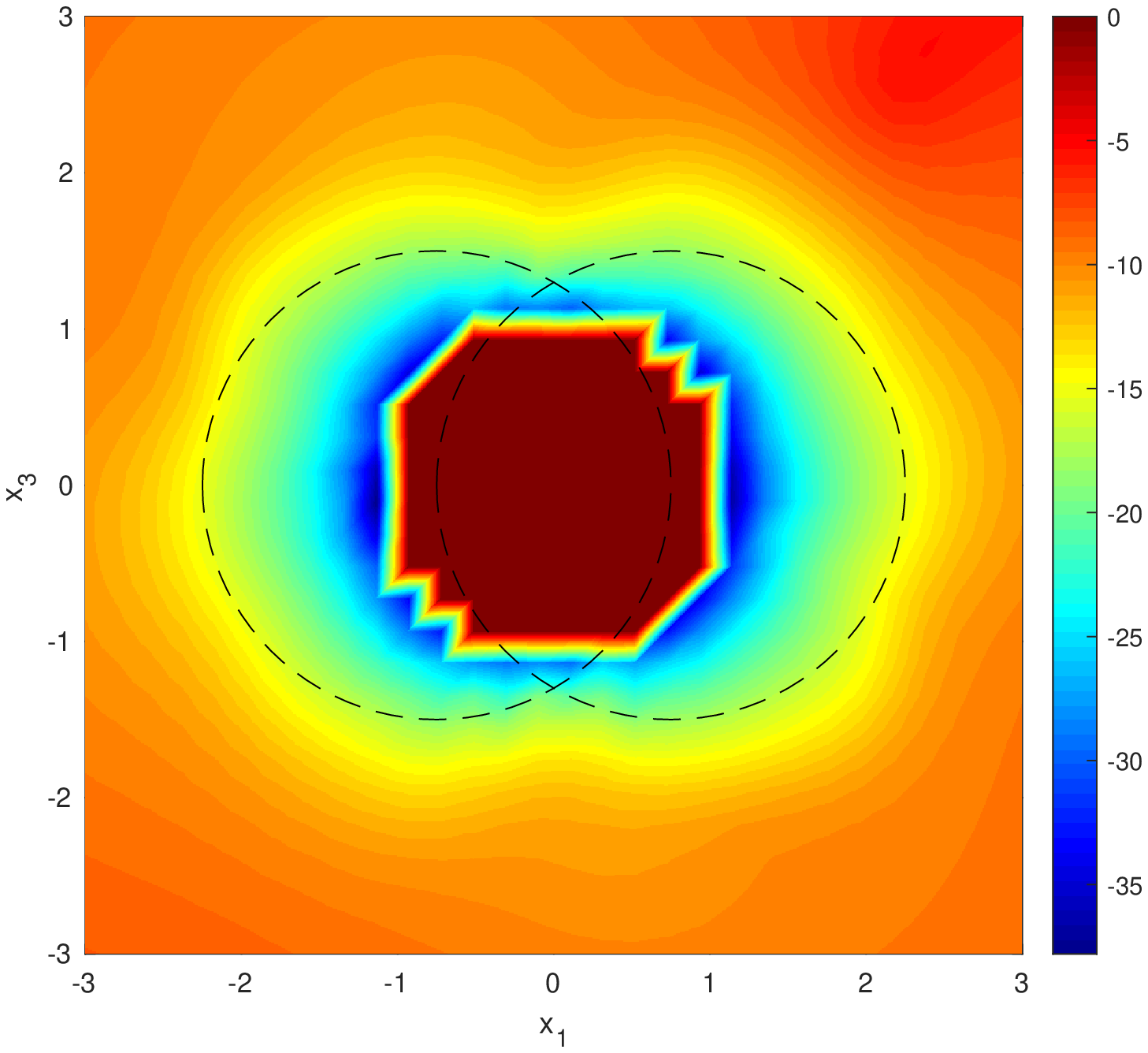}
   \includegraphics[width=0.32\linewidth]{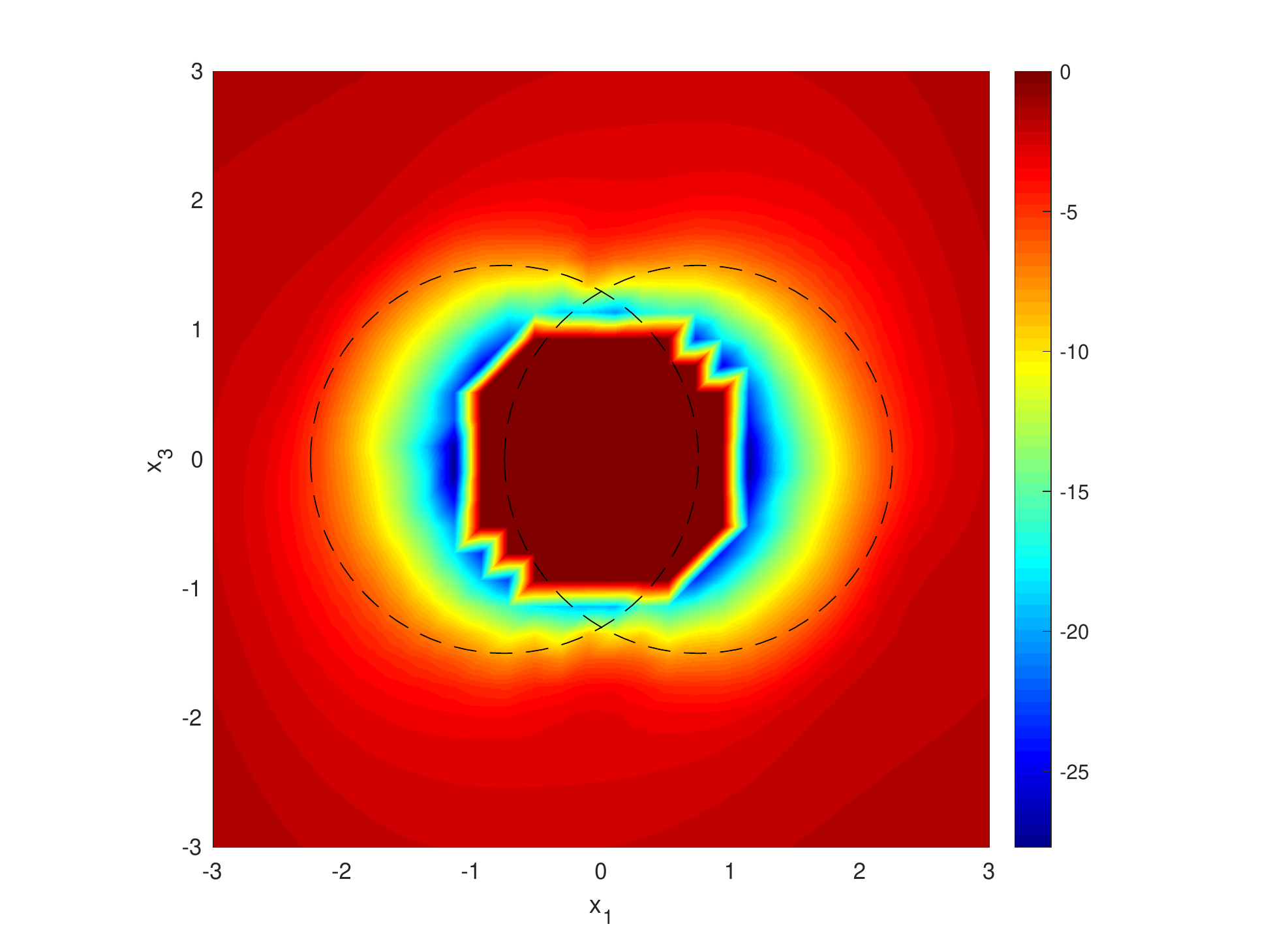}
\includegraphics[width=0.32\linewidth]{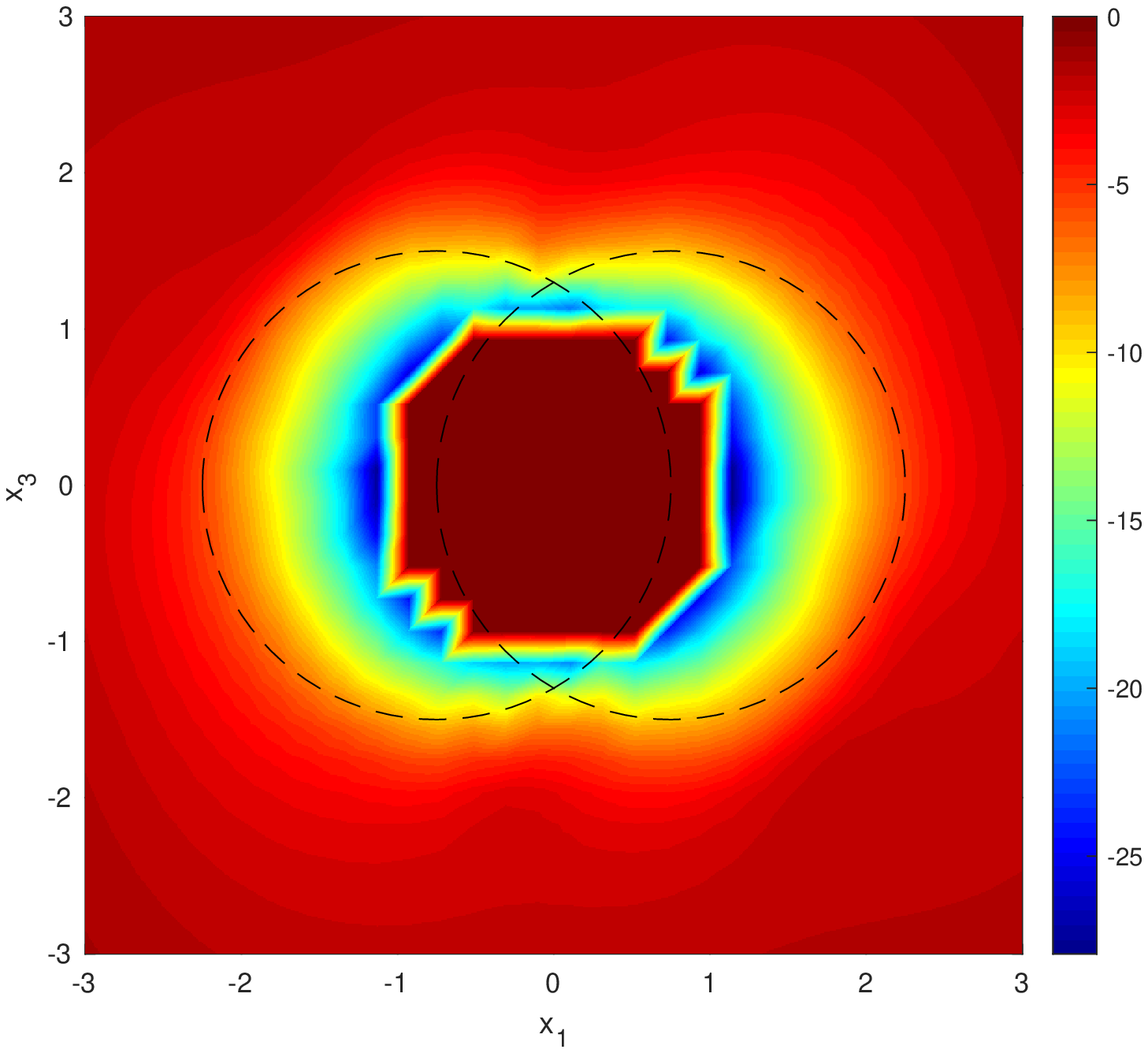}
     \caption{
Images of a peanut-shape cavity. Left column: $k=0.5$. Middle column: $k=0.75$. Right column: $k=1$. First row: three dimensional images. Second row: $x_1 x_2-$cross section images. Third row: $x_2x_3-$cross section images. Last row: $x_1x_3-$cross section images.
     } \label{peanut_wavenums}
    \end{figure}

For the first numerical example, the cavity $D$ is a ball centered at origin with radius $1.5$ defined by
$$
\{ \bx=(x_1,x_2,x_3): \sqrt{x_1^2 + x_2^2+ x_3^2} < 1.5\}.
$$
Fig. \ref{ball} displays the three dimensional image, and Fig. \ref{ball_cs} displays the cross-section images where the cross-section of the exact ball is indicated by the dashed line.

The second numerical example is for  a peanut-shape cavity $D$ defined by
$$
\{ \bx=(x_1,x_2,x_3): \sqrt{(x_1-0.75)^2 + x_2^2+ x_3^2} < 1.5 \mbox{ or } \sqrt{(x_1+0.75)^2 + x_2^2+ x_3^2} < 1.5\}.
$$
Fig. \ref{peanut} displays the three dimensional image, and Fig. \ref{peanut_cs} displays the cross-section images where the cross-section of the exact  peanut-shape cavity is indicated by the dashed line.

The third numerical example is for a cylinder-shape cavity $D$ defined by
$$
\Bigg\{ \bx=(x_1,x_2,x_3):
\begin{array}{ccc}
\sqrt{x_1^2 + x_2^2} <1.5  &  \mbox{ for } & |x_3| < 0.75  \\
\sqrt{x_1^2 + x_2^2 + (x_3-0.75)^2} <1.5  & \mbox{ for }  &  0.75 < x_3 < 2.25\\
\sqrt{x_1^2 + x_2^2 + (x_3+0.75)^2} <1.5  &  \mbox{ for }  &   -2.25 < x_3 <  -0.75
\end{array}\Bigg\}.
$$
Fig. \ref{cyl} displays the three dimensional image, and Fig. \ref{cyl_cs} displays the cross-section images where the cross-section of the exact cylinder-shape cavity is indicated by the dashed line. So far, our algorithm gives good images for the shape of the ball, the peanut, and the cylinder. We can clearly distinguish those different cavities.

We further illustrate the performance of the linear sampling method for different wave numbers $k\in \{0.5, 0.75, 1\}$ in Fig. \ref{peanut_wavenums}. Here the exact geometry is a peanut-shape cavity as in Fig. \ref{peanut}. 

{Finally we compare the numerical reconstructions of a cuboid cavity using a single polarization vector and three linearly independent polarization vectors (similar comparisons were reported in  \cite{haddar2002linear} for the typical scattering case). The cuboid is given by
$$
\{ \bx=(x_1,x_2,x_3): |x_1| < 1.5, |x_2| < 1.5, |x_3| < 2 \}.
$$
For the best visualization and comparison, we use iso-surface plot for all of the images in Fig. \ref{cuboid}. The second, third, and last columns are reconstructions using a single polarization vector, three polarization vectors, and three polarization vectors with noiseless data respectively. We observe that the cuboid can be (slightly) better imaged using three polarization vectors $(1,0,0)$, $(0,1,0)$, $(0,0,1)$. 
We finally remark that the rounded corners may be due to the regularization.}

\begin{figure}[hb!]
\includegraphics[width=0.24\linewidth]{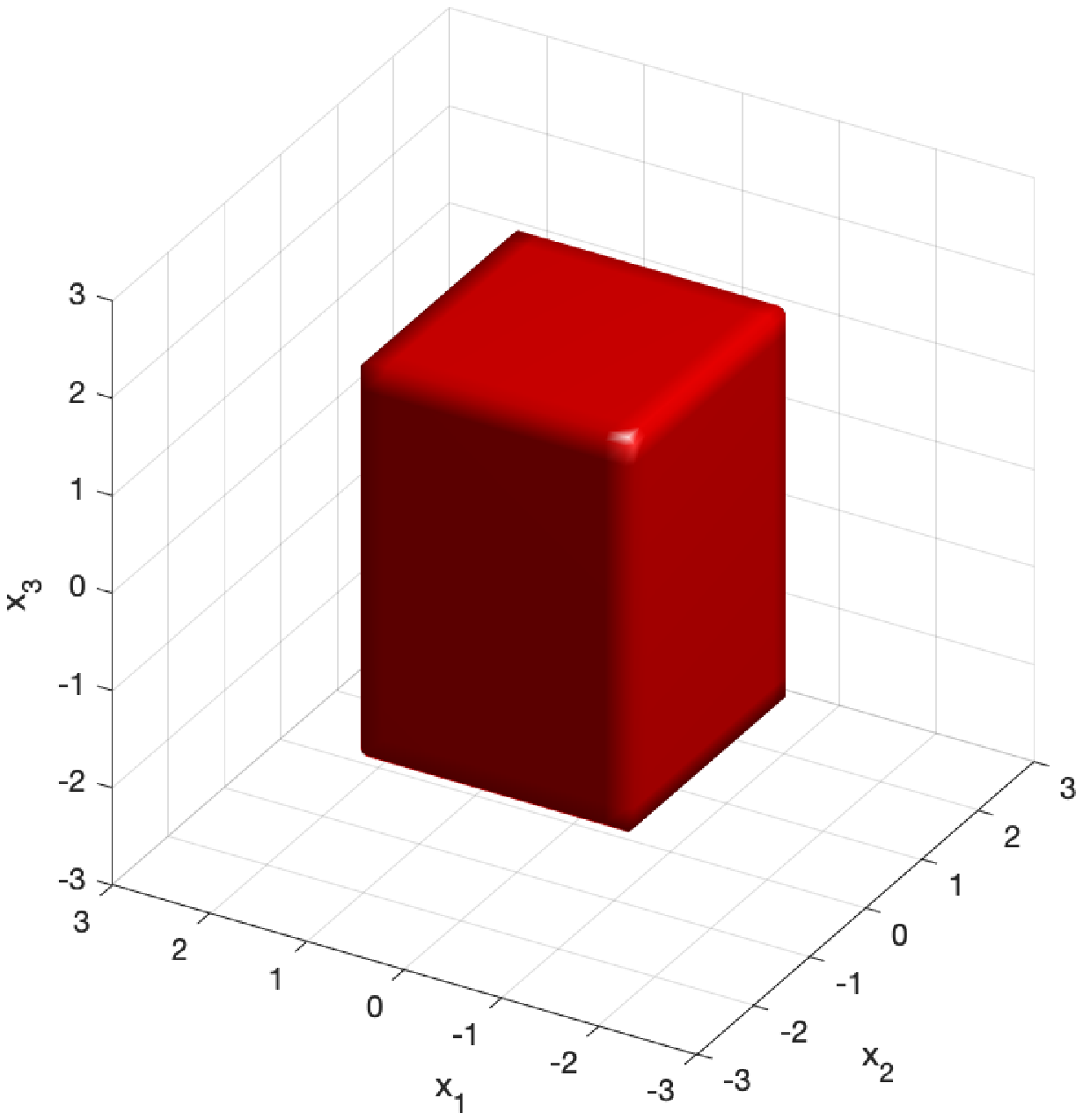}\includegraphics[width=0.24\linewidth]{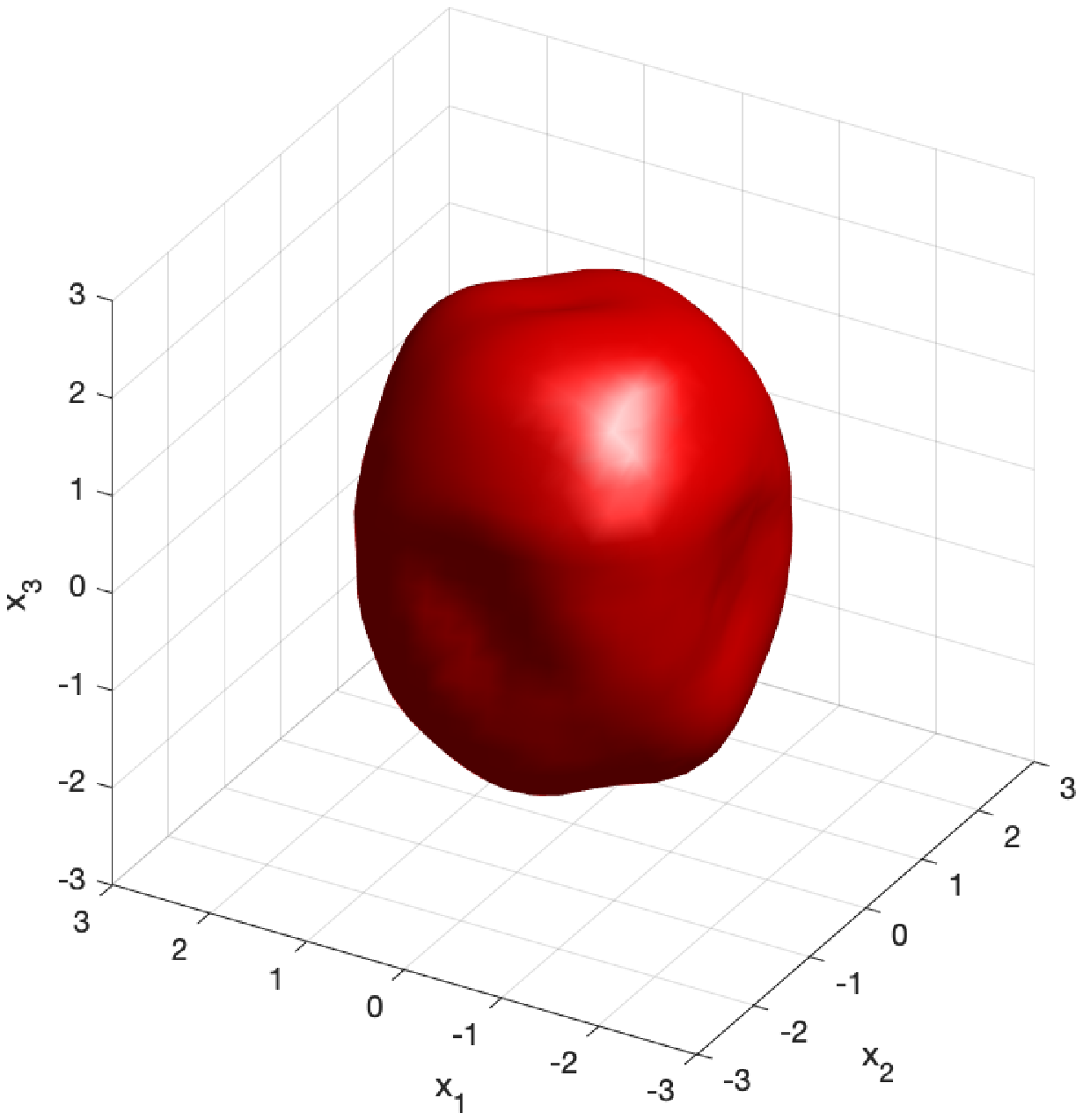}
\includegraphics[width=0.24\linewidth]{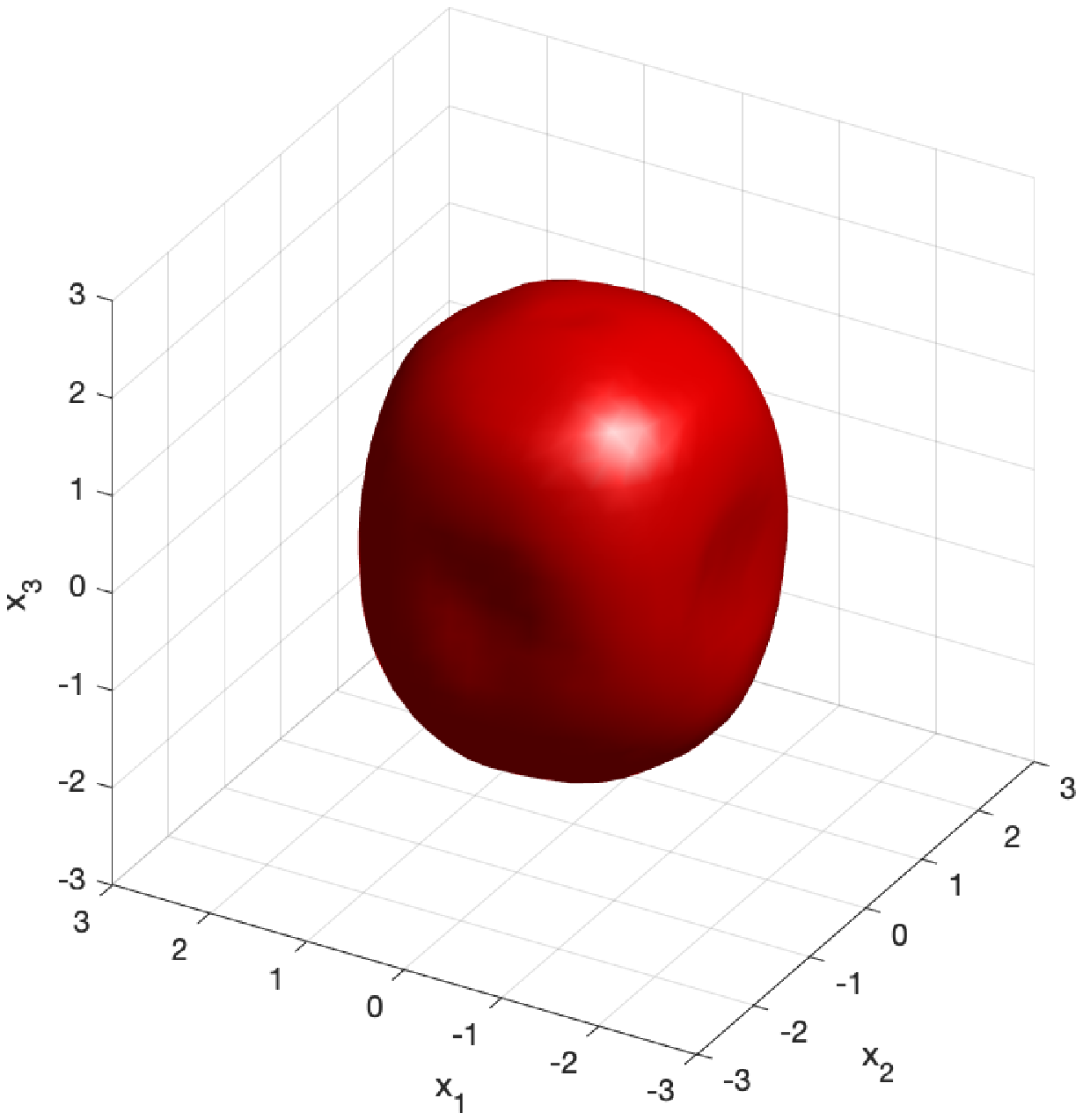}
\includegraphics[width=0.24\linewidth]{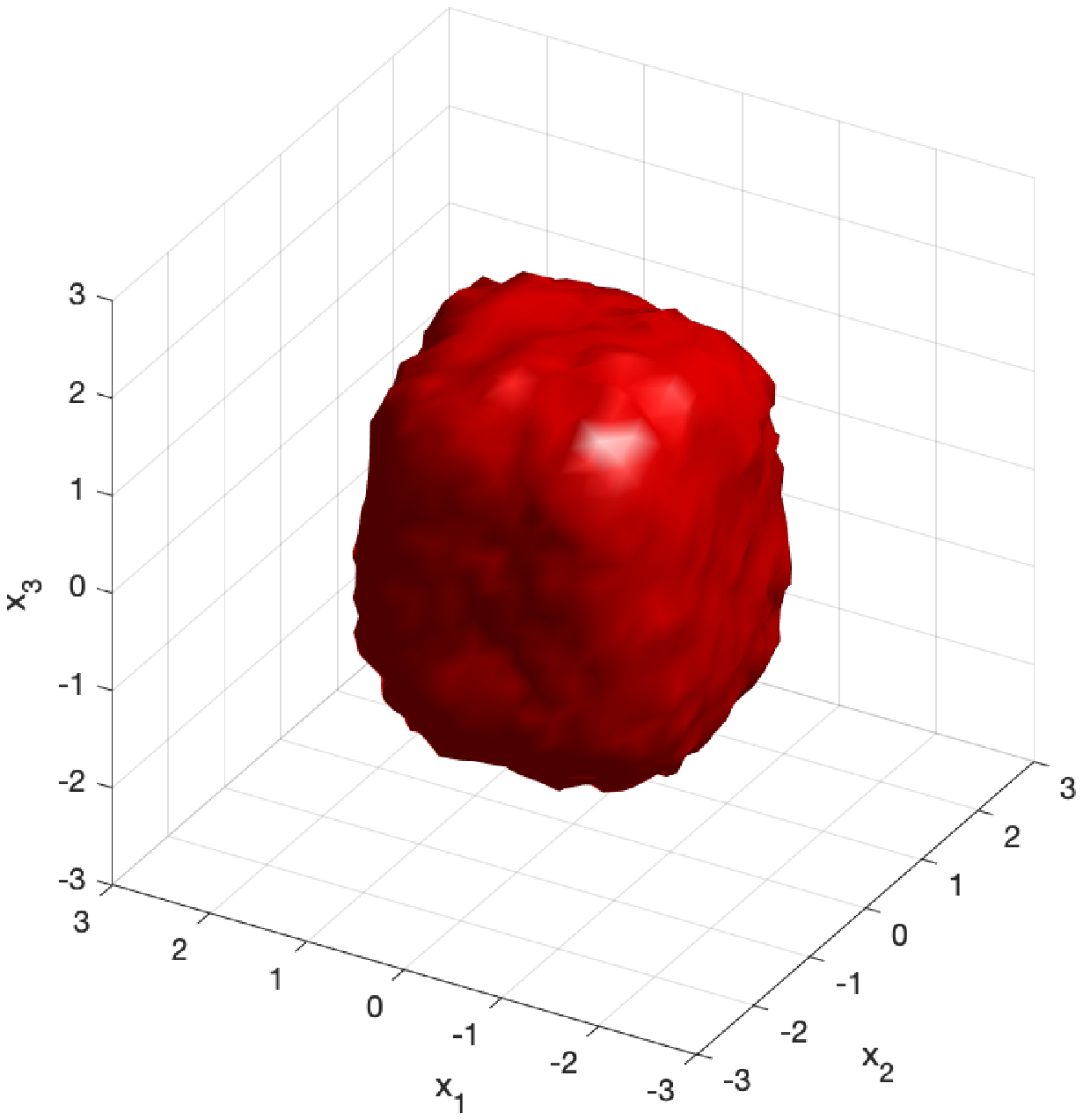}
            
\includegraphics[width=0.24\linewidth]{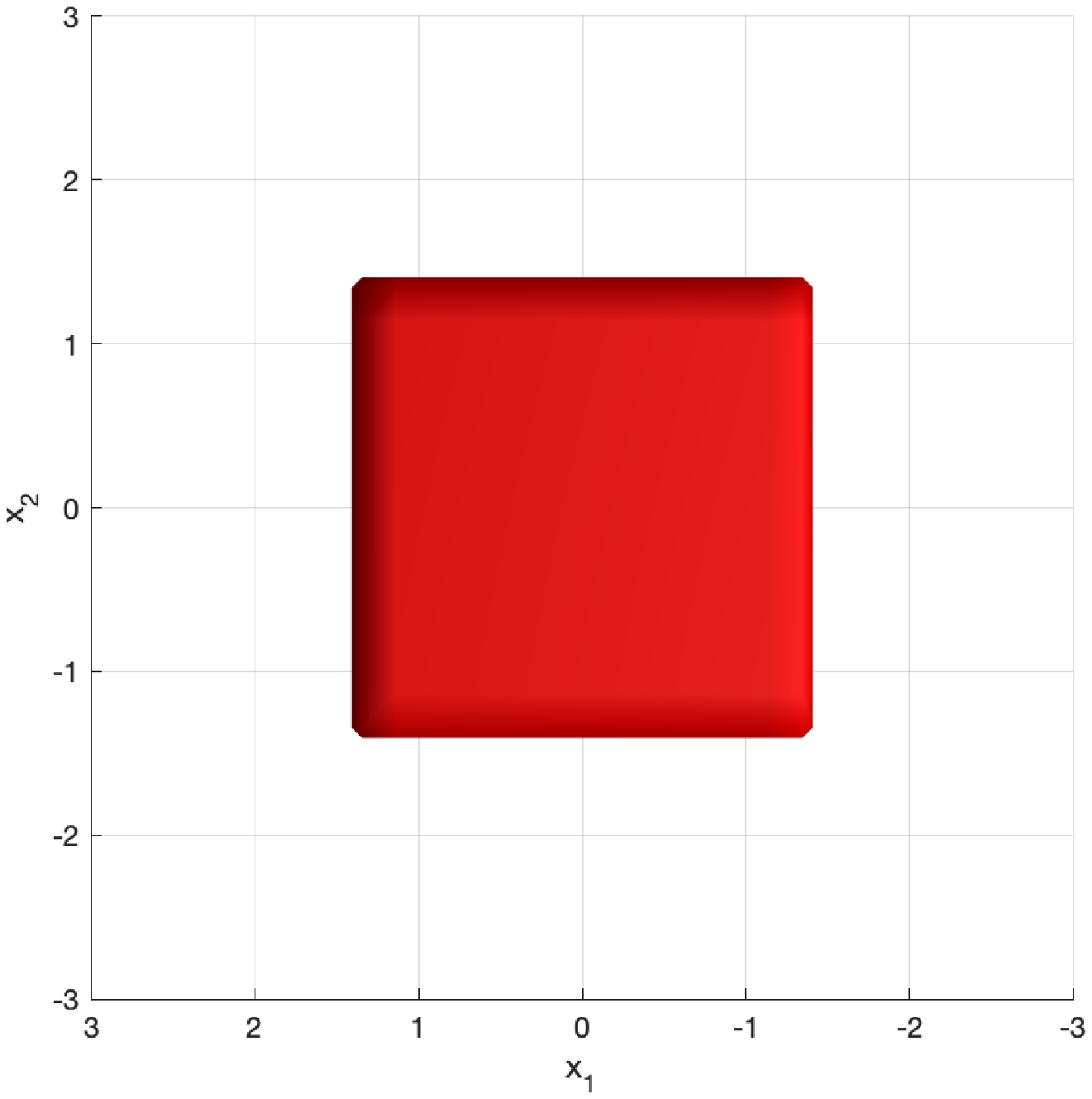}
\includegraphics[width=0.24\linewidth]{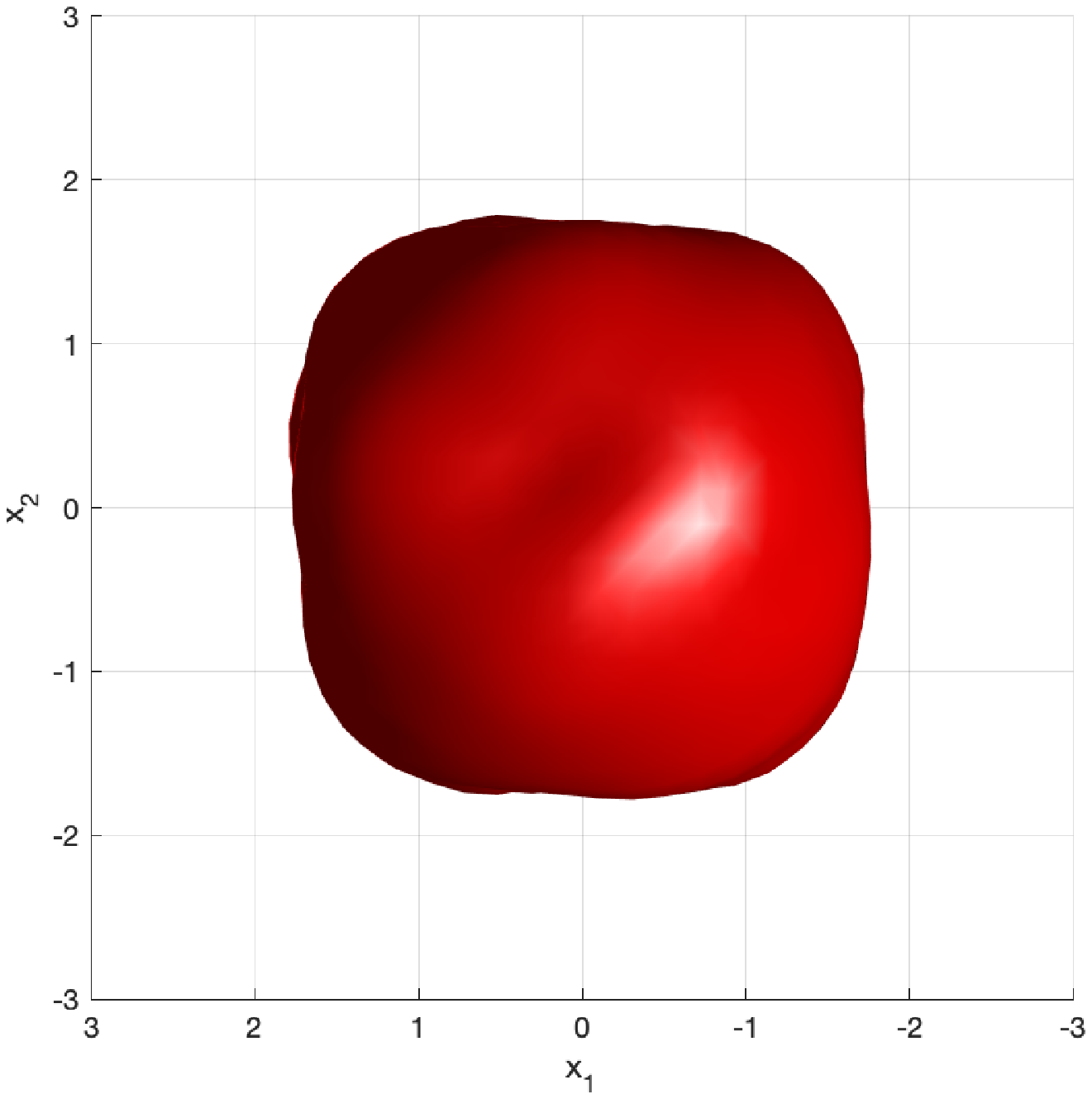}
\includegraphics[width=0.24\linewidth]{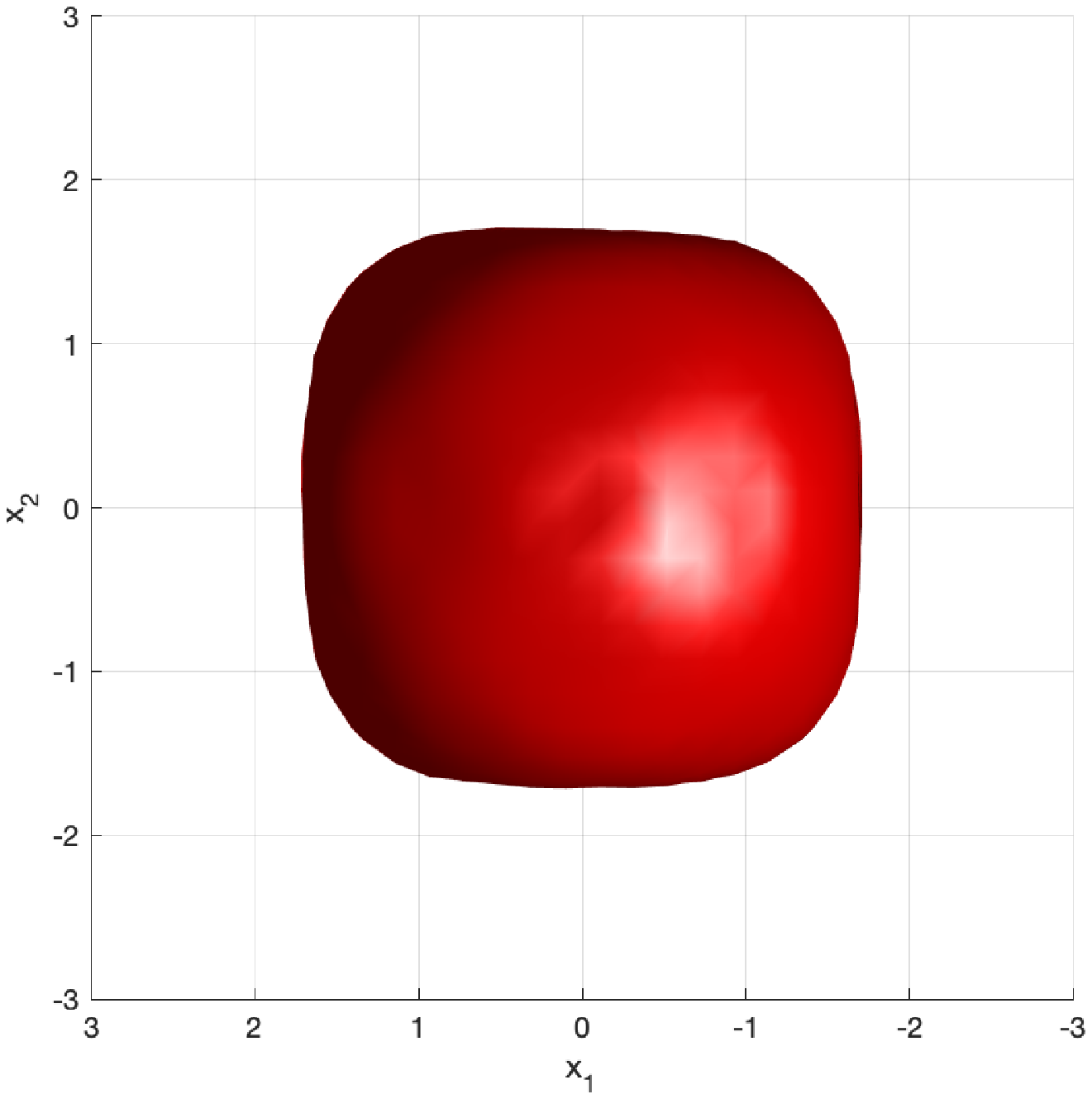}
\includegraphics[width=0.24\linewidth]{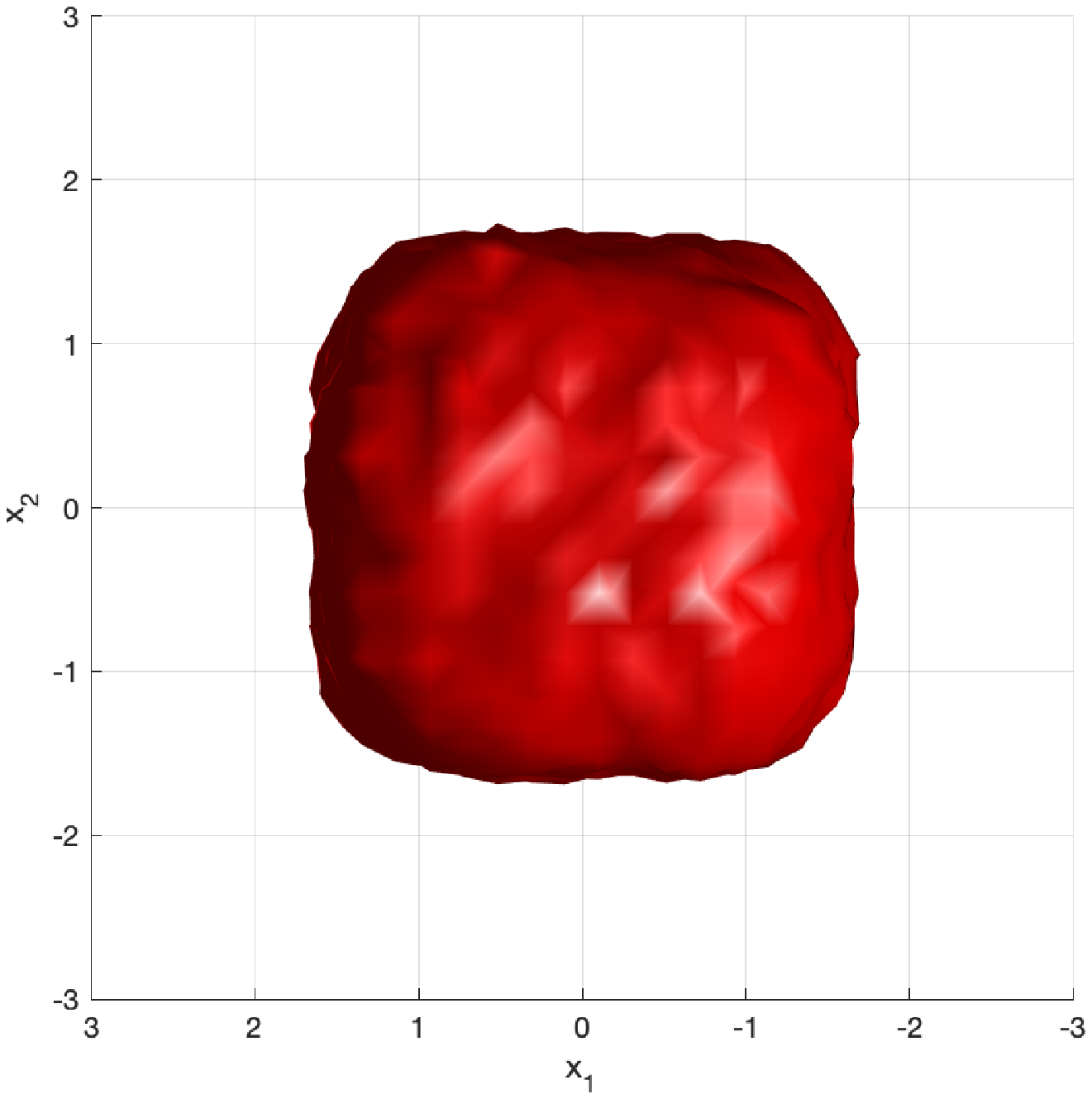}

\includegraphics[width=0.24\linewidth]{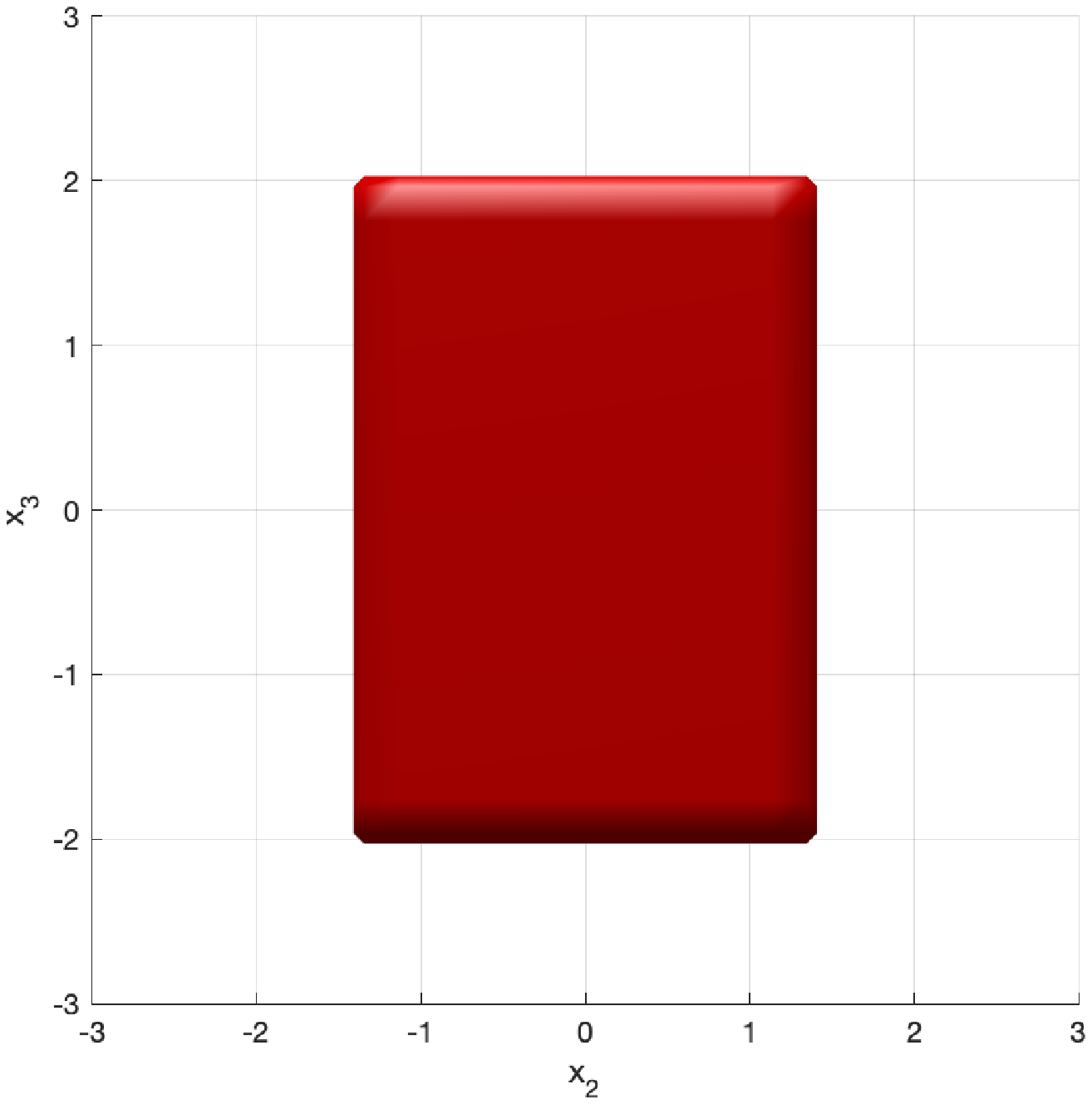}
\includegraphics[width=0.24\linewidth]{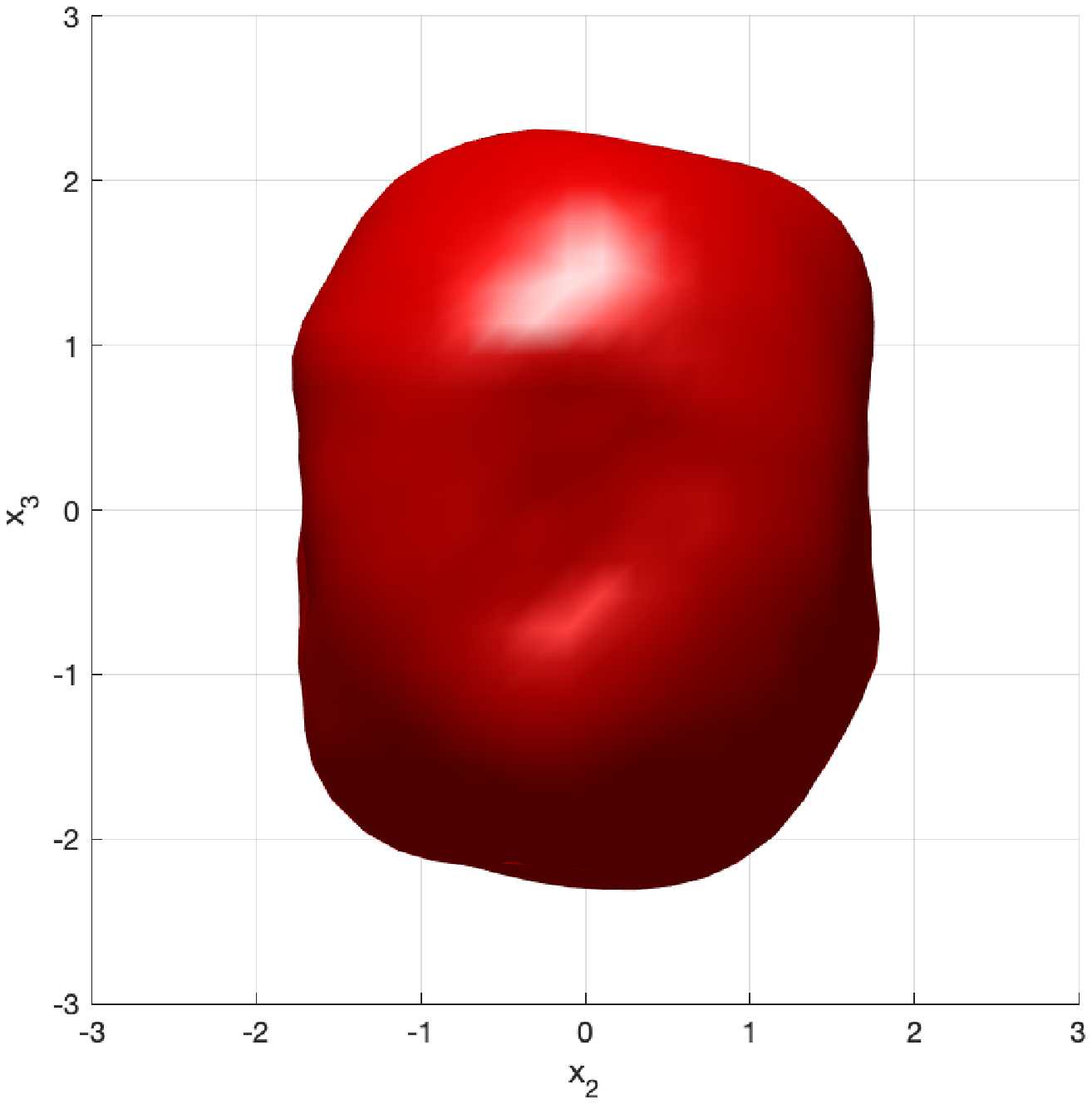}
\includegraphics[width=0.24\linewidth]{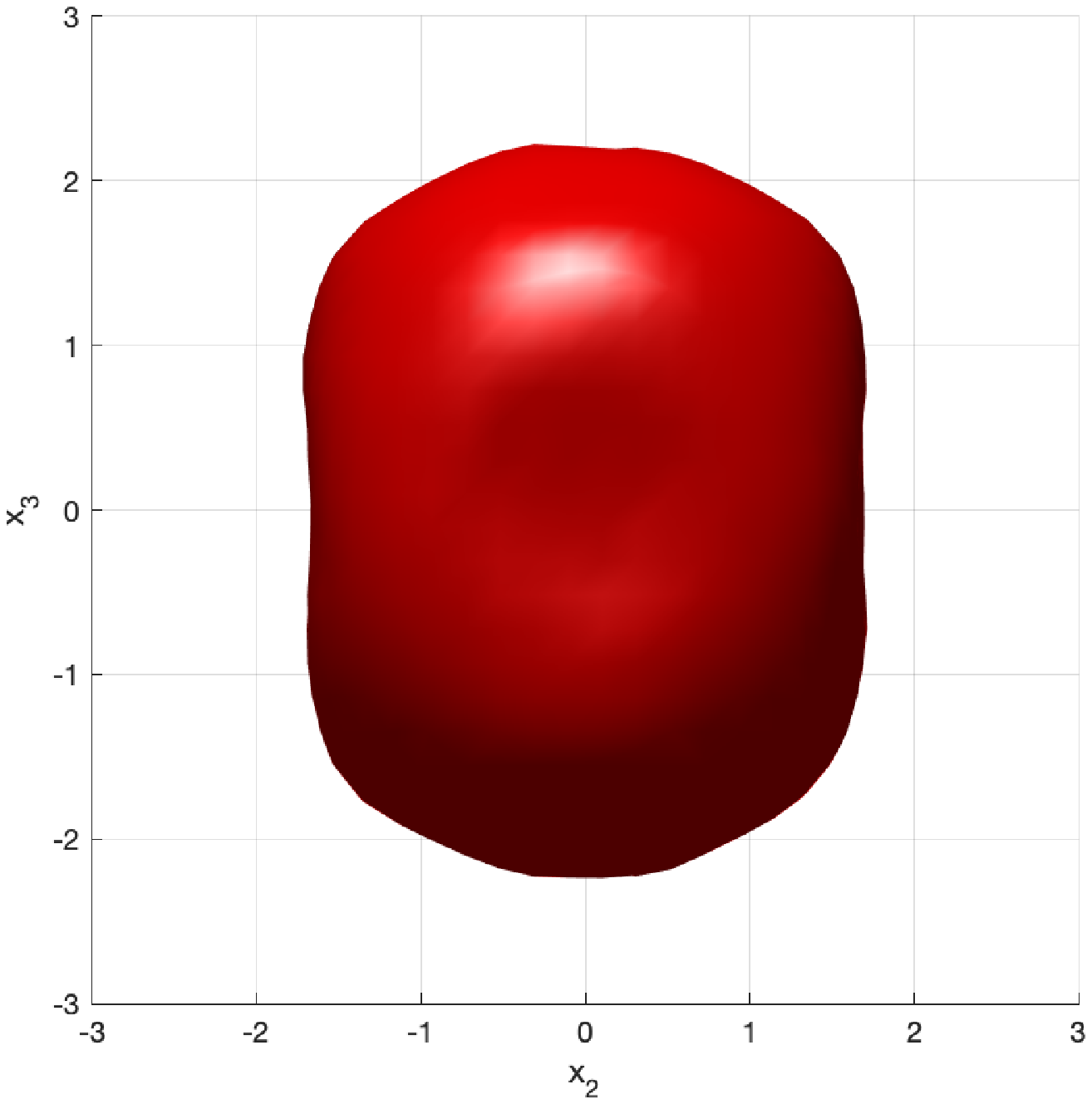}
\includegraphics[width=0.24\linewidth]{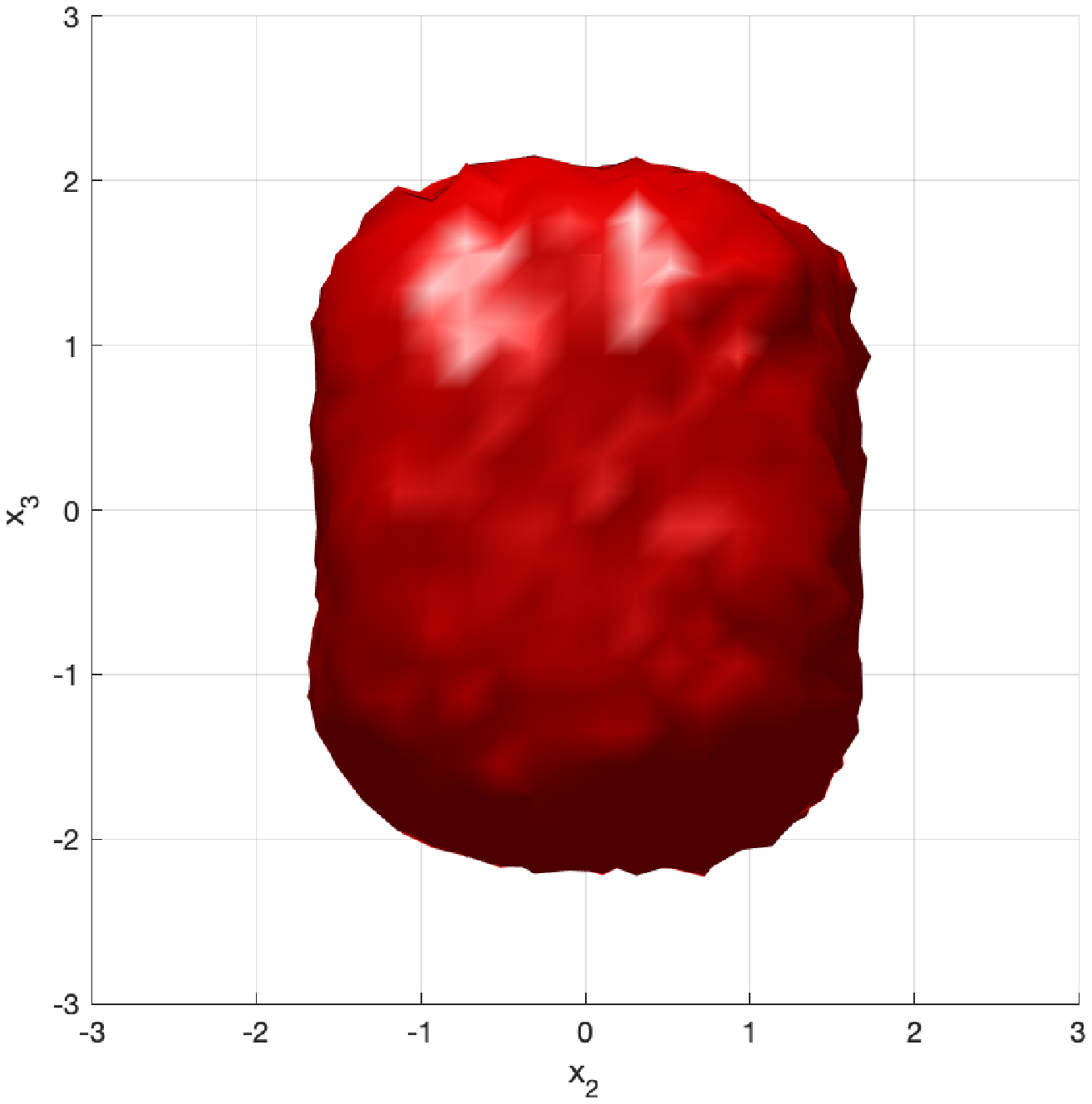}

\includegraphics[width=0.24\linewidth]{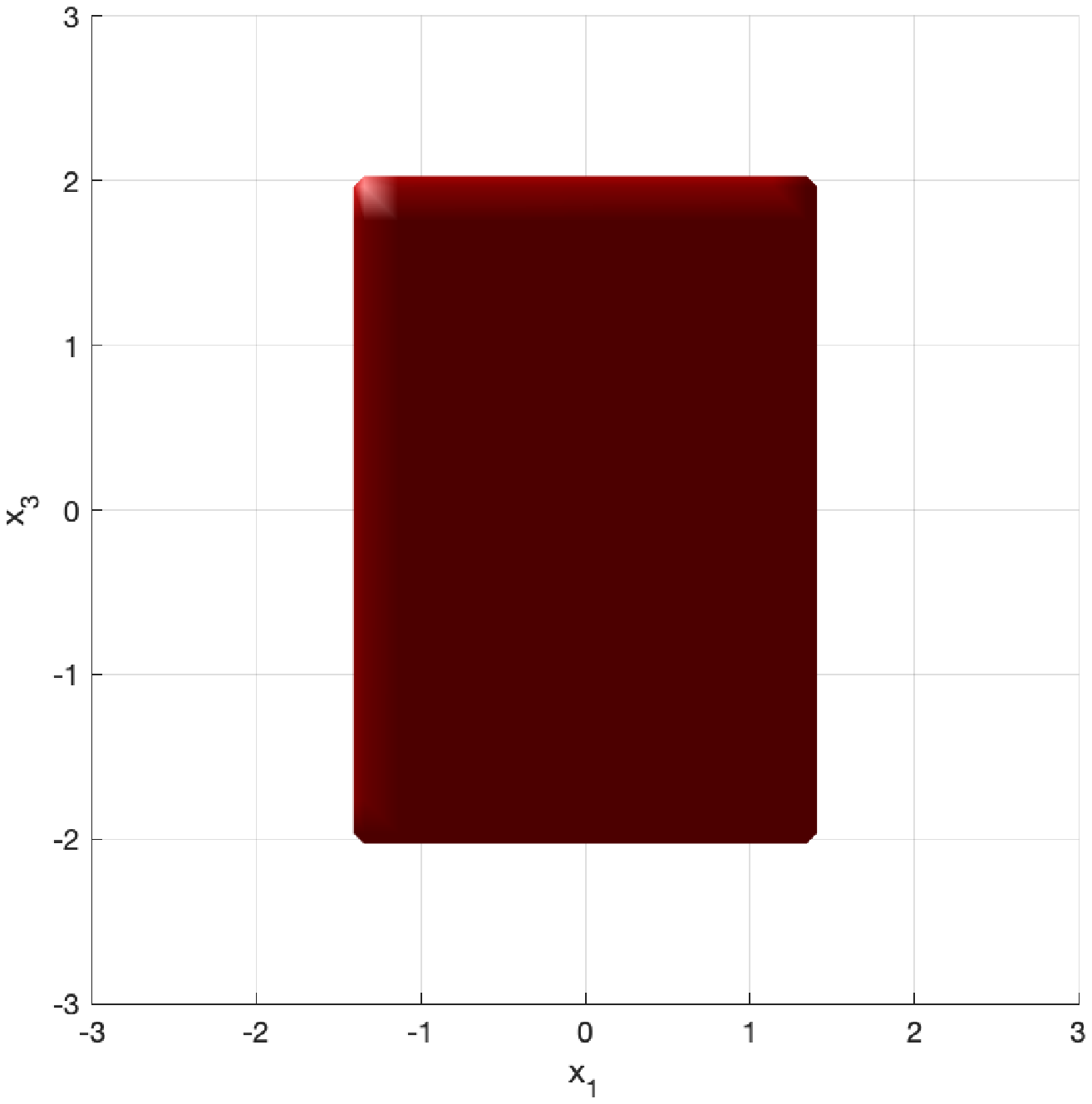}
\includegraphics[width=0.24\linewidth]{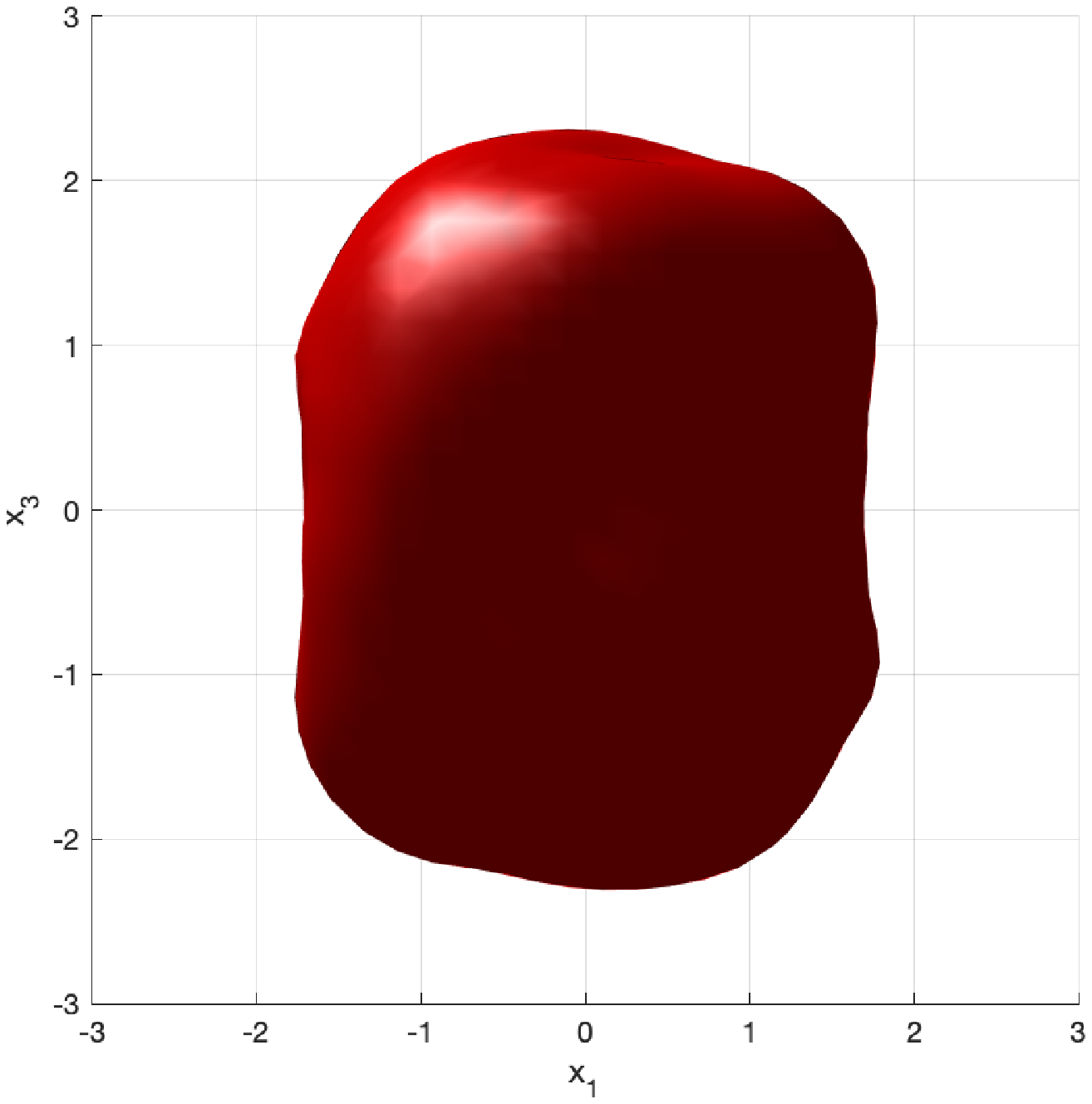}
\includegraphics[width=0.24\linewidth]{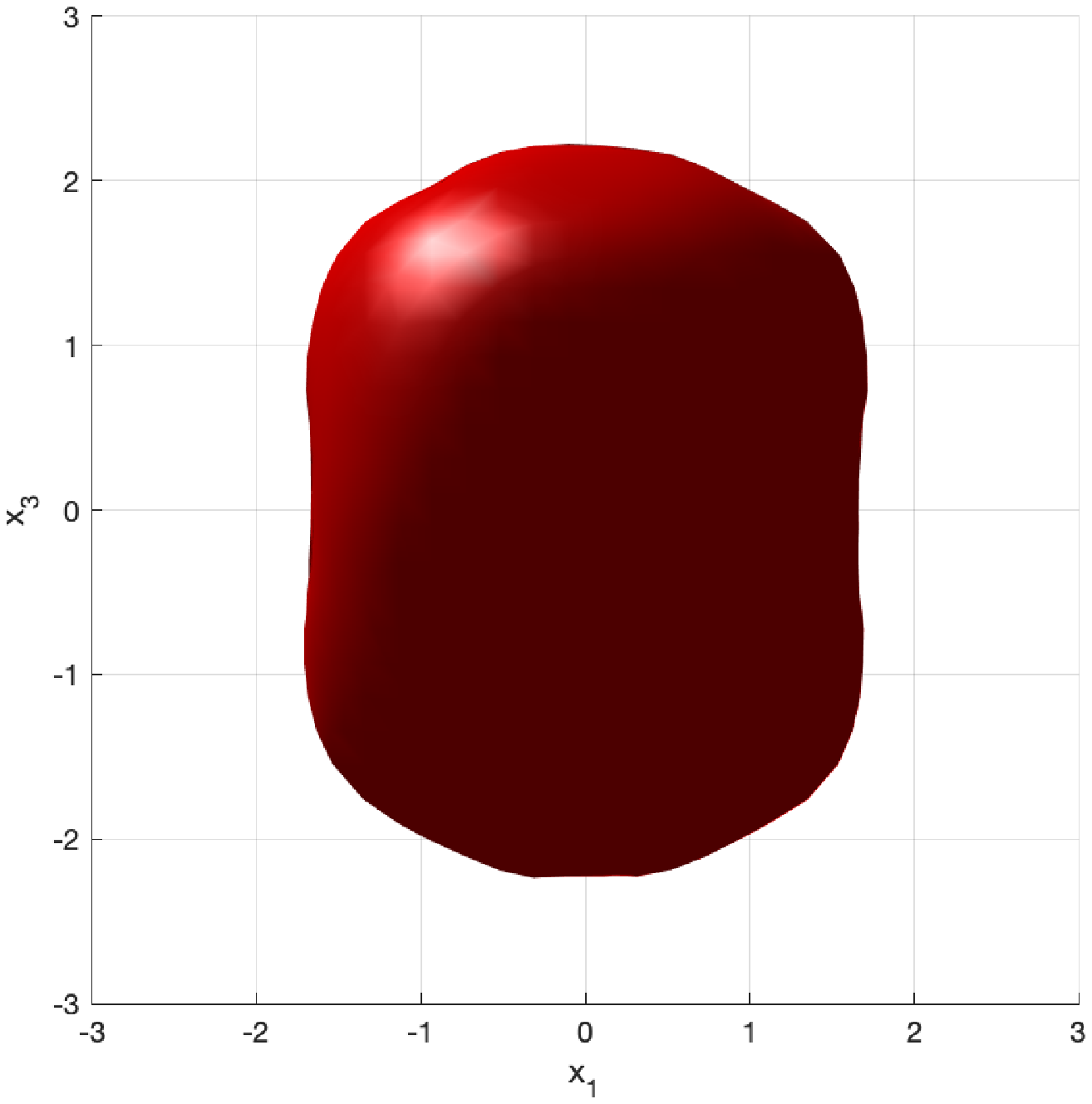}
\includegraphics[width=0.24\linewidth]{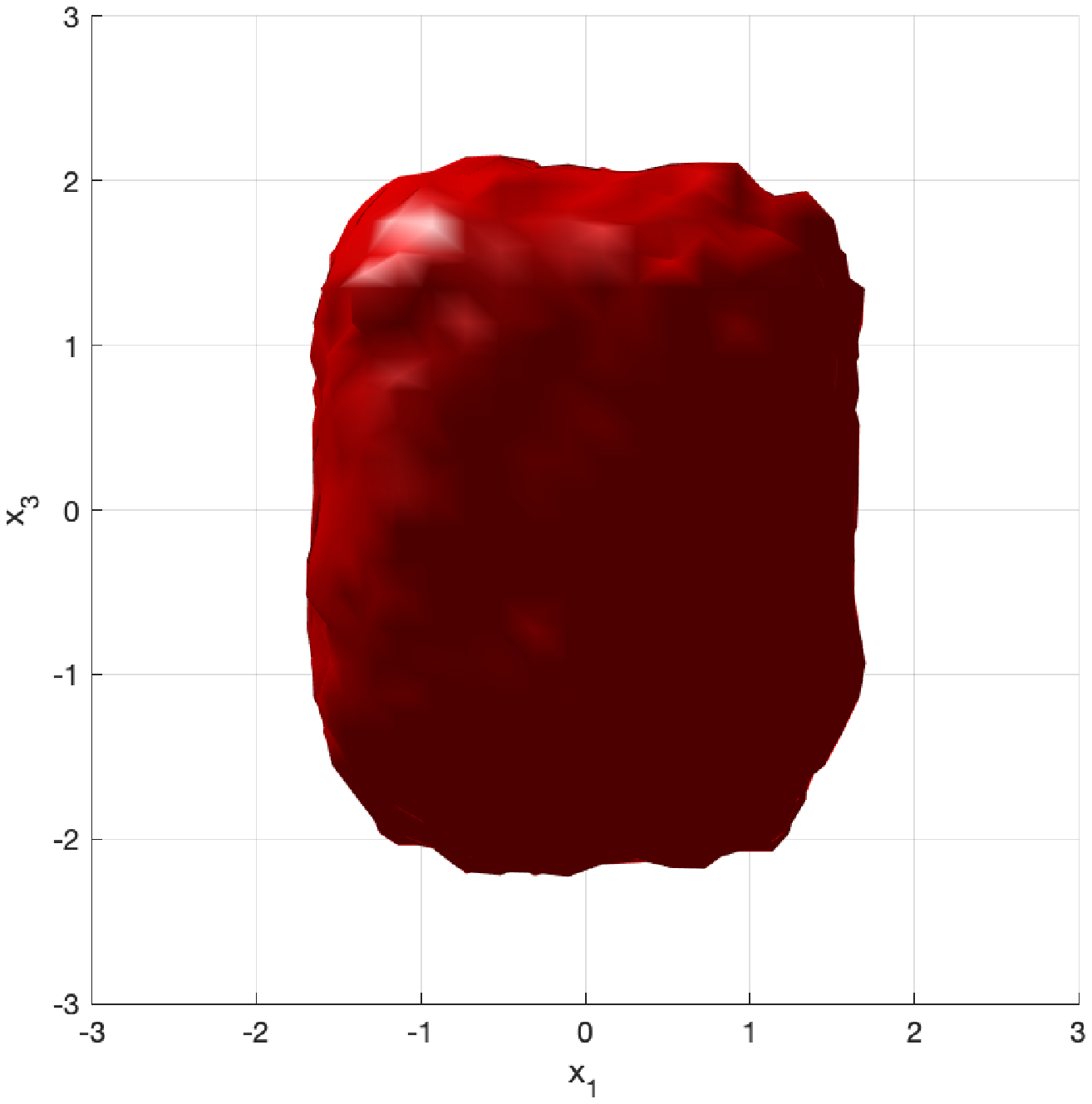}
     \caption{
Images of a cuboid cavity. First row: three dimensional view. Second row: $x_1 x_2-$cross section view. Third row: $x_2x_3-$cross section view. Last row: $x_1x_3-$cross section view. First column: exact geometry. Second column: image using a single polarization vector (iso-surface plot with iso-value -7). Third column: image using three polarization vectors (iso-surface plot with iso-value -7). Last column: image using three polarization vectors with noiseless data (iso-surface plot with iso-value -10).
     } \label{cuboid}
\end{figure}

\section{Summary} \label{Discussion}
We provide a complete study of the interior inverse electromagnetic scattering problem for a penetrable cavity. First, we prove that the homogeneous cavity is uniquely determined by the internal measurements. Our approach is based on the so-called {\it exterior transmission problem}, where such method allows us to consider a general background medium as we avoid constructing the Green's function. Second, we develop the linear sampling method to reconstruct the cavity and provide numerical examples to demonstrate the viability of the method. Finally, we mention several questions worth investigating in the future. Similar to the interior transmission problem, we expect that the exterior transmission eigenvalues form a discrete set and there exists infinitely many such eigenvalues. We expect that the exterior transmission eigenvalues provide information about the surrounding medium and such eigenvalues can be used in non-destructive material testing. Another future interest is to investigate the factorization method, the generalized linear sampling method, and other related robust imaging methods.

\section*{Acknowledgement} 
The work of FZ was supported by NSFC grant (11771068).

\section{Appendix}\label{Appendix}

{\bf Proof of Lemma \ref{ForwardMaxwellUnique}}:
\begin{proof}Assume $\vecE^s$ is such that (\ref{Scattered})-(\ref{SilverMuller}) holds with $\vecf=0$, it is sufficient to show that $\vecE^s=0$. Following exactly the same line of Theorem 2.1 in \cite{CaCo1}, we can directly obtain that  $\vecE^s=0$ in $\R^3 \backslash \overline{D}_1$. Since $A$ and $N$ have $C^1(\overline{D}_1 \backslash D)$ entries, we can extend $A$ and $N$ to $\tilde{A}$ and $\tilde{N}$ with $C^1(\R^3)$ entries. Since $\vecE^s=0$ in $\R^3 \backslash \overline{D}_1$ and $\vecE^s \in \vecH_{loc}(\mcurl, \R^3)$, then
\begin{eqnarray*}
(\nu \times \vecE^s)^- = 0 \quad &\mbox{on}& \quad \partial D_1, \\
(\nu \times A\mcurl \vecE^s)^- = 0 \quad &\mbox{on}& \quad \partial D_1,
\end{eqnarray*}
where $(\cdot)^-$ denotes the tangential traces of $\vecE^s|_{D_1}$ on $\partial D_1$. The above boundary conditions allow us to extend $\vecE^s|_{D_1 \backslash \overline{D}}$ by zero to $\tilde{\vecE}^s$ such that $\tilde{\vecE}^s \in \vecH(\mcurl,\R^3 \backslash \overline{D})$ and
$$
\mcurl (\tilde{A}\mcurl \tilde{\vecE}^s)- k^2 \tilde{N} \tilde{\vecE}^s = 0 \quad \mbox{in} \quad \R^3 \backslash \overline{D}.
$$
Since $\tilde{\vecE}^s=0$ in an open set in $\R^3 \backslash \overline{D}_1$, we have from unique continuation (Corollary 2.2 in \cite{O}) that $\vecE^s = \tilde{\vecE}^s =0$ in $D_1 \backslash \overline{D}$. A similar argument allows us to prove $\vecE^s$ vanish in $D$ and hence in $\R^3$. This completes the proof. 
\end{proof}

{\bf Proof of Lemma \ref{MaxwellTransmissionRegularity}}:
\begin{proof}
From Section \ref{Introduction}, we have that the forward problem  (\ref{Scattered})-(\ref{SilverMuller}) is equivalent to finding $\vecE \in \vecH_{loc}(\mcurl, \R^3 \backslash \overline{D})$ and $\vecE^s \in \vecH(\mcurl, D)$ such that
\begin{eqnarray}
\mcurl(A\mcurl \vecE) - k^2 N \vecE = 0 \quad &\mbox{in}& \quad \R^3 \backslash \overline{D}, \label{Appen 2.5 eqn 1}\\
\mcurl^2 \vecE^s - k^2  \vecE^s = 0 \quad &\mbox{in}& \quad D, \label{Appen 2.5 eqn 2}\\
\nu \times \vecE - \nu \times \vecE^s = \nu \times \vecE^i \quad &\mbox{on}& \quad \partial D, \label{Appen 2.5 eqn 3}\\
\nu \times  A\mcurl \vecE - \nu \times \mcurl \vecE^s = \nu \times \mcurl\vecE^i \quad &\mbox{on}& \quad \partial D, \label{Appen 2.5 eqn 4}\\
\lim\limits_{|\bx|\to \infty}\left(\mcurl \vecE^s\times \bx-ik|\bx|\vecE^s\right)=0, \label{Appen 2.5 eqn 5}
\end{eqnarray}
{where the first two equations are understood in the distributional sense, and the third and fourth equation hold in $\vecH^{-\frac{1}{2}}(\mdiv,\partial D)$.}

{Note that $\mdiv \vecE^s=0$ due to equation \eqref{Appen 2.5 eqn 2}, therefore we have that $\vecE^s \in \vecH(\mdiv, D)$ and $\mcurl^2\vecE^s \in \vecH(\mdiv, D)$. Therefore the normal traces $(\vecE^s \cdot \nu)|_{\partial D}$ and  $(\mcurl^2 \vecE^s \cdot \nu)|_{\partial D}$ belong to $H^{-1/2}(\partial D)$. Similarly $(N\vecE \cdot \nu)|_{\partial D}$ and  $(\mcurl (A\mcurl \vecE) \cdot \nu)|_{\partial D}$ belong to $H^{-1/2}(\partial D)$. Now from equations \eqref{Appen 2.5 eqn 1} -- \eqref{Appen 2.5 eqn 2}, we can derive that $(\mcurl^2 \vecE^s \cdot \nu)|_{\partial D} = k^2(\vecE^s \cdot \nu)|_{\partial D}$ and $(\mcurl (A\mcurl \vecE) \cdot \nu)|_{\partial D} = k^2 (N\vecE \cdot \nu)|_{\partial D}$ hold in $H^{-1/2}(\partial D)$.
Note that $(\nu \times  \mcurl \vecE^s)|_{\partial D}$ and $(\nu \times  A\mcurl \vecE)|_{\partial D}$ belong to $\vecH^{-\frac{1}{2}}(\mdiv,\partial D)$, {we can apply the surface divergence operator $\mdiv_{\partial D}$ to $(\nu \times  A\mcurl \vecE)|_{\partial D}$ and $(\nu \times  \mcurl \vecE^s)|_{\partial D}$ respectively, this allows us to derive that $\mdiv_{\partial D}  (\nu \times  \mcurl \vecE^s)|_{\partial D} = - (\mcurl^2 \vecE^s \cdot \nu)_{\partial D}$ and $\mdiv_{\partial D}  (\nu \times  A\mcurl \vecE)|_{\partial D} = - (\mcurl (A \mcurl \vecE) \cdot \nu)_{\partial D}$ hold in $H^{-1/2}(\partial D)$. Therefore  we can obtain} that $\mdiv_{\partial D}  (\nu \times  \mcurl \vecE^s)|_{\partial D} = - (\mcurl^2 \vecE^s \cdot \nu)|_{\partial D} = -k^2 (\vecE^s \cdot \nu)|_{\partial D}$ and $\mdiv_{\partial D}  (\nu \times  A\mcurl \vecE)|_{\partial D} = - (\mcurl (A \mcurl \vecE) \cdot \nu)|_{\partial D} = -k^2 (N \vecE \cdot \nu)|_{\partial D}$ in $H^{-1/2}(\partial D)$. Similarly $\mdiv_{\partial D}  (\nu \times  \mcurl \vecE^i)|_{\partial D} = - (\mcurl^2 \vecE^i \cdot \nu)|_{\partial D} = -k^2( \vecE^i \cdot \nu)|_{\partial D}$ in $H^{-1/2}(\partial D)$. Now we can derive from the above that $\mdiv_{\partial D}  (\nu \times  A\mcurl \vecE)|_{\partial D} -\mdiv_{\partial D}  (\nu \times  \mcurl \vecE^s)|_{\partial D}   = -k^2 \big((N\vecE \cdot \nu)|_{\partial D} - (\vecE^s \cdot \nu)|_{\partial D}\big)$.
Together with the  boundary condition \eqref{Appen 2.5 eqn 4} and $\mdiv_{\partial D}  (\nu \times  \mcurl \vecE^i)|_{\partial D}=-k^2( \vecE^i \cdot \nu)|_{\partial D}$, we can obtain that $(N\vecE \cdot \nu)|_{\partial D} - (\vecE^s \cdot \nu)|_{\partial D} = (\vecE^i \cdot \nu)|_{\partial D}$ in $H^{-1/2}(\partial D)$. Then if $\vecE^i \in \vecH_{loc}^1(\R^3 \backslash \overline{D})$, we have from the above identity and equation \eqref{Appen 2.5 eqn 3} that}
$$ \label{MaxwellTransmissionRegularityEqn1}
(\nu \times \vecE)|_{\partial D} - (\nu \times \vecE^s)|_{\partial D}=(\nu \times \vecE^i)|_{\partial D} \in \vecH_t^{\frac{1}{2}}(\partial D),
$$  
and
$$ (N\vecE \cdot \nu)|_{\partial D} - (\vecE^s \cdot \nu)|_{\partial D} = (\vecE^i \cdot \nu)|_{\partial D} \in H^{\frac{1}{2}}(\partial D).
$$

Let $\chi \in C_0^\infty (B_R)$ be a cutoff function that $\chi=1$ near $\partial D$ and has compact support in $D_1$. We remark that $A$ and $N$ are continuous in $D_1 \backslash \overline{D}$ so that $\mdiv(N \chi\vecE)= \chi \mdiv (N\vecE) + \nabla \chi \cdot (N\vecE)$ in $B_R  \backslash \overline{D}$. Taking the divergence of equation \eqref{Appen 2.5 eqn 1}, we can obtain $\mdiv (N\vecE) = 0$, then we have that $\mdiv(N \chi\vecE)= \nabla \chi \cdot (N\vecE)$. Now it is directly verified that
$$
\chi \vecE \in \vecH(\mcurl, B_R \backslash \overline{D}),~~ \vecE^s \in \vecH(\mcurl, D), ~~ \mdiv (N\chi \vecE) \in \vecL^2(B_R \backslash \overline{D}), ~~ \mdiv (\vecE^s) \in \vecL^2(D),
$$
$$
(\nu \times \chi\vecE)|_{\partial D} - (\nu \times \vecE^s)|_{\partial D}=(\nu \times \vecE^i)|_{\partial D} \in \vecH_t^{\frac{1}{2}}(\partial D),
$$
and
$$(N \chi\vecE \cdot \nu)|_{\partial D} - (\vecE^s \cdot \nu)|_{\partial D} = (\vecE^i \cdot \nu)|_{\partial D} \in H^{\frac{1}{2}}(\partial D).
$$
Now we apply Theorem 2.6 and Remark 2.7 in \cite{CaCo1} to get $\vecE^s \in \vecH^1(D)$, $\chi \vecE \in \vecH^1(B_R \backslash \overline{D})$, and (as in equation (2.41) in \cite{CaCo1})
\begin{eqnarray*}
&&\|\vecE^s\|_{\vecH^1(D)} +  \|\chi \vecE\|_{\vecH^1(D_1 \backslash \overline{D})} \\
&\le& C \left(  \|\vecE^s\|_{\vecH(\mcurl, D)}  +  \|\vecE\|_{\vecH(\mcurl, B_R \backslash \overline{D})}+\|\nu \times \vecE^i\|_{\vecH_t^{\frac{1}{2}}(\partial D)}+ \|\vecE^i \cdot \nu\|_{H^{\frac{1}{2}}(\partial D)} \right).
\end{eqnarray*}
Since $\vecE=\vecE^i + \vecE^s$, then
$$
\|\vecE^s\|_{\vecH^1(D)} +  \|\chi \vecE^s\|_{\vecH^1(D_1 \backslash \overline{D})} \le C \left(  \|\vecE^s\|_{\vecH(\mcurl, D)}  +  \|\vecE^s\|_{\vecH(\mcurl, B_R \backslash \overline{D})}+\|\vecE^i\|_{\vecH^1(B_R \backslash \overline{D})} \right).
$$
Similarly, with the help of the cutoff function we can show that  $(1-\chi)\vecE \in \vecH^1(D_1 \backslash \overline{D}))$, $\vecE^s \in \vecH^1(B_R \backslash \overline{D}_1)$, and
\begin{eqnarray*}
&&\|(1-\chi)\vecE^s\|_{\ \vecH^1(D_1 \backslash \overline{D}))} +  \| \vecE^s\|_{\vecH^1(B_R \backslash \overline{D}_1)} \\
&\le& C \left(  \|\vecE^s\|_{\vecH(\mcurl, D)}  +  \|\vecE^s\|_{\vecH(\mcurl, B_R \backslash \overline{D})}+\|\vecE^i\|_{\vecH^1(B_R \backslash \overline{D})} \right).
\end{eqnarray*}
Summing above two estimates yields the lemma. 
\end{proof}


\end{document}